\documentclass[10pt,letterpape,reqno]{amsart}
\usepackage[margin=1.2in]{geometry}
\usepackage{amsmath, amssymb}
\usepackage{xcolor}
\usepackage{bm, bbm}
\usepackage{mathabx} 
\usepackage{mathrsfs}
\usepackage{microtype}
\usepackage{bbold}
\usepackage[mathscr]{euscript}

\let\euscr\mathscr
\newcommand{\esc}{\euscr}
\usepackage{upgreek}

\usepackage{amsfonts}
\usepackage{yfonts}
\usepackage{mathtools}

\usepackage{tikz-cd}
\usepackage[all,cmtip]{xy}

\numberwithin{equation}{section}
\usepackage{soul}
\usepackage{comment}

\usepackage{hyperref}
\hypersetup{colorlinks}
\definecolor{darkred}{rgb}{0.5,0,0}
\definecolor{darkgreen}{rgb}{0,0.5,0}
\definecolor{darkblue}{rgb}{0,0,0.5}
\hypersetup{colorlinks, linkcolor=darkblue, filecolor=darkgreen, urlcolor=darkred, citecolor=darkblue}
\makeatletter 
\@addtoreset{equation}{section}
\makeatother  

\numberwithin{equation}{section}

\newtheorem{thm}{Theorem}[section]
\newtheorem{cor}[thm]{Corollary}

\newtheorem{prop}[thm]{Proposition}
\newtheorem{lemma}[thm]{Lemma}
\newtheorem{step}{Step}

\theoremstyle{definition}
\newtheorem{defn}[thm]{Definition}
\theoremstyle{remark}
\newtheorem{rem}[thm]{Remark}

\newtheorem{notation}[thm]{Notation}

\newcommand{\beq}{\begin{equation}}
\newcommand{\eeq}{\end{equation}}
\newcommand{\beqn}{\begin{equation*}}
\newcommand{\eeqn}{\end{equation*}}
\newcommand{\ov}{\overline}
\newcommand{\mb}{\mathbb}
\newcommand{\mt}{\mathtt}

\newcommand{\mc}{\mathcal}
\newcommand{\mf}{\mathfrak}
\newcommand{\ms}{\mathscr}

\newcommand{\pst}{\uds{\bf POSET}^\dagger}

\newcommand{\wt}{\widetilde}
\newcommand{\wh}{\widehat}

\newcommand{\uds}[1]{\underline{\smash{#1}}}

\renewcommand{\subset}{\subseteq}

\newcommand{\ev}{{\rm ev}}

\makeatletter
\newcommand{\colim@}[2]{%
  \vtop{\m@th\ialign{##\cr
    \hfil$#1\operator@font colim$\hfil\cr
    \noalign{\nointerlineskip\kern1.5\ex@}#2\cr
    \noalign{\nointerlineskip\kern-\ex@}\cr}}%
}
\newcommand{\colim}{%
  \mathop{\mathpalette\colim@{\rightarrowfill@\textstyle}}\nmlimits@
}
\makeatother

\newcommand{\gx}[1]{{\color{magenta}{#1}}}

\title[Reduced GW without ghost bubble censorship]{Reduced Gromov--Witten invariants without ghost bubble censorship}

\author{Guangbo Xu}
\address{Department of Mathematics, Rutgers University, Hill Center 
 -  Busch Campus,
110 Frelinghuysen Road, Piscataway, NJ 08854-8019, USA}
\email{gx49@math.rutgers.edu}

\date{\today}

\begin{document}

\begin{abstract}
We give a definition of all-genus reduced Gromov--Witten invariants of symplectic manifolds by using effectively supported multivalued perturbations on derived orbifold/Kuranishi charts, which bypasses the hard analytical result of sharp compactification/ghost bubble censorship of Zinger \cite{Zinger_sharp, Zinger_reduced}, Doan--Walpuski \cite{Doan_Walpuski_embedded}, and Ekholm--Shende \cite{Ekholm_Shende_ghost, Ekholm_Shende_bare}. 
\end{abstract}

\maketitle

\setcounter{tocdepth}{1}
\tableofcontents

\section{Introduction}


The definition of Gromov--Witten invariants only requires a modest amount of structures on moduli spaces of stable maps. Many important geometric problems ask for more refined invariants, such as {\it relative} Gromov--Witten invariants \cite{Ionel_Parker_relative} or {\it reduced} Gromov--Witten invariants \cite{Zinger_reduced}. The traditional approaches towards these refinements are closely tied with the specific geometry and rely crucially on the corresponding variants of Gromov compactness. 

We introduce a new topological approach where the refined invariants can be extracted from two extra data on the moduli spaces: a natural stratification and a (stable) complex structure. In particular, one works entirely with the stable map compactification. We demonstrate this approach by defining all-genus reduced Gromov--Witten invariants for a general compact symplectic manifold, a long-standing challenge in symplectic topology beyond the genus one case by Zinger.

\subsection{Main result}

The moduli spaces of stable maps contain configurations which are not limits of smooth curves. One typical case is  a smooth curve attached with a higher genus ($g\geq 1$) ghost component. The contributions from these degenerate curves distort the curve counting meaning of Gromov--Witten invariants. 
The genus $g$ {\it reduced Gromov--Witten invariants} can be viewed as the count of maps with smooth domains while contributions from those ``fake configurations'' are removed, which has only been defined rigorously in genus one by Zinger \cite{Zinger_reduced}. 
The main result of this paper is the definition of reduced Gromov--Witten invariants in all genera.

\begin{thm} \label{thm11} (Theorem \ref{thm52}) 
Let $(X, \omega)$ be a compact symplectic manifold and $A \in H_2(X; {\mb Z})$. Then for all $g, n \geq 0$, there are  well-defined homology classes 
\beqn
[\ov{\mc M}{}_{g,n}(X, J, A)]^{\rm red} \in H_*( \ov{\mc M}{}_{g,n} \times X^n; {\mb Q})
\eeqn
which only depends on the symplectic deformation class of $\omega$. Moreover, when $g = 0$, it coincides with the Gromov--Witten virtual fundamental class.
\end{thm}

\begin{rem}
In fact, similar to the main result of \cite{Bai_Xu_integer}, one can define a collection of refinements of Gromov--Witten invariants labelled by combinatorial types of curves. The reduced invariant is the one from the main stratum of smooth curves. 
\end{rem}

An existing method of constructing reduced Gromov--Witten invariants, which goes back to Zinger \cite{Zinger_sharp}, is to geometrically identify the limits of smooth curves (the {\it sharp compactification}) and then construct the virtual fundamental class on the sharp compactification. 
In genus $1$ case, Zinger provided a simple criterion for stable maps to be limits of smooth curves 
and defined the reduced invariants in genus $1$ (see \cite{Zinger_reduced}). Following the same ideas, there are the works of Wang \cite{Wei_Wang_reduced} and Niu \cite{Niu_thesis} aiming at constructing (symplectic) reduced genus-two Gromov--Witten invariants. Recently Ekholm--Shende gave a new approach for higher genus reduced Gromov--Witten invariants for Calabi--Yau threefolds (\cite{Ekholm_Shende_bare} (for both closed and open curves). Their approach, based on polyfolds, depends crucially on their {\it ghost bubble censorship} \cite{Ekholm_Shende_ghost} (for general symplectic manifolds), a higher genus extension of Zinger's sharp compactification. 

There are parallel algebraic approaches towards reduced Gromov--Witten invariants. 
 The basic idea is to resolve singularities in the moduli spaces of curves to
separate the ``main'' component from other components, and construct virtual fundamental class on this component. See \cite{Vakil_Zinger_2008}\cite{Hu_Li_2010}\cite{Hu_Li_Niu}\cite{RSW_2019}\cite{Hu_Niu_2019, Hu_Niu_2020}\cite{Battistella_Carocci_2023}\cite{Rabano_Mann_Manolache_Picciotto} for works in this spirit, including algebraic definitions of reduced Gromov--Witten invariants in various cases. 


We will prove in a forthcoming paper that our definition agrees with Zinger's definition in genus one. One important ingredient will be to choose perturbations with certain ``forgetful map property.''

\begin{thm}\cite{Xu_reduced_comparison}
The genus 1 reduced Gromov--Witten invariants defined in this paper agree with those defined by Zinger \cite{Zinger_reduced}.
\end{thm}

\begin{rem}
There is a conjectural relation between standard and reduced Gromov--Witten invariants in all genera. The genus one comparison was completely proved by Zinger \cite{Zinger_comparison}. In higher genus, for Calabi--Yau threefold, there is a general conjecture describing the relation (cf. \cite[Theorem 1.4]{Ekholm_Shende_bare}).
Potentially, one can reduce the verification of such a formula to model cases calculated by Pandharipande \cite{Pandharipande_1999} via the cluster argument of Ionel--Parker \cite{IP-GV}. A similar argument was used in the proof of \cite[Theorem 9.3]{Ekholm_Shende_bare}.
\end{rem}

\begin{rem}
An important property of Zinger's genus one reduced Gromov--Witten invariants is that they satisfy the hyperplane property (see \cite{Li_Zinger_2009}) while the ordinary ones do not. Using this property, Zinger proved the genus one mirror symmetry for the quintic \cite{Zinger_JAMS}. One should be able to follow the same symplectic approach of \cite{Li_Zinger_2009} to prove the hyperplane property of the reduced invariants defined in this paper. 
\end{rem}

\begin{rem}
In the case of open curves in Calabi--Yau threefolds with Maslov zero Lagrangian boundary condition, Ekholm--Shende \cite{Ekholm_Shende_ghost, Ekholm_Shende_bare} were still able to define their bare curve counting. It is not yet clear to the author whether the method developed in the current paper can produce the same result or not for open curve counting. 
\end{rem}

This work originates from a discussion with Shaoyun Bai, who asked whether Ekholm--Shende's new approach can be reproduced using global Kuranishi charts invented by Abouzaid--McLean--Smith \cite{AMS} (with higher genus extensions by Hirschi--Swaminathan \cite{Hirschi_Swaminathan_2024} and themselves \cite{AMS2}). 
The author then realized the specific perturbations Zinger and Ekholm--Shende used need to respect a natural stratification on the stable map moduli spaces; meanwhile, the Kuranishi adaptation of their perturbations is very similar to the FOP (Fukaya--Ono \cite{Fukaya_Ono_integer} and Parker \cite{BParker_integer}) perturbations in the construction of integer-valued Gromov--Witten invariants by Bai and the author \cite{Bai_Xu_integer}.

A key ingredient in both \cite{Bai_Xu_integer} and the current paper is that one needs the ``normal complex structure'' on the moduli spaces to manage the failure of transversality: they allow us to use specific perturbations (in Kuranishi charts) to define virtual fundamental cycles of the ``main stratum'' instead of the whole moduli space. In some sense, the hard analytical results on sharp compactification/ghost bubble censorship are bypassed by the soft differential topological argument.



\begin{rem}
A slight complication towards finishing the construction is that the author does not know how to construct a global Kuranishi chart with the stratification we need. More precisely, one expects to decompose the obstruction spaces into summands which are supported on individual components of the curves. However, as a curve may have automorphisms which interchange components, such a decomposition may only be possible locally. To achieve our goal, instead of a global chart, we choose to construct a Kuranishi atlas with local charts and coordinate changes preserving such stratifications. We still use some ingredient of the global Kuranishi chart approach to simplify the construction of the atlas. On the other hand, it is possible to use the original approach of Kuranishi structures of Fukaya {\it et al.} \cite{Fukaya_Ono, FOOO_Book, FOOO_Kuranishi} to carry out the construction. 
\end{rem}

\begin{rem}
The idea of this paper can also be applied to situations where the stratification comes from other geometric setup. We briefly discuss such a possibility in the relative/degeneration situation. It was predicted in \cite[Section 10]{Tian_CDM} that via degeneration Gromov--Witten invariants can be decomposed into relative invariants of components meeting along normal crossing divisors. However, the usual stable map compactification was regarded not suitable for this problem as curves contained in the divisor could be obstructed and their contributions are difficult to identify. Corresponding relative Gromov--Witten invariants \cite{Ionel_Parker_relative} or log Gromov--Witten invariants \cite{Gross_Siebert_log}\cite{Chen_log_1}\cite{Abramovich_Chen_log_2} were then defined and corresponding degeneration/sum formula were proved (see \cite{Ionel_Parker_sum}\cite{Li_Ruan}\cite{ACGS_glue}). In those works, one needs to consider different compactifications of curves only meeting the divisors are isolated points (see also \cite{Tehrani_log}\cite{Tehrani_Swaminathan} and \cite{Daemi_Fukaya_complement_1, Daemi_Fukaya_complement_2, Daemi_Fukaya_complement_3}) and carry out the corresponding virtual construction. The technique developed in this paper suggests that one only needs to consider the ordinary stable map compactification whose main stratum consists of curves without components contained in the divisor; the relative or log compactifications can be viewed as certain geometric ``blowups'' of the boundary of the main stratum. This  would provide a conceptually simpler approach to the degeneration formula in the symplectic setting which would be also useful in Floer theory.
\end{rem}

\subsection{Outline}

In Section \ref{section2} we set up notions related to stratification, normal complex structure, and the type of perturbations we will consider in Kuranishi models. In Section \ref{section3}, we define a new notion of transversality designed for the problem; it is parallel to the corresponding part of \cite{Bai_Xu_integer}. In Section \ref{section4} we set up a version of Kuranishi atlas theory and provide the abstract construction of the virtual fundamental classes of the main stratum. In Section \ref{section5}, we equip a stable map moduli space with a Kuranishi atlas, define the reduced Gromov--Witten invariants, and prove Theorem \ref{thm11}.

\subsection{Acknowledgements}

I would like to thank Shaoyun Bai for bringing his attention to \cite{Ekholm_Shende_bare} and for encouragements. I would also like to thank Vivek Shende, Young-Hoon Kiem, Mark McLean for helpful discussions and thank Aleksey Zinger for encouragements. The work is partially supported by NSF DMS-2345030, DMS-2506403, and Simons Foundation Travel Grant for Mathematicians.

\section{Topological Setup}\label{section2}

The purpose of this section is to introduce a class of multisections of an orbifold vector bundle which has a certain stratification-like structure.

\subsection{Topological stratifications}

We first recall the basic notion of stratifications of topological spaces. Later we will need to consider Whitney stratifications in smooth manifolds. A {\bf partition} of a topological space $X$ is a set ${\mf X} = \{ X_\alpha^*\ |\ \alpha \in {\mc A} \}$ of {\it nonempty} subsets of $X$ such that
\beqn
X = \bigsqcup_\alpha X_\alpha^*.
\eeqn
Here $\bigsqcup$ means disjoint union. A partition is called a {\bf stratification} if 
\begin{enumerate}

\item The partition is locally finite.

\item Each $X_\alpha^*$ (called a stratum) is a locally closed subset.

\item (Axiom of frontier) For any $\alpha, \beta \in {\mc A}$, $X_\alpha^* \cap \ov{X_\beta^*}\neq \emptyset \Longrightarrow X_\alpha^* \subset \ov{X_\beta^*}$. 
\end{enumerate}
The axiom of frontier implies that the indexing set has a  partial order
\beqn
\alpha \leq \beta \Longleftrightarrow X_\alpha^* \subset \ov{X_\beta^*}.
\eeqn

\begin{notation}
We follow the following notational convention. We denote
\beqn
X_\alpha:= \ov{X_\alpha^*}
\eeqn
which is a closed subset of $X$. Then $\alpha \leq \beta$ if and only if $X_\alpha \subset X_\beta$.
\end{notation}

\noindent {\bf Convention.} We always assume that strata   are connected. 

We can pull back partitions by continuous maps. Let $f: Y \to X$ be a continuous map and let ${\mf X}$ be a partition of $X$. Then the {\bf pullback} of ${\mf X}$ by $f$ is the partition
\beq\label{naive_pullback}
f^* {\mf X} = \bigsqcup_{X_\alpha^*\in {\mf X}} \Big\{ {\rm connected\ components\ of\ }f^{-1}(X_\alpha^*) \Big\}.
\eeq
When $O \subset X$ is a subset, let ${\mf X}|_O$ be the pullback via the inclusion map $O \hookrightarrow X$, called the restriction of ${\mf X}$ to $O$. ${\mf X}$ being a stratification does not necessarily implies that $f^* {\mf X}$ is a stratification. However, when it is a stratification, we call it the pullback stratifications.

A self-homeomorphism $f: X \to X$ is said to {\bf preserve} a partition or a stratification ${\mf X}$ if $f^* {\mf X} = {\mf X}$ as sets (where $f$ could send a stratum to a different one). Similarly, an action of a group $G$ on $X$ is said to {\bf preserve} a partition or a stratification ${\mf X}$ if each group element $g \in G$ preserves ${\mf X}$.

\subsubsection{Structural cosheaf} 
One can use cosheaves to describe stratifications. 

\begin{defn}
Let $\pst$ denote the category of  partially ordered sets whose morphisms are monotonic maps which send maximal elements to maximal elements. 
\end{defn}

Recall that a $\pst$-valued cosheaf on a topological space $X$ 
is a covariant functor
\beqn
{\mc F}: \underline{\bf OPEN}(X) \to \pst
\eeqn
such that for any open subset $O \subset X$,
\beqn
\colim \left( \prod_{i, j} {\mc F}(O_i \cap O_j) \rightrightarrows \prod_i {\mc F}(O_i) \right) \to {\mc F}(O)
\eeqn
is an isomorphism of posets. 

\begin{defn}\label{defn_structural_cosheaf}
Let $X$ have a stratification ${\mf X}$. Its {\bf structural cosheaf}  is the $\pst$-valued cosheaf ${\mc O}^X$ with value on an open set $U \subset X$ being
\beqn
{\mc O}^X(U):= {\mf X}|_U.
\eeqn
\end{defn}

The gluing axiom for cosheaf can be readily checked. The following lemma summarizes some basic properties of the structural cosheaf.

\begin{lemma}\label{lemma_property_cosheaf}
Let $X$ be a stratified space. Then 
\begin{enumerate}
    \item The structural cosheaf ${\mc O}^X$ is {\bf constructible}, i.e., for each stratum $X_\alpha^* \subset X$, the cosheaf ${\mc O}^X|_{X_\alpha^*}$ is locally constant. 

    \item Each $x \in X$ has an open neighborhood $U_x$ such that the natural map
    \beqn
    {\mc O}^X_x \to {\mc O}^X(U_x)
    \eeqn
    is an isomorphism; for each $y \in U_x$, the natural map 
    \beqn
    {\mc O}_y^X \to {\mc O}^X(U_x) \cong {\mc O}_x^X
    \eeqn
    is injective and for each $z \in U_x$ sufficiently close to $y$, the following diagram commutes.
    \beqn
    \xymatrix{ O_z^X \ar[d] \ar[rd]  & \\
    O_y^X \ar[r] & O_x^X }
\eeqn
    \end{enumerate}

    \end{lemma}

The information of the stratification ${\mf X}$ can be recovered from the cosheaf ${\mc O}^X$. We will use the following concept to describe stratified vector bundles.

\begin{defn}\label{defn_stratification_like}
A {\bf stratification-like} cosheaf over a stratified space  is a $\pst$-valued cosheaf satisfying the two properties listed in Lemma \ref{lemma_property_cosheaf}.
\end{defn}

\subsection{Linear stratifications}

We will consider stratified manifolds and orbifolds which locally are vector spaces with stratification given by subspaces. This subsection provides preliminary discussions about such local models.

\begin{defn}
A {\bf linear stratification} on a finite-dimensional (real or complex) vector space is a finite stratification encoded in a cosheaf ${\mc O}^V$ such that the closure of each stratum is a linear subspace of $V$. A vector space $V$ with a linear stratification is called a {\bf stratified linear space}.
\end{defn}

The definition implies that there is a unique maximal stratum, denoted by $V_{\rm main}^*$, a unique minimal stratum, denoted by $V_0$; there is also no real codimension one strata. For any stratum $V_\alpha^* \subset V$, there is an associated stratified linear space $V_{\geq \alpha}$ obtained from $V$ by only remembering strata living above $\alpha$. In particular, $(V_{\geq \alpha})_0 = V_\alpha$. 


\begin{rem}
On the linear level, there are several differences between the current setting and the setting in \cite{Bai_Xu_integer}. In general, linear stratifications of the same combinatorial type (described by the dimensions of subspaces obtained by taking intersections and sums) may have a moduli. For example, consider the stratification of ${\mb R}^2$ formed by many 1-dimensional subspaces. The linear models considered in \cite{Bai_Xu_integer}, i.e., representations of finite groups, do not have moduli. On the other hand, a representation canonically split into the direct sum of the trivial part and the nontrivial part. In the current setting, such a canonical splitting does not exist.
\end{rem}

\begin{rem}
One could try to define the notion of stratified vector bundles over stratified spaces. However, we choose to do so only on stratified manifolds as the local structure will be simpler. 
\end{rem}

\subsubsection{Normal complex structures}


\begin{defn}\label{defn_NCS_space}
Let $V$ be a stratified linear space.
\begin{enumerate}

\item A {\bf normal complex structure} (NC structure for short) on $V$, is a complex structure on the quotient space $V/ V_0$ such that for all strata $\alpha$, the subspace
\beqn
V_\alpha/ V_0 \subset V / V_0
\eeqn
is a complex subspace.

\item A {\bf normally complex stratified linear space}, or an {\bf NCS linear space}, denoted by $\mb{V}$, is a stratified linear space $V$ equipped with an NC structure. Notice that a complex stratified linear space is also an NCS linear space.



\end{enumerate}
\end{defn}


For each global stratum $\alpha \in {\mc O}^V(V)$, an NC structure on $V$ canonically induces an NC structure on $V_{\geq \alpha}$ via the exact sequence
\beqn
\xymatrix{ 0 \ar[r] &   V_\alpha/V_0 \ar[r] & V/ V_0 \ar[r] &    V / V_\alpha \ar[r] & 0}.
\eeqn
Denote this NCS linear space by ${\mb V}_{\geq \alpha}$.

\subsubsection{Splittings}

\begin{defn}\label{defn_linear_splitting}

\begin{enumerate}

\item Let $V$ be a stratified linear space.  A {\bf splitting} of $V$ is a collection of complements $V = V_\alpha \oplus V_\alpha^\vee$ such that $\alpha \leq \beta \Longrightarrow  V_\beta^\vee \subset V_\alpha^\vee$.

\item Let ${\mb V}$ be an NCS linear space. A {\bf normally complex splitting} (NC splitting) of ${\mb V}$, is a splitting of $V$ (so all $V_\alpha^\vee$ are complex) such that the inclusion $V_\beta^\vee \hookrightarrow V_\alpha^\vee$ for all $\alpha \leq \beta$ is complex linear.

\end{enumerate}
\end{defn}

\subsection{Complex stratified maps between vector spaces}

We consider certain stratified maps between stratified linear spaces. We first introduce a convenient notion.

\begin{defn}\label{defn_linear_pair}
A {\bf complex stratified virtual space} consists of a pair of complex stratified linear space ${\mb V}$ and ${\mb W}$ together with a morphism ${\mc O}^V(V) \to {\mc O}^W (W)$ of $\pst$. A complex stratified virtual space is denoted as $({\mb V}, {\mb W})$ where the morphism ${\mc O}^V(V) \to {\mc O}^W (W)$ is hidden.

\end{defn}


\subsubsection{Polynomial maps}

First consider a complex stratified virtual space $({\mb V}, {\mb W})$. Let ${\rm Poly}^d(V, W)$ be the space of complex polynomial maps from $V$ to $W$ whose degrees are at most $d$. Let $\alpha\mapsto \beta(\alpha)$ be the poset map ${\mc O}^V (V) \to {\mc O}^W (W)$. Define
\beqn
{\rm Poly}^d ({\mb V}, {\mb W}):=  \Big\{ P\in {\rm Poly}^d (V, W)\ |\ P(V_\alpha) \subset W_{\beta(\alpha)} \Big\}.
\eeqn
Then ${\rm Poly}^d ( {\mb V}, {\mb W}) \subset {\rm Poly}^d (V, W)$ is a subspace. 

There are some basic properties about the space of stratified polynomial maps.

\begin{lemma}\label{finite_generation}
${\rm Poly}( \mb{V}, \mb{W} )$ is a finitely generated ${\mb C}[V]$-module.
\end{lemma}

\begin{proof}
${\mb C}[V]$ is Noetherian and ${\rm Poly}(V, W)$ is obviously finitely generated over ${\mb C}[V]$. Hence the submodule ${\rm Poly}( \mb{V}, \mb{W} )$ is finitely generated. 
\end{proof}

We could consider more generally holomorphic maps satisfying the same stratum-preserving property. Let $U \subset V$ be an open subset, equipped with the induced stratification, i.e., $U_\alpha = U \cap V_\alpha$. Let 
\beqn
{\rm Hol}_{\rm stratified} (U, \mb{W})
\eeqn
be the space of holomorphic maps $f: U \to W$ such that $f(U_\alpha) \subset W_{\beta(\alpha)}$. It is a module over the ring of holomorphic functions ${\mc O}(U)$. 

\begin{lemma}\label{lemma_holomorphic_generation}
${\rm Hol}_{\rm stratified}(U, \mb{W})$ is generated over ${\mc O}(U)$ by ${\rm Poly}(\mb{V}, \mb{W})$.
\end{lemma}

\begin{proof}
Inside the ring ${\mb C}[V]$, let ${\mc I}_{V_\alpha}$ be the ideal generated by linear functionals vanishing on $V_\alpha$, and let $L_{\alpha, 1}, \ldots, L_{\alpha, m_\alpha}: W \to {\mb C}$ be defining linear functionals of $W_{\beta(\alpha)}$. Let ${\mc I}_{U_\alpha} \subset {\mc O}_U$ be the ideal corresponding to $U_\alpha$. Then 
\beqn
{\rm Poly}{}( {\mb V}, {\mb W} ) = \Big\{ f \in {\rm Poly}(V, W)\ |\ L_{\alpha, j} \circ f \in {\mc I}_{V_\alpha},\ \forall j=1, \ldots, m_\alpha\ \forall \alpha \in {\mc O}^V(V)   \Big\}.
\eeqn
In other words,
\beqn
{\rm Poly}(\mb{V}, \mb{W}) = {\rm Ker} \left( \xymatrix{ {\rm Poly}(V, W) \ar[r]    &   \displaystyle \bigoplus_{\alpha \in {\mc O}^V (V)} \bigoplus_{j=1}^{m_\alpha} {\mb C} [V]/ {\mc I}_{V_\alpha}      } \right).
\eeqn
Then tensoring over ${\mb C}[V]$ with ${\mc O}_U$, one obtains
\beqn
{\rm Poly}(\mb{V}, \mb{W}) \underset{{\mb C}[V]}{\otimes} {\mc O}_U = {\rm Ker} \left(  \xymatrix{ {\rm Poly}(V, W) \underset{\mb{C}[V]}{\otimes} {\mc O}_U \ar[r]    &   \displaystyle \bigoplus_{\alpha \in {\mc O}^V(V) } \bigoplus_{j=1}^{m_\alpha} {\mc O}_U / {\mc I}_{U_\alpha}      } \right)
\eeqn
where the arrow is the naturally induced one. However, the right hand side is exactly ${\rm Hol}_{\rm stratified}(U, \mb{W})$.
\end{proof}

\begin{rem}
The space ${\rm Poly}( {\mb V}, {\mb W} )$ is a generalization of the space of equivariant polynomial maps considered in \cite{Bai_Xu_integer}. If $V, W$ are complex representations of a finite group $G$, then $V$ and $W$ carries the stratification by isotropy subgroups $H \subset W$. Let $V_H\subset V$ and $W_H \subset W$ be the invariant subspaces. A $G$-equivariant map $f: V \to W$ necessarily satisfies $f(V_H) \subset W_H$.
\end{rem}

Following \cite{BParker_integer} and \cite{Bai_Xu_integer}, we consider a more convenient space containing all stratified polynomial maps.

\begin{defn}
Let $(\mb{V}, \mb{W})$ be a complex stratified virtual space. A {\bf complex stratified map}  from ${\mb V}$ to ${\mb W}$ is an element of 
\beqn
C^\infty_{\rm CS} (\mb{V}, \mb{W}) := C^\infty(V; {\mb C}) \Big( {\rm Poly}(\mb{V}, \mb{W}) \Big).
\eeqn
If $O \subset V$ is open, denote by $C^\infty_{\rm CS}(O, {\mb W})$ be the space of restrictions of complex stratified maps to $O$.
\end{defn}

\begin{rem}
Typically, the inclusion $C^\infty_{\rm CS}({\mb V}, {\mb W}) \subset C_{\rm stratified}^\infty( V, W )$ is not $C^\infty$-dense. For example, if ${\mb V} \cong {\mb W} \cong {\mb C}$ are stratified with top stratum the total space and a lower stratum being the origin, then ${\rm Poly}{}(\mb{V}, \mb{W})$ is generated by a linear map. However, there are smooth maps $f:V \to W$ satisfying $f(0) = 0$ which are not multiples of a linear map, such as $f(z) = \ov{z}$. However, one has the $C^0$-density as shown in the next lemma.
\end{rem}

\begin{lemma}\label{lemma_c0_density}
$C^\infty_{\rm CS}(\mb{V}, \mb{W})$ is $C^0$-dense in $C_{\rm stratified}^\infty ( V, W )$.
\end{lemma}

\begin{proof}
It is a consequence of Lemma \ref{lemma36}, which shows one can achieve any prescribed value at a given point by a CS map. The $C^0$-approximation can be obtained by a partition-of-unity argument.
\end{proof}

\subsection{Stratified manifolds, bundles, and virtual manifolds}

We move up to the nonlinear situation for stratifications on manifolds and bundles. 

\begin{defn}\label{defn_manifold_stratification}
Let $U$ be a smooth manifold with a stratification.
\begin{enumerate}

\item A {\bf stratified chart} of $U$ consists of a stratified linear space $V$, an open neighborhood $O \subset V$ of the origin, and a stratified open embedding
\beqn
\phi: O  \to U.
\eeqn

\item $U$ is called a {\bf stratified manifold} if there exists a covering by stratified charts. In this case, the tangent space at each $x \in U$ is canonically a stratified linear space.


\item Let $G$ be a Lie group. A {\bf stratified $G$-manifold} is a stratified manifold $U$ with a smooth stratum-preserving $G$-action such that for each $x \in U$, there exists a stratified $G_x$-equivariant chart, i.e., a stratified space $V$ with a $G_x$-action and a stratified chart $\phi: O \to U$ around $x$ with a $G_x$-invariant open ball $O \subset V$.
\end{enumerate}
\end{defn}

A basic feature of stratified manifolds is that each stratum extends to immersed submanifolds with clean self-intersections.

\begin{lemma}\label{lemma_strata_extension}
Let $U$ be a stratified manifold. For each stratum $U_\alpha^*$, its closure is the image of a smooth immersion 
\beqn
\iota_\alpha: \tilde U_\alpha \to U
\eeqn
with clean self-intersections.
\end{lemma}

\begin{proof}
Choose $x \in U_\alpha \setminus U_\alpha^*$. Then $U_\alpha^*$ intersects a small neighborhood of $x$ in possibly several connected components. The existence of local stratified chart implies that each local  component extends to a smooth submanifold and each pair of them intersect cleanly. 
\end{proof}

\subsubsection{Stratified vector bundles}

\begin{defn}\label{defn_stratified_vector_bundle}
Let $U$ be a stratified manifold. A {\bf stratified vector bundle} (real or complex) over $U$  consists of the following objects.

\begin{enumerate}

\item a (real or complex) vector bundle $E \to U$,

\item a stratification-like cosheaf ${\mc O}^E$ over $U$ (see Definition \ref{defn_stratification_like}), and 

\item a linear stratification on each fibre $E_x$ by the costalk ${\mc O}_x^E$.

\end{enumerate}
These structures need to satisfy the local triviality condition, namely, for each $x \in U$, there exists a local trivialization of $E$ over a neighborhood $U_x$ 
\beqn
\phi_x: E|_{U_x} \cong U_x \times E_x
\eeqn
such that for each $y \in U_x$ sufficiently close to $x$, the restriction of $\phi_x$ to $E_y$ is stratified, i.e. the following diagram commutes.
\beqn
\xymatrix{  E_y \ar[r] \ar[d]  & E_x \ar[d]\\
{\mc O}^E_y \ar[r]  & {\mc O}^E_x }
\eeqn
Such a bundle chart is called a {\bf stratified bundle chart}.
\end{defn}

The tangent bundle $TU \to U$ is an example. 
Notice that the condition that ${\mc O}^E$ is constructible implies that the fibrewise stratification of $E$ when restricted to a stratum $U_\alpha^*$ is locally constant. In particular, when $U$ is trivially stratified, i.e., if $U$ has only one stratum, then the cosheaf ${\mc O}^E$ is locally constant. The condition on ${\mc O}^E$ also implies that fibres over deeper strata are more refined than fibres over higher strata. Moreover, one has the following fact.

\begin{lemma}
For each stratum $U_\alpha^* \subset U$ and the induced immersion $\tilde U_\alpha \to U$, the pullback bundle $E|_{\tilde U_\alpha} \to \tilde U_\alpha$ has a locally constant stratification which extends the stratification of fibres of $E|_{U_\alpha^*}$.
\end{lemma}

\begin{proof}
One only needs to identify the stratification at non-embedding point. Let $x \in U_\alpha \setminus U_\alpha^*$. A local branch of $\tilde U_\alpha$ corresponds to a local stratum $\beta \in {\mc O}_x^U$. Choose a stratified bundle chart of $E$ over a small neighborhood $U_x$ of $x$. Then for each $y \in (U_x)_\beta^*$, the bundle chart sends the stratification on $E_y$ to a stratification on $E_x$. If $z \in (U_x)_\beta^*$ is another point, as the local stratum $(U_x)_\beta^*$ is connected by our convention, one can show that the stratification on $E_x$ is independent of the choice of $y$. The independence of stratified bundle charts can also be verified.
\end{proof}

To have a stratification on the total space of a stratified vector bundle, one needs an additional piece of information, namely a cosheaf map.

\begin{defn}\label{defn_virtual_manifold}
A {\bf stratified virtual manifold}, written as $(U, E)$, consists of a stratified manifold $U$, a stratified vector bundle $E \to U$, and a cosheaf morphism $\rho: {\mc O}^U \to {\mc O}^E$.
\end{defn}

Given a stratified virtual manifold $(U, E)$, the total space of $E$ has a stratification indexed by $\alpha \in {\mc O}^U (U)$ with corresponding stratum
\beqn
E_\alpha^* = \bigsqcup_{x\in U_\alpha^*} \Big\{ v \in (E_x)_{\beta_x}^*\ |\ \exists \alpha_x \in {\mc O}_x^U, \beta_x \in {\mc O}_x^E,\ \rho_x(\alpha_x) = \beta_x, \alpha_x \mapsto \alpha \Big\}.
\eeqn
One can check that the total space $E$ is also a stratified manifold.



\begin{defn}\label{defn_flat_stratified_bundle}
Let $U$ be a stratified manifold. A {\bf flat stratified vector bundle} over $U$ is a smooth stratified vector bundle $E \to U$ together with a flat connection $\nabla^E$ such that for each $x \in U$, the flat bundle chart of $E$ near $x$ induced by parallel transport is also a stratified bundle chart.
\end{defn}

\subsubsection{NC structures and splittings on stratified vector bundles}

There are two ways to define normal complex structures or splittings on stratified vector bundles. The weak version will be simply requiring NC structures or splittings on fibres. The strong version will require the existence of local trivializations of such fibrewise structures. In this paper, we only use the weak version. 

\begin{defn}\label{defn_NCS_virtual_manifold}
Let $U$ be a stratified manifold and $E \to U$ be a stratified vector bundle. 

\begin{enumerate}

\item An {\bf NC structure} on $E$ consists of, for each stratum $U_\alpha^*$, a fibrewise NC structure on the bundle $E|_{\tilde U_\alpha}$ such that, for all $\alpha \leq \beta$, each $\tilde x \in \tilde U_\beta$ with immersion image $x \in U_\alpha^*$, the identification $(E|_{\tilde U_\beta})|_{\tilde x} \cong E_x$ preserves the NC structure. 

\item A {\bf splitting} of $E$ consists of, for each stratum $U_\alpha^*$, a fibrewise splitting of $E|_{\tilde U_\alpha}$ such that, for all $\alpha \leq \beta$, the identification $E|_{\tilde x} \cong E_x$ preserves the splitting.

\item When $E$ is flat, a splitting of $E$ is called flat if the fibrewise splitting of $E|_{\tilde U_\alpha}$ is parallel with respect to the pullback connection.

\item A {\bf normally complex stratified manifold} (NCS manifold) consists of a stratified manifold $U$ and an NC structure on $TU$.

\item A {\bf normally complex stratified virtual manifold} (NCS virtual manifold for short) consists a stratified virtual manifold $(U, E)$ with an NC structure on $TU$ and a complex structure on $E$. It is called a {\bf flat} NCS virtual manifold if $E$ is a flat stratified vector bundle.
\end{enumerate}
\end{defn}


\begin{lemma}
Let $(U, E)$ be a flat NCS virtual manifold. Then $TE\to E$ has an induced NC structure. 
\end{lemma}

\begin{proof}
The connection $\nabla^E$ induces a splitting of the exact sequence on $E_\alpha^*$:
\beqn
\xymatrix{ 0 \ar[r] & \pi_E^* (E/E_{\beta(\alpha)}) \ar[r] & TE/TE_\alpha^* \ar[r] &  \pi_E^* NU_\alpha^* }.
\eeqn
Hence the normal bundle $TE/TE_\alpha^*$ has the induced complex structure. The complex structure extends to the normal bundle of the immersion $E|_{\tilde U_\alpha} \to E$ as the consideration is local. 
\end{proof}

\subsubsection{Straightenings on manifolds}

We would like to consider sections which behave like complex polynomials in normal directions to each stratum. For this purpose one needs to specify those normal directions by the so-called ``straightening'' process.

\begin{notation}
Let $M$ be a smooth manifold and $S\subset M$ be a closed submanifold. The normal bundle $NS$ is defined via the exact sequence
\beq\label{tangent_normal_complex}
\xymatrix{ 0 \ar[r] & TS \ar[r] & TM|_S \ar[r] & NS \ar[r] & 0 }.
\eeq
For each continuous function $\epsilon: S \to {\mb R}_+$, let $N^\epsilon S \subset NS$ be the disk bundle of radius $\sqrt{\epsilon}$ (with respect to a norm on $NS$ specified in the context). In many situations, we do not need to specify the function $\epsilon$; in those cases, $N^\epsilon S$ simply means a sufficiently small disk bundle.
\end{notation}

\begin{defn} 
A {\bf tubular neighborhood} of a locally closed submanifold $S \subset M$ is an open embedding
\beqn
\rho: N^\epsilon S \to M
\eeqn
such that 1) $\rho$ extends the inclusion $S \hookrightarrow M$ and 2) the fibrewise derivative $d\rho: NS \to TM|_S$ defines a splitting of the exact sequence \eqref{tangent_normal_complex}. Denote the image of a tubular neighborhood $(N^\epsilon S, \rho)$ by $|N^\epsilon S|$.


\end{defn}

The case of stratified manifolds is more complicated because one needs to require a certain compatibility condition for adjacent strata. Let $U$ be a stratified manifold (Definition \ref{defn_manifold_stratification}). As the normal bundle $NU_\alpha^*$ is stratified, (locally) for each $\beta > \alpha$, there is a subbundle $(NU_\alpha^*)_\beta \subset NU_\alpha^*$ which is the normal bundle of $U_\alpha^* \hookrightarrow U_\beta^*$. If one has a stratified tubular neighborhood $\rho_\alpha: N^\epsilon U_\alpha^* \to U$, then for each $(x, v) \in (N^\epsilon U_\alpha^*)_\beta$, $\rho_\alpha(x, v) \in U_\beta^*$. Therefore, we would like the normal direction to be compatible between adjacent strata.

Now we describe the corresponding notions of tubular neighborhoods in stratified manifolds. 

\begin{defn}\label{defn_NC_tubular}
Let $U$ be a stratified  manifold. 
\begin{enumerate}

\item A {\bf stratified tubular neighborhood} of a stratum $U_\alpha^*$ is a tubular neighborhood $\rho_\alpha: N^\epsilon U_\alpha^* \to U$ which is a stratified open embedding.

\item A {\bf straightening} of $U$ along $NU_\alpha^*$ consists of a stratified tubular neighborhood $\rho_\alpha: N^\epsilon U_\alpha^* \to U$ together with a splitting of the normal bundle $NU_\alpha^* \to U_\alpha^*$. Two straightenings along $NU_\alpha^*$ are {\bf germ equivalent} if they have the same splitting of the normal bundle $NU_\alpha^*$ and their tubular neighborhoods agree within a sufficiently small disk bundle in $NU_\alpha^*$. 

\item When $U$ is an NCS manifold, a straightening along $NU_\alpha^*$ is said to be compatible with the NC structure if the splitting of $NU_\alpha^*$ is complex.


\end{enumerate}
\end{defn}

\begin{lemma}\label{lemma_compatible_straightening}
Let $U$ be a stratified manifold and $U_\alpha^*$ be a stratum. A straightening along $NU_\alpha^*$ with image $|N^\epsilon U_\alpha^*|$ canonically induces a straightening along $N(U_\beta^* \cap |N^\epsilon U_\alpha^*|)$ for each $\beta > \alpha$. When $U$ is an NCS manifold and the straightening along $NU_\alpha^*$ is compatible with the NC structure, so is the induced straightening.
\end{lemma}

\begin{proof}
We can restrict the consideration to the open disk bundle $N^\epsilon U_\alpha^*$ hence assume $U = N^\epsilon U_\alpha^*$. Within the total space of $NU_\alpha^*$, the stratum corresponding to $\beta$ is the union of subbundles. Then the fibrewise splitting of $NU_\alpha^*$ induces a splitting of $NU_\beta^*$ within $|N^\epsilon U_\alpha^*|$. The tubular neighborhood of $U_\beta^*$ is essentially induced by the same map $\rho_\alpha$ with domain regarded as a disk bundle of $NU_\beta^*$. The compatibility with the NC structure is readily to check.
\end{proof}

\begin{defn}\label{defn_manifold_straightening}
Let $U$ be a stratified manifold. A {\bf straightening} of $U$ consists of a collection of germs of straightenings along $NU_\alpha^*$ for all strata $U_\alpha^* \subset U$ such that for each pair $\alpha < \beta$, the restriction of the straightening along $NU_\beta^*$ to $|N^\epsilon U_\alpha^*|$ for a sufficiently small $\epsilon$ agrees with the straightening along $N(U_\beta^* \cap |N^\epsilon U_\alpha^*|)$ induced from the straightening along $NU_\alpha^*$ (see Lemma \ref{lemma_compatible_straightening}) as germs. When $U$ has an NCS structure, a straightening of $U$ is required to be compatible with the fibrewise complex structures of $NU_\alpha^*$ for all stratum $\alpha$.
\end{defn}

The existence result of straightenings will be proved when we discuss the straightenings of stratified orbifolds.

To locally view sections of a vector bundle as maps between vector spaces, one also needs to trivialize the vector bundle in a way compatible with the straightening on the base. In this paper we ease the pain of choosing such ``bundle straightenings'' by restricting the consideration to flat bundles. Let $U$ be a stratified manifold and $E \to U$ be a flat stratified vector bundle (Definition \ref{defn_flat_stratified_bundle}). Suppose one is given a straightening along $NU_\alpha^*$ for a stratum $U_\alpha^*$. Then using the flat connection and parallel transport along the normal direction in $|N^\epsilon U_\alpha^*|$, one obtains a bundle isomorphism
\beqn
E|_{|N^\epsilon U_\alpha^*} \cong \pi_\alpha^* (E|_{U_\alpha^*})
\eeqn
where $\pi_\alpha: |N^\epsilon U_\alpha^*| \to U_\alpha^*$ is the tubular neighborhood projection. The flatness of the connection implies a certain compatibility condition.

\begin{lemma}
Suppose $U$ is equipped with a straightening and $E \to U$ is a flat stratified vector bundle. Then for each stratum $U_\alpha^*$, $x \in U_\alpha^*$, $\beta \in {\mc O}^U_x$, and normal vector $v = (v_\beta, v_\beta^\vee) \in N_x^\epsilon U_\alpha^*$ with respect to the splitting $N_x U_\alpha^* \cong (N_x U_\alpha^*)_\beta \oplus (N_x U_\alpha^*)_\beta^\vee$, the diagram commutes.
\beqn
\xymatrix{ E_x \ar[r] \ar[rd]  &    E_{\rho_\alpha(x, v_\beta)}  \ar[d]\\
&   E_{\rho_\alpha(x, v)}}
\eeqn
\end{lemma}

\begin{proof}
This is because the connection $\nabla^E$ used to define the parallel transports is flat.
\end{proof}

\subsection{Normally complex stratified sections}

\subsubsection{Fibrewise complex stratified maps}

\begin{defn}\label{defn_NCS_bundle}
Let $Y$ be a manifold regarded as trivially stratified. 
A {\bf complex stratified virtual vector bundle} consists of a pair of complex stratified vector bundles $(F, E)$ together with a cosheaf map $\lambda: {\mc O}^F \to {\mc O}^E$.
    
\end{defn}

Given a complex stratified virtual vector bundle $(F, E)$ over $Y$, we will consider certain nonlinear bundle maps from $F$ to $E$ which are fibrewise normally complex maps between fibres. Let 
\beqn
C_{\rm stratified}^\infty( F, E) \subset C^\infty(F, E) 
\eeqn
be the space of smooth (nonlinear) bundle maps $S: F \to E$ satisfying 
\beqn
S( (F_x)_{\alpha})  \subset (E_x)_{\lambda_x(\alpha)}, \ \forall x\in Y,\ \alpha \in {\mc O}^F_x.
\eeqn
Here $\lambda_x: {\mc O}^F_x \to {\mc O}_x^E$ is induced from the cosheaf map $\lambda$. For each $d \geq 0$, there is a vector bundle 
\beqn
 {\rm Poly}{}_{\rm stratified}^d(F, E) \to Y
\eeqn
of fibrewise normally polynomial maps of degree at most $d$ whose fibre at $x \in Y$ is the space $
 {\rm Poly}{}^d(\mb{F}_x, \mb{E}_x)$, where $({\mb F}_x, {\mb E}_x)$ is the complex stratified virtual space corresponding to the fibres  at $x$ and the poset map $\lambda_x$.

\begin{defn}
A nonlinear bundle map $S: F \to E$ is called a {\bf normally complex stratified bundle map} (NCS map) if it is contained in the space
\beqn
C^\infty(F) \cdot \Gamma(  {\rm Poly}{}^d_{\rm stratified}(F, E)) \subset C_{\rm stratified}^\infty( F, E)
\eeqn
for some $d \geq 0$. Let $C_{\rm NCS}^\infty(F, E)$ be the space of all NCS bundle maps, i.e.
\beqn
C_{\rm NCS}^\infty (F, E)  = \colim_{d} C^\infty(F) \cdot \Gamma( {\rm Poly}{}_{\rm stratified}^d(F, E)) \subset C_{\rm stratified}^\infty( F, E ).
\eeqn
\end{defn}

\subsubsection{NCS sections}

Given a stratified virtual manifold $(U, E)$ (which contains a cosheaf map $\rho: {\mc O}^U \to {\mc O}^E$, denote by 
\beqn
\Gamma_{\rm stratified} ( U, E ) \subset \Gamma (U, E)
\eeqn
the subset of {\bf stratified sections}, i.e., sections $S: U \to E$ satisfying 
\beqn
S(x) \in (E_x)_{\rho_x(\alpha_x)}\ \forall x \in U
\eeqn
where $\alpha_x \in {\mc O}_x^U$ is the local stratum containing $x$.

We would like to consider sections which are locally normally complex maps discussed before. The notion depends on how we locally linearize the manifold, which is the reason why one needs the notion of straightening. Let $(U, E)$ be an NCS virtual manifold and $S \in \Gamma_{\rm stratified}  ( U, E)$ be a stratified section. Suppose $U$ has a straightening, in particular, a collection of tubular neighborhoods $\rho_\alpha: N^\epsilon U_\alpha^* \to U$. Moreover, over the neighborhood $|N^\epsilon U_\alpha^*|$ there is an induced projection map
\beqn
\pi_\alpha: |N^\epsilon U_\alpha^*| \to U_\alpha^*.
\eeqn
Using the parallel transport, one obtains a bundle isomorphism
\beqn
\hat \rho_\alpha: E|_{|N^\epsilon U_\alpha^*|} \cong \pi_\alpha^* E|_{U_\alpha^*}.
\eeqn
Using $\rho_\alpha$ and $\hat\rho_\alpha$, one obtains a bundle map
\beqn
S_\alpha:= \hat \rho_\alpha \circ S \circ \rho_\alpha^{-1}: N^\epsilon U_\alpha^* \to E|_{U_\alpha^*}
\eeqn
over $U_\alpha^*$. As the connection preserves the stratification, 
this is a stratified bundle map, denoted by 
\beqn
S_\alpha \in C_{\rm stratified}^\infty ( N^\epsilon U_\alpha^*, E|_{U_\alpha^*}).
\eeqn
We would like to require that $S_\alpha$ is the restriction of an NCS bundle map.

\begin{defn}
A stratified section $S\in \Gamma_{\rm stratified} (U, E)$ is called     a {\bf normally complex stratified section} (NCS section) with respect to the straightening of $U$ if for each stratum $U_\alpha^*$, the stratified bundle map $S_\alpha \in C^\infty_{\rm stratified} ( N^\epsilon U_\alpha^*, E|_{U_\alpha^*})$ is the restriction of an element of $C_{\rm NCS}^\infty(NU_\alpha^*, E|_{U_\alpha^*})$. 
\end{defn}



Denote by
\beqn
\Gamma_{\rm NCS}(U, E) \subset \Gamma_{\rm stratified}  (U, E)
\eeqn
the subset of NCS sections. It is a $C^\infty(U)$-submodule.

\begin{lemma}\label{lemma235}
$\Gamma_{\rm NCS}( U, E )$ is $C^0$-dense in $\Gamma_{\rm stratified}  ( U, E )$.
\end{lemma}

\begin{proof}
This is a corollary of Lemma \ref{lemma_c0_density}.
\end{proof}

\subsubsection{Stabilizations}

Let $(U, E)$ be a stratified virtual manifold and let $F \to U$ be another  stratified vector bundle together with a cosheaf morphism ${\mc O}^U \to {\mc O}^F$. 

\begin{defn}\label{defn_NCS_manifold_stabilization}
The stabilization of $(U, E)$ by $F$ is the NCS virtual manifold $(\hat U, \hat E)$ where
\begin{enumerate}

\item $\hat U$ is the total space of $F$ equipped with the stabilization stratification induced from ${\mc O}^U \to {\mc O}^F$. Notice that there is a cosheaf map
\beqn
{\mc O}^{\hat U} \to \pi_F^{-1} ({\mc O}^U).
\eeqn

\item $\hat E \to \hat U$ is the stratified vector bundle with structural cosheaf being
\beqn
{\mc O}^{\hat E} = \pi_F^{-1}( {\mc O}^E \times {\mc O}^F)
\eeqn
and the fibres have the product linear stratification.

\item The cosheaf morphism ${\mc O}^{\hat U} \to {\mc O}^{\hat E}$ is the composition
\beqn
{\mc O}^{\hat U} \to \pi_F^{-1} ({\mc O}^U )\to \pi_F^{-1} ({\mc O}^E \times {\mc O}^F) = {\mc O}^{\hat E}.
\eeqn
\end{enumerate}
\end{defn}

\begin{lemma}\label{lemma_stabilization_NCS_manifold}
When $E$ and $F$ are flat, $(\hat U, \hat E)$ is canonically an NCS virtual manifold and $\hat E$ is flat. \end{lemma}

\begin{proof}
The cosheaf morphism ${\mc O}^U \to {\mc O}^F$ provides, for each stratum $U_\alpha^*$, a subbundle $F_\alpha \subset F|_{U_\alpha^*}$. Let $\hat U_\alpha^*$ be the total space of $F_\alpha$. The collection of all $\hat U_\alpha^*$ form a structure of stratified manifold of the total space $\hat U$. Then one has the exact sequence of bundles over $U_\alpha^*$
\beqn
\xymatrix{ 0 \ar[r] & \pi_F^*(F|_{U_\alpha^*}/ F_\alpha) \ar[r] & T(F|_{U_\alpha^*})/ TF_\alpha \ar[r] & \pi_F^* NU_\alpha^* \ar[r] & 0 }.
\eeqn
As $F$ is equipped with a flat connection, there is a corresponding splitting of this exact sequence. Then the complex structure on $NU_\alpha^*$ and the complex structure on $F|_{U_\alpha^*}/F_\alpha$ induce via the splitting a complex structure on $T(F|_{U_\alpha^*})/ TF_\alpha$. It is easy to check that these complex structures for all $\alpha$ gives an NC structure on the total space of $F$. Moreover, the flat connections on $E$ and $F$ are pulled back via $F \to U$ to a flat connection on $\hat E \to \hat U$.
\end{proof}

 On the other hand, there is a natural {\bf stabilization map}
\beq\label{stabilization_map}
\Gamma_{\rm stratified} (U, E) \to \Gamma_{\rm stratified} (\hat U, \hat E),\ S \mapsto \pi_{F}^* S \oplus \tau_{F}
\eeq
where $\tau_{F}$ is the tautological section of $\pi_{F}^* F$. The condition for the stabilized map being an NCS map also depends on the straightening on the total space. 

\begin{defn}\label{defn_flat_bundle_splitting}
Let $U$ be a stratified manifold and $F \to U$ be a flat complex stratified vector bundle. A {\bf splitting} of $F$ consists of splittings of fibres of $F$ (Definition \ref{defn_linear_splitting}) such that for each $x \in U$ and a nearby point $y$, the isomorphism $F_y \cong F_x$ defined by parallel transport sends the complement of a stratum of $F_y$ to the complement of the corresponding stratum of $F_x$.
\end{defn}

\begin{lemma}\label{lemma_stabilization_straightening}
In the situation of Lemma  \ref{lemma_stabilization_NCS_manifold}, suppose $U$ is equipped with a straightening and $F \to U$ is equipped with a splitting in the sense of Definition \ref{defn_flat_bundle_splitting}, then the total space $\hat U$ of $F$ has an induced straightening. In this situation, the stabilization map \eqref{stabilization_map} sends NCS sections to NCS sections.
\end{lemma}

\begin{proof}
Let $F_\alpha^\vee \to U_\alpha^*$ be the complement of $F_\alpha$ provided by the splitting of $F$. Then the normal bundle $N\tilde U_\alpha^*$ is identified with $\pi_{F_\alpha}^* F_\alpha^\vee \oplus \pi_{F_\alpha}^* NU_\alpha^*$. The parallel transport of $F$ along normal fibres of $N^\epsilon U_\alpha^*$ provides the tubular neighborhoods of $\tilde U_\alpha^*$. Moreover, the splitting of $NU_\alpha^*$ and the induced splitting of $F_\alpha^\vee$ provides a splitting of $N\hat U_\alpha^*$. One can check the compatibility of different tubular neighborhoods in a straightforward way. The claim about the stabilization map is also obvious as one only needs to consider the local picture. 
\end{proof}

\subsection{Stratified orbifolds,  bundles, and virtual orbifolds}

Recall that an effective orbifold is a topological space ${\mc U}$ equipped with an atlas of mutually compatible smooth orbifold charts of the form 
\beqn
C = (G, U, \phi)
\eeqn
where $G$ is a finite group, $U$ is an effective smooth $G$-manifold (possibly with boundary or corners), and $\phi: U/G \to {\mc U}$ is a homeomorphism onto an open subset. A chart embedding from $C_1 = (G_1, U_1, \phi_1)$ to $C_2 = (G_2, U_2, \phi_2)$ consists of a group embedding $G_1 \to G_2$ and an equivariant open embedding $U_1 \to U_2$ such that the following diagram commutes.
\beqn
\xymatrix{  U_1 \ar[r] \ar[d]  &    U_2 \ar[d] \\
            {\mc U} \ar[r]_{{\rm Id}_{\mc U}}  & {\mc U} }
\eeqn
Here the vertical arrows are induced from $\phi_1$ and $\phi_2$. The compatibility condition means, for any two charts $C_1, C_2$ in the atlas, around any point $x$ in their overlap, there is a third chart $C_3$ around $x$ which embeds into both $C_1$ and $C_2$.

For effective orbifolds, we consider possibly non-effective suborbifolds. A suborbifold ${\mc Y}$ of ${\mc U}$ is a subset such that for each $y \in {\mc Y}$, there exists an orbifold chart $C = (G, U, \phi)$ of ${\mc U}$ around $y$, a smooth submanifold $Y \subset U$ containing $y$ and $G_Y \subset G$ the maximal subgroup fixing $Y$ setwise, such that the restriction 
\beqn
\phi: Y/G_Y \to {\mc Y}
\eeqn
is a homeomorphism onto an open neighborhood of $y$. Notice that the isotropy group of $y$, viewed as a point of ${\mc U}$ or a point of ${\mc Y}$ could be different. It is possible that $G_Y$ acts non-effectively on $Y$.

\begin{defn}\label{defn_stratified_orbifold}
An effective {\bf stratified orbifold} consists of an effective orbifold ${\mc U}$ together with a stratification 
    \beqn
    {\mc U} = \bigsqcup_{\alpha \in {\mc A}} {\mc U}_\alpha^*
    \eeqn
    satisfying the following conditions.

    \begin{enumerate}
        \item Each ${\mc U}_\alpha^*$ is a (possibly non-effective) suborbifold of ${\mc U}$.

        \item For any orbifold chart $C = (G, U, \phi)$, the pullback stratification by $\phi: U \to {\mc U}$ makes $U$ a stratified $G$-manifold. 
    \end{enumerate}
A {\bf normal complex structure} on ${\mc U}$ consists of $G$-invariant normal complex structures on $U$ for all charts $C = (G, U, \phi)$ such that all chart embeddings preserve these complex structures. A {\bf normally complex stratified orbifold} (NCS orbifold) is a stratified orbifold ${\mc U}$ together with a normal complex structure. 
\end{defn}

Now we define the corresponding notion of stratified vector bundles. 

\begin{defn}
Let ${\mc U}$ be a stratified orbifold. A {\bf (flat) stratified orbifold vector bundle} on ${\mc U}$, denoted by ${\mc E} \to {\mc U}$, consists of an orbifold vector bundle ${\mc E}$ together with, for each bundle chart $\hat C = (G, U, E, \hat \phi)$, a structure of $G$-equivariant (flat) stratified vector bundle, which are compatible with orbibundle embeddings.
\end{defn}

\begin{defn}\label{defn_NCS_virtual_orbifold}
A {\bf NCS virtual orbifold}, denoted by $({\mc U}, {\mc E})$, consists of an NCS orbifold ${\mc U}$, a stratified flat complex orbifold vector bundle ${\mc E} \to {\mc U}$, and for each bundle chart $\hat C = (G, U, E, \hat\phi)$, a structure of $G$-invariant NCS virtual manifolds on $(U, E)$ which are invariant under bundle chart embeddings. A {\bf stratified section} of an NCS virtual orbifold $({\mc U}, {\mc E})$ is a section ${\mc S}: {\mc U} \to {\mc E}$ such that for each bundle chart $\hat C = (G, U, E, \hat \phi)$, the pullback of ${\mc S}$ is a $G$-equivariant stratified section of $(U, E)$. Let $\Gamma_{\rm stratified}({\mc U}, {\mc E})$ denote the set of smooth stratified sections.  
\end{defn}

\subsubsection{Straightenings on orbifolds}

\begin{defn}\label{defn_orbifold_straightening}
Let ${\mc U}$ be an effective stratified orbifold. A {\bf straightening} of ${\mc U}$ consists of, for each chart $C = (G, U, \phi)$, a $G$-equivariant straightening on $U$, such that for each chart embedding $\phi_{21}: C_1 \to C_2$ and for stratum $U_{1, \alpha_1}^* \subset U_1$ which is sent to $U_{2, \alpha_2}^* \subset U_2$, the following diagram commutes as germs.
\beqn
\xymatrix{   N^{\epsilon_1} U_{1, \alpha_1}^* \ar[r]^-{\rho_{1, \alpha_1}}  \ar[d]_-{d\phi_{21}^U}   &      U_1 \ar[d]^-{\phi_{21}^U} \\
N^{\epsilon_2} U_{2, \alpha_2}^*  \ar[r]_-{\rho_{2, \alpha_2}}   &  U_2  }
\eeqn
\end{defn}

The following is the existence theorem for orbifold straightenings.

\begin{prop}\label{prop_orbifold_straightening}
Let ${\mc U}$ be an effective stratified orbifold. Let $Y \subset {\mc U}$ be a closed subset. Given any straightening in an open neighborhood of $Y$, there exists a straightening of ${\mc U}$ which agrees with the given one near $Y$. If ${\mc U}$ is normally complex and the given straightening respects the NC structure, then the new straightening can be chosen to respect the NC structure. 
\end{prop}

\begin{proof}
See Subsection \ref{straightening_proof}.
\end{proof}

\subsubsection{Multisections}

For any set $A$ and positive integer $k$, let ${\rm Sym}^k (A)$ be the $k$-th symmetric power of $A$. There is a natural map 
\beqn
{\rm Sym}^k(A) \to {\rm Sym}^{kl}(A)
\eeqn
by repeating each entry $l$ times. Then denote
\beqn
{\rm Sym}^\bullet(A):= \colim {\rm Sym}^k(A).
\eeqn

Let ${\mc E} \to {\mc U}$ be an orbifold vector bundle. There is the bundle
\beqn
{\rm Sym}^\bullet({\mc E}) \to {\mc U}
\eeqn
whose fibre at $x \in {\mc U}$ is ${\rm Sym}^\bullet({\mc E}_x)$. A {\bf smooth multisection} of ${\mc E}$ is a section
\beqn
{\mc S}: {\mc U} \to {\rm Sym}^\bullet {\mc E}
\eeqn
such that for each $p \in {\mc U}$, there exists an orbifold bundle chart $\hat C = (G, U, E, \hat\phi)$, a positive integer $k$, and a $k$-tuple of smooth sections (called a {\bf local lift} of ${\mc S}$)
\beqn
(S_1, \ldots, S_k) \in \Gamma(E)^k
\eeqn
such that locally ${\mc S}$ is given by $[S_1, \ldots, S_k]$. A multisection is called {\bf  transverse} if each branch of each local lift is transverse to the zero section in the classical sense. Let
\beqn
\Gamma^{\rm multi}({\mc U}, {\mc E})
\eeqn
denote the set of all smooth multisections of ${\mc E}$.

We also need to measure the distance between two multisections.

\begin{defn}\label{defn_closedness}
Suppose ${\mc E}$ is equipped with a norm and ${\mc S}, {\mc S}' \in \Gamma_{\rm multi}({\mc U}, {\mc E})$. We say that ${\mc S}$ and ${\mc S}'$ are {\bf $\delta$-close} if there exists a covering of ${\mc U}$ by bundle charts $\hat C = (G, U, E, \hat\phi)$, such that for each of these charts, ${\mc S}$ resp. ${\mc S}'$ have lifts $(S_1, \ldots, S_k)$ resp. $(S_1', \ldots, S_k')$ such that 
\beqn
\sup_{1 \leq i \leq k} \| S_i - S_i'\|_{C^0(U)} < \delta.
\eeqn
\end{defn}

\begin{defn}
Let $({\mc U}, {\mc E})$ be an NCS virtual orbifold (Definition \ref{defn_NCS_virtual_orbifold}).

\begin{enumerate}

\item A multisection ${\mc S} \in \Gamma^{\rm multi}({\mc U}, {\mc E})$ is said to be {\bf stratified} if for each bundle chart $\hat C = (G, U, E, \hat\phi)$, the corresponding multisection $S$ of $E\to U$ has each local branch a stratified section. Denote by $\Gamma_{\rm stratified}^{\rm multi}({\mc U}, {\mc E})$ be the space of stratified multisections.

\item A stratified section ${\mc S} \in \Gamma_{\rm stratified}^{\rm multi}({\mc U}, {\mc E})$ is called {\bf normally complex} with respect to a given straightening on $({\mc U}, {\mc E})$ if for each bundle chart on which ${\mc S}$ admits a lift, each branch of this lift is a normally complex section with respect to the induced straightening on the chart. 
\end{enumerate}
\end{defn}

Below is a corollary to Lemma \ref{lemma235}.

\begin{cor}
For any continuous stratified single valued section ${\mc S}_0: {\mc U} \to {\mc E}$, there exists a single-valued NCS section of $({\mc U}, {\mc E})$ which can be arbitrarily $C^0$-close to ${\mc S}_0$.
\end{cor}





\subsection{Construction of straightening}\label{straightening_proof}

\subsubsection{Formal properties of straightenings}

Let ${\mc U}$ be a stratified orbifold with an NC structure. We define a presheaf $\mc{ST} \to {\mc U}$ whose value on any open set $O \subset {\mc U}$ is the set $\mc{ST}(O)$ of all straightenings on $O$ regarded as itself an NCS orbifold. When $O_1 \subset O_2$, there is hence a natural restriction map
\beqn
\mc{ST}(O_2) \to \mc{ST}(O_1).
\eeqn
It is easy to see that $\mc{ST}$ satisfies gluing, hence a sheaf. Then for each closed set $Y \subset {\mc U}$, one can define
\beqn
\mc{ST}(Y):= \colim_{Y \subset O} \mc{ST}(O)
\eeqn
which also admits the restriction map: 
\beqn
\mc{ST}(Y_2) \to \mc{ST}(Y_1)\ \forall Y_1 \subset Y_2.
\eeqn
One has the corresponding gluing property for closed sets.

\begin{lemma}\label{lemma_closed_sheaf}
Suppose $Y_1, Y_2 \subset {\mc U}$ are closed sets, $a_1 \in \mc{ST}(Y_1)$, $a_2 \in \mc{ST}(Y_2)$ such that their restrictions to $Y_1 \cap Y_2$ agree. Then there exists a unique $a_1 \# a_2 \in \mc{ST}(Y_1 \cup Y_2)$ whose restriction to $Y_1$ resp. $Y_2$ is $a_1$ resp. $a_2$. 
\end{lemma}

\begin{proof}
Choose an open neighborhood $O_1$ of $Y_1$ resp. $O_2$ of $Y_2$ and $\tilde a_1 \in \mc{ST}(O_1)$ resp. $\tilde a_2 \in \mc{ST}(O_2)$ representing the germ $a_1$ resp. $a_2$. As the orbifold is a normal space, by shrinking $O_1$ and $O_2$, one can assume that the restrictions of $a_1$ and $a_2$ to $O_1 \cap O_2$, which is an open neighborhood of $Y_1 \cap Y_2$, agree. 
\end{proof}

\subsubsection{Chart-wise induction}

The proof of Proposition \ref{prop_orbifold_straightening} is based on a reduction to local charts. 

\begin{lemma}\label{lemma_straightening_chart}
Proposition \ref{prop_orbifold_straightening} holds for the special case when ${\mc U}$ is covered by a single stratified orbifold chart $(G, U, \phi)$.
\end{lemma}

Assuming this lemma. We can prove Proposition \ref{prop_orbifold_straightening}. 

\begin{proof}[Proof of Proposition \ref{prop_orbifold_straightening}]
One can choose countably many stratified orbifold charts $(G_i, U_i, \phi_i)$ and $G_i$-invariant precompact open subsets $U_i' \subset U_i$ such that ${\mc U}$ is covered by the union of $\phi_i(U_i')$. Denote $Y_i:= \phi_i( \phi_i^{-1}(Y))$. 

Then we start the inductive construction. By Lemma \ref{lemma_straightening_chart}, there exists a straightening on $\phi_1(U_1)$ which agrees with the existing one on $Y_1$. This chartwise straightening then induces an element of $\mc{ST}(\phi_1(\ov{U_1'}))$ which agrees with the existing one on $\mc{ST}(Y)$. By closed set gluing (Lemma \ref{lemma_closed_sheaf}), they induce an extension in $\mc{ST}(Y \cup \phi_1(\ov{U_1'}))$. Replace $Y$ by the larger closed set $Y \cup \phi_1(\ov{U_1'})$ and replace $U_1$ by $U_2$, one can obtain an extension to a larger closed set. Since one can require the countable open cover to be locally finite, inductively, one obtains a straightening on the whole ${\mc U}$ which agrees with the existing one near $Y$.
\end{proof}

\subsubsection{Chartwise extension}

Now we prove Lemma \ref{lemma_straightening_chart}. We work in the normally complex situation. We start with some basic linear algebra.

\begin{defn}
Let $\mb{V}$ be an NCS linear space. A {\bf normal Hermitian structure} on ${\mb V}$ is an Hermitian inner product on the quotient $V/V_0$ with respect to the complex structure on $V/V_0$. Notice that a normal Hermitian structure together with a decomposition $V  = V_0 \oplus \check V_0$ induces a splitting of ${\mb V}$ which is compatible with the NC structure.
\end{defn}

\begin{lemma}\label{lemma_linear_NH}
Let ${\mb V}$ be an NCS linear space with an action by a finite group $G$. Then there exists a $G$-invariant normal Hermitian structure on ${\mb V}$. Moreover, any two $G$-invariant normal Hermitian structures can be connected.
\end{lemma}

\begin{proof}
Choose an arbitrary $G$-invariant inner product $\langle \cdot, \cdot \rangle$. The orthogonal complement $V_0^\bot \subset V$ is then identified with $V/V_0$. However, the induced inner product on $V_0^\bot$ may not be Hermitian with respect to the complex structure $I^{V/V_0}$ on $V/V_0$. We redefine the inner product by replacing its restriction to $V_0^\bot$ by
\beqn
(v, w) \mapsto \frac{1}{2} \left( \langle v, w \rangle + \langle I^{V/V_0} v, I^{V/V_0} w \rangle\right)
\eeqn
which is Hermitian and $G$-invariant. Notice that if $\langle \cdot, \cdot \rangle$ is already Hermitian, then this construction does not alter it. 

Lastly, two $G$-invariant normally Hermitian structures can be connected by taking convex combinations. 
\end{proof}

Now consider the chartwise construction. Let ${\mb V}$ be an NCS linear space acted linearly by a finite group $G$. Let $U \subset {\mb V}$ be a $G$-invariant ball centered at the origin. Let $Y \subset U$ be a $G$-invariant closed set.

\begin{lemma}
Let $U_0:= U\cap V_0$ be the lowest stratum of $U$. Denote $Y_0:= U_0 \cap Y$. Suppose the tangent bundle restriction $TU|_{U_0}$ is equipped with a $G$-invariant NC splitting near $Y_0$. Then there exists a $G$-invariant NC splitting of $TU|_{U_0}$ which agrees with the existing one near $Y_0$.
\end{lemma}

\begin{proof}
First, an NC splitting of the tangent bundle $TU|_{U_0}$ induces a splitting of the exact sequence
\beq\label{bundle_exact_sequence}
\xymatrix{ 0 \ar[r] & TU_0 \ar[r] & TU|_{U_0} \ar[r] & NU_0 \ar[r] & 0 }
\eeq
and any two such splittings can be interpolated by convex linear combinations using $G$-invariant cut-off functions on $U_0$. For the given NC splitting near $Y_0$, one chooses a family of $G$-invariant Hermitian inner products on $NU_0$ near $Y_0$, which hence induces a normal Hermitian structure on $TU|_{U_0}$ near $Y_0$. On the other hand, by Lemma \ref{lemma_linear_NH}, one can also find a constant splitting of the exact sequence \eqref{bundle_exact_sequence} and a normal Hermitian structure on $TU|_{U_0}$. Since both the splitting of \eqref{bundle_exact_sequence} and the normal Hermitian structure can be interpolated using convex linear combinations, one can use a $G$-invariant smooth cut-off function supported near $Y_0$ to obtain an NC splitting of $TU|_{U_0}$ which agrees with the existing one near $Y_0$.
\end{proof}

\begin{cor}
The existing $G$-equivariant straightening near $Y\subset U$ can be extended to one near $Y \cup U_0$ which agrees with the existing one near $Y$.
\end{cor}

\begin{proof}
First, choose a splitting of \eqref{bundle_exact_sequence} and an NC splitting of $TU|_{U_0}$ which agrees with the existing one near $Y \cap U_0$. Let $NU_0 \to TU|_{U_0}$ be the corresponding bundle inclusion, where for each $x \in U_0$, a vector $v \in N_x U_0$ is then identified with a vector in $V$. Then define the tubular neighborhood 
\beqn
\rho_1: N^\epsilon U_0 \to U,\ \rho(x, v) = x + v
\eeqn
which is clearly a $G$-invariant tubular neighborhood of $U_0$. Then choose a $G$-invariant cut-off function $\lambda: U_0 \to [0, 1]$ supported near $Y \cap U_0$. Let the existing tubular neighborhood be $\rho_0: N^\epsilon U_0 \to U$ (which is only defined near $Y\cap U_0$). Define
\beqn
\rho: N^\epsilon U_0 \to U,\ \rho(x, v) = \lambda(x) \rho_0(x, v) + (1- \lambda(x)) \rho_1(x, v).
\eeqn
This interpolation agrees with the existing one near $Y\cap U_0$. 
\end{proof}

To finish proving Lemma \ref{lemma_straightening_chart}, one use another layer of induction on strata. By deleting the lowest stratum $U_0$, one obtains a chart with one fewer stratum. Then one can inductively build tubular neighborhoods and NC splittings of normal bundles while preserving the symmetry.

\section{Whitney Stratification and Transversality}\label{section3}

In this section we define the transversality condition for normally complex stratified multisections, called NCS transversality. It shares similar features as the FOP transversality condition defined in \cite{Bai_Xu_integer}. In fact, the transversality condition depends on certain canonical Whitney stratification on a particular kind of complex algebraic variety.

\subsection{Canonical Whitney stratification on the variety $Z$}

\subsubsection{Review of Whitney stratifications}

Recall that in a smooth manifold $M$, an ordered pair of disjoint submanifolds $(U, V)$ is said to satisfy Whitney's {\bf condition (b)} if, for each $x \in \ov{U} \cap V$, for any two sequences $x_i \in U$ and $y_i \in V$ both of which converging to $x$ such that the sequence of secant lines $\ov{x_i y_i}$ converging to a line $l \subset T_x M$ and the sequence of tangent spaces $T_{x_i} U$ converging to a subspace $H \subset T_x M$, there holds $l \subset H$. Let $Z \subset M$ be a subset. A stratification of $Z$ is called a {\bf Whitney stratification} if each stratum is a smooth submanifold of $M$ and each pair of distinct strata satisfies Whitney's condition (b).

Whitney stratifications are convenient for discussing transversality against singular sets. A smooth map $f: N \to M$ between two manifolds is said to be {\bf transverse} to a subset $Z \subset M$ (with respect to a Whitney stratification ${\mf Z}$) if $f$ is transverse to all strata of $Z$. If $f$ is transverse to $Z$, then the pullback $f^* {\mf Z}$ (see \eqref{naive_pullback}) is a Whitney stratifications on $f^{-1}(Z)$.

Whitney \cite{Whitney_1965} proved that complex algebraic subvarieties in an ambient smooth variety always have a Whitney stratification. In fact, Whitney's constructive proof indeed provides a canonical one, which is ``minimal'' in the following sense. More precisely, a Whitney stratification ${\mf Z}$ on $Z \subset M$ induces a filtration 
\beqn
\cdots \supseteq {\mf Z}_n \supseteq {\mf Z}_{n-1} \supseteq \cdots\ \ {\rm where}\ {\mf Z}_n:= \bigsqcup_{{\rm dim} Z_\alpha^* \leq n} Z_\alpha^*.
\eeqn
We write ${\mf Z} < {\mf Z}'$ if there exists a dimension $k$ such that 
\beqn
{\mf Z}_l = {\mf Z}_l'\ \forall l > k\ {\rm and}\ {\mf Z}_k \subsetneq {\mf Z}_k'.
\eeqn
Then define ${\mf Z} \leq {\mf Z}'$ if either ${\mf Z} = {\mf Z}'$ and ${\mf Z} < {\mf Z}'$. This is a partial order among equivalence classes of Whitney stratifications (with connected strata).

Whitney's construction, which provides a canonical (also minimal) Whitney stratification of a complex algebraic variety, can be extended to the complex analytic category in a relative setting.

\begin{prop}\label{prop31}\cite[Proposition 3.6]{Bai_Xu_integer}
Let $M$ be a complex manifold and $Z \subset M$ be a closed complex analytic set. Let ${\mf M} = \{ M_\alpha^* \ |\ \alpha \in {\mc A}\}$ be a stratification of $M$ by strongly analytic submanifolds $M_\alpha^*$ (namely, $\ov{M_\alpha^*}$ and $\ov{M_\alpha^*} \setminus M_\alpha^*$ are both closed analytic sets). Then there exists a minimal refinement of the partition of $Z$
\beqn
{\mf M} \cap Z = \{ M_\alpha^* \cap Z\neq \emptyset \ |\ \alpha \in {\mc A} \},
\eeqn
denoted by ${\mf Z}$, which is a Whitney stratification. We call ${\mf Z}$ the {\bf canonical Whitney stratification} of $Z$ {\bf relative to} ${\mf M}$. Moreover, ${\mf Z}$ satisfies the following conditions.
\begin{enumerate}

\item Each stratum is a strongly analytic submanifold of $M$.

\item For any $C^\infty$ diffeomorphism $f: M \to M$ which preserves ${\mf M}$ and which preserves $Z$ setwise, then $f^* {\mf Z} = {\mf Z}$.

\item For any open subset $O \subset M$, the restriction ${\mf Z}|_O$ coincides with the canonical Whitney stratification of $Z\cap O$ relative to ${\mf M}|_O$. 
\end{enumerate}
\end{prop}

The last property stated in Proposition \ref{prop31} can be generalized as follows. It is useful in comparing canonical Whitney stratifications in different spaces and transferring transversality conditions among them. 

\begin{prop}\label{prop_submersion_pullback} \cite[Proposition B.20]{Bai_Xu_integer}
Let $M$, $Z$, and ${\mf M}$ be as in Proposition \ref{prop31}. Let $\pi: \tilde M \to M$ be a holomorphic submersion. Then
\beqn
\pi^* {\mf Z} =  \tilde {\mf Z}
\eeqn
where ${\mf Z}$ is the canonical Whitney stratification on $Z$ relative to ${\mf M}$ and $\tilde {\mf Z}$ is the canonical Whitney stratification on $\tilde Z = \pi^{-1}(Z)$ relative to $\pi^* {\mf M}$.
\end{prop}

\subsubsection{The universal vanishing loci}

Recall that for a complex stratified (CS) virtual space $(\mb{V}, \mb{W})$, one has the space of stratified polynomial maps ${\rm Poly}{}^d(\mb{V}, \mb{W})$. Define
\beqn
M^d(\mb{V}, \mb{W}):= V \times {\rm Poly}{}^d({\mb V}, {\mb W})
\eeqn
which is a complex vector space. Consider the subset
\beqn
Z^d ( {\mb V}, {\mb W} ):= \Big\{(v, P) \in V \times {\rm Poly}^d ({\mb V}, {\mb W} )\ |\ P(v) = 0 \Big\}
\eeqn
which is a complex algebraic variety.

\begin{defn}
For each $d \geq 0$, the {\bf canonical Whitney stratification} on $Z^d( {\mb V}, {\mb W})$, denoted by ${\mf Z}^d(\mb{V}, \mb{W})$, is the canonical Whitney stratification relative to the stratification on $M^d({\mb V}, {\mb W})$ with strata
\beqn
M_\alpha^*:= V_\alpha^* \times {\rm Poly}{}^d({\mb V}, {\mb W})
\eeqn
provided by Proposition \ref{prop31}.
\end{defn}

\subsection{Properties of the canonical Whitney stratification}

We need various properties of the canonical Whitney stratifications in order to established a well-defined and well-behaving transversality condition on stratified orbifolds. The discussion in this subsection is parallel to \cite[Section 3]{Bai_Xu_integer}. 

In this subsection, $(\mb{V}, \mb{W})$ always denote a complex stratified virtual vector space with underlying poset map ${\mc O}^V(V) \to {\mc O}^W(W)$.

\subsubsection{Invariance}

An automorphism of ${\mb V}$ resp. ${\mb W}$ is a linear automorphism which preserves the stratification and the complex structure. Then each element of ${\rm Aut}({\mb V}) \times {\rm Aut}({\mb W})$ induces a self-diffeomorphism on $V \times {\rm Poly}{}^d({\mb V}, {\mb W})$ which preserves the stratification and the subset $Z^d({\mb V}, {\mb W})$. 

\begin{prop}\label{prop34}
${\rm Aut}({\mb V})\times {\rm Aut}({\mb W})$ preserves the canonical Whitney stratification on $Z^d({\mb V}, {\mb W})$.
\end{prop}

\begin{proof}
This is the consequence of (2) of Proposition \ref{prop31}.
\end{proof}

This invariance property allows us to extend the canonical Whitney stratification to the family case. Let $Y$ be a smooth manifold and let $(F, E)$ be a complex stratified virtual vector bundle $Y$. Each fibre pair $({\mb F}_x, {\mb E}_x)$ is isomorphic to a fixed virtual vector space $({\mb V}, {\mb W})$. The structure group of the bundle $M^d(F, E) \to Y$ is ${\rm Aut}({\mb V}) \times {\rm Aut}({\mb W})$. There is also a subbundle
\beqn
Z^d({\mb F}, {\mb E}) \subset M^d({\mb F}, {\mb E}).
\eeqn
Then Proposition \ref{prop34} implies that there is a locally trivial Whitney stratification on the bundle, denoted by ${\mf Z}^d({\mb F}, {\mb E})$.

There is another invariance property related to nonlinear reparametrizations.

\begin{prop}
Let $\phi: V \to {\rm Aut}({\mb W})$ be a smooth family of automorphisms of ${\mb W}$. Consider the map 
\beqn
\begin{split}
\Phi: V \times {\rm Poly}{}^d({\mb V}, {\mb W}) \to &\ V \times {\rm Poly}{}^d({\mb V}, {\mb W})\\
(v, P) \mapsto &\ \Big( v, \phi(v)(P(\cdot)) \Big).
\end{split}
\eeqn
(This is well-defined because $\phi(v)$ is complex-linear.) Then 
\beqn
\Phi^* {\mf Z}^d({\mb V}, {\mb W})) = {\mf Z}^d(\mb{V}, \mb{W}).
\eeqn
\end{prop}

\begin{proof}
Because $\Phi$ preserves the stratification on $ M^d({\mb V}, {\mb W})$ as well as the set $ Z^d({\mb V}, {\mb W})$, this proposition follows from Proposition \ref{prop31}. 
\end{proof}

\subsubsection{Stratified transversality}

Let $\alpha \mapsto \beta(\alpha)$ denote the poset map ${\mc O}^V(V) \to {\mc O}^W(W)$ contained in the virtual vector space $({\mb V}, {\mb W})$. Because of the condition $P(V_\alpha) \subset W_{\beta(\alpha)}$ for $P \in  {\rm Poly}{}(\mb{V}, \mb{W})$, in general the variety $ {Z}{}^d(\mb{V}, \mb{W})$ cannot be regular. However its restriction to each stratum $\alpha$ is still cut out cleanly. The following lemma is similar to Fukaya--Ono's lemma (see \cite[Lemma 5]{Fukaya_Ono_integer} or \cite[Lemma 2.53]{Bai_Xu_integer}) in the context of equivariant polynomial maps. It also implies the $C^0$-density of the space of normally complex sections (Lemma \ref{lemma_c0_density}).

\begin{lemma}\label{lemma36}
There exists $d_0\geq 0$ (depending on $\mb{V}$ and $\mb{W}$) such that for any stratum $\alpha \in {\mc O}^V(V)$, there is a smooth function
\beqn
f_\alpha: V_\alpha^* \times W_{\beta(\alpha)}  \to  {\rm Poly}{}^{d_0}( \mb{V}, \mb{W})
\eeqn
such that $f_\alpha(v, w)(v) = w$.
\end{lemma}

\begin{proof}
We can assume that $V_0 = W_0 = 0$. Fix a stratum $V_\alpha \subset V$. For each stratum $V_\beta$ with $V_\alpha \nsubseteq V_\beta$ and each $v \in V_\alpha^*$, by the axiom of frontier, one must have $v \notin V_\beta$. Then there exists a linear function $l_{\beta, v}: V \to {\mb C}$ such that 
\begin{align*}
    &\ l_{\beta, v}|_{V_\beta} \equiv 0,\ l_{\beta, v}(v) = 1
    \end{align*}
This can be made smoothly dependent on $v \in V_\alpha^*$. Then we define
\beqn
f_v = \prod_{v \notin V_\beta^*} l_{\beta, v} \in {\mb C}[V]
\eeqn
which is equal to $1$ at $v$ and vanishes on all stratum $V_\beta$ which does not contain $V_\alpha$. Then define
\begin{align*}
&\ F_\alpha: V_\alpha^* \times W_{\beta(\alpha)} \to {\rm Poly}(V, W),\ &\ F_\alpha(v, w) = f_v w\in {\rm Poly}(V, W).
\end{align*}
We verify that $F_\alpha(v, w) \in {\rm Poly}( \mb{V}, \mb{W} )$, namely, $F_\alpha (v, w) (V_\gamma) \subset W_{\beta(\gamma)}$ for all strata $V_\gamma \subset V$. Given an arbitrary $\gamma$, if $v \in V_\gamma$, then as $v \in V_\alpha^*$, it follows that $\alpha \leq \gamma$, implying $\beta(\alpha) \leq \beta(\gamma)$. Then for $u \in V_\gamma$, one has $f_v(u) w \in W_{\beta(\alpha)} \subset W_{\beta(\gamma)}$. If $v \notin V_\gamma$, then $F_\alpha(v, w)|_{V_\gamma} = 0 \in W_{\beta(\gamma)}$. Lastly, the constructed polynomial map clearly has an upper bound on its degree. 
\end{proof}

\begin{cor}
When $d \geq d_0$, for all $\alpha$, $ Z_\alpha^d(\mb{V}, \mb{W} )$ is a smooth manifold.
\end{cor}

\begin{proof}
$Z_\alpha^d({\mb V}, {\mb W})$ is defined by the equation
\beqn
{\rm ev}(v, P) = 0 \in W_{\beta(\alpha)},\ v \in V_\alpha^*, P \in {\rm Poly}^d({\mb V}, {\mb W}).
\eeqn
Lemma \ref{lemma36} implies that all solutions are transverse once $d \geq d_0$.
\end{proof}

\subsubsection{Degree independence}

\begin{prop}\label{prop_degree_change}
There exists $d_1 \geq d_0$ such that for all $d' > d \geq d_1$, the natural inclusion map
\beqn
\iota:  M^d(\mb{V}, \mb{W})  \to  M^{d'}(\mb{V}, \mb{W})
\eeqn
pulls back the canonical Whitney stratification on $  Z^{d'}(\mb{V}, \mb{W})$ to the canonical one on $ Z^d(\mb{V}, \mb{W})$.
\end{prop}

\begin{proof}
This is a consequence of the fact that the space of stratified polynomial maps is a finitely generated module over the ring of polynomial functions (Lemma \ref{finite_generation}). 
Let $Q_1, \ldots, Q_m$ be a set of generators of ${\rm Poly}( \mb{V}, \mb{W})$ over ${\mb C}[V]$ and let $d_1$ be the maximum of their degrees. Now we can decompose
\beqn
{\rm Poly}^{d'}(\mb{V}, \mb{W} ) = {\rm Poly}^d( \mb{V}, \mb{W} ) \oplus {\rm Poly}^{(d, d']}( \mb{V}, \mb{W} )
\eeqn
where the second summand consists of stratified polynomial maps spanned over ${\mb C}$ by homogeneous ones of degrees between $d$ and $d'$. Choose a ${\mb C}$-basis $P_1, \ldots, P_n$ of ${\rm Poly}^{(d, d']}( {\mb V}, \mb{W})$. Then one can write
\beqn
P_i = \sum_{j=1}^m \rho_{ij} Q_j
\eeqn
where $\rho_{ij} \in  {\mb C}[V]$. Then for $P = P' + P'' \in {\rm Poly}^{d'}(\mb{V}, \mb{W})$ where $P' \in {\rm Poly}^d( {\mb V}, \mb{W} )$ and $P'' \in {\rm Poly}^{(d, d']}( \mb{V}, \mb{W} )$, write $P'' = \sum a_i P_i$ with $a_i \in {\mb C}$. Define 
\beq\label{eqn_sigma_map}
\sigma: (v, P' + P'') \mapsto \left( v, P' + \sum_{i, j} a_i \rho_{ij}(v) Q_j \right) \in  M^d(\mb{V}, \mb{W}).
\eeq
This map preserves the evaluation. Moreover, it is obvious that $\sigma$ is a left-inverse of the inclusion $\iota$ and $\sigma$ is a holomorphic submersion. 

The rest of the proof is similar to that of \cite[Proposition 3.14]{Bai_Xu_integer}. Let $ {\mf Z}^d$ resp. $ {\mf Z}^{d'}$ be the canonical Whitney stratification on $ Z^d(\mb{V}, \mb{W})$ resp. $ Z^{d'}(\mb{V}, \mb{W})$. Then by Proposition \ref{prop_submersion_pullback}, one has 
\beqn
\sigma^*  {\mf Z}^{d'} = {\mf Z}^d.
\eeqn
Then 
\beqn
 {\mf Z}^d = (\sigma \circ \iota)^* {\mf Z}^d =  \iota^* \sigma^*  {\mf Z}^d  = \iota^*  {\mf Z}^{d'}. \qedhere
\eeqn
\end{proof}

\subsubsection{Relaxing the constraints}

The inductive construction of transverse multisections is related to the following comparison of the canonical Whitney stratifications. Recall that given a complex stratified virtual vector space $({\mb V}, {\mb W})$, for each stratum $\alpha \in {\mc O}^V(V)$, there is an induced complex stratified vector space
\beqn
(\mb{V}_{\geq \alpha}, \mb{W})
\eeqn
by only remembering strata of $V$ which are higher than or equal to $\alpha$. Then there is a natural inclusion map
\beqn
\xi_\alpha: V \times {\rm Poly}{}^d(\mb{V}, \mb{W}) \to V \times {\rm Poly}{}^d(\mb{V}_{\geq \alpha}, \mb{W}).
\eeqn
Notice that 
\beqn
\xi_\alpha^{-1}( Z^d({\mb V}_{\geq \alpha}, {\mb W})) = Z^d({\mb V}, {\mb W}).
\eeqn
We would like to compare the canonical Whitney stratifications. 

\begin{prop}\label{prop39}
Consider the open subset 
\beqn
V_\alpha^+:= \bigcup_{\alpha \leq \beta} V_\beta^* \subset V.
\eeqn
$\xi_\alpha$ is transverse to ${\mf Z}^d({\mb V}_{\geq \alpha}, \mb{W})$ over $V_\alpha^+ \times {\rm Poly}^d({\mb V}, {\mb W})$ and 
\beqn
(\xi_\alpha)^* \Big( {\mf Z}^d ({\mb V}_{\geq \alpha}, \mb{W} )|_{V_\alpha^+ \times {\rm Poly}^d({\mb V}_{\geq \alpha}, {\mb W})} \Big) = {\mf Z}^d({\mb V}, \mb{W})|_{V_\alpha^+ \times {\rm Poly}^d({\mb V}, {\mb W})}.
\eeqn
\end{prop}

To prove this lemma, we need some preparations. The key is to construct a holomorphic submersive left inverse to $\xi_\alpha$. However, we can only do it locally. 

\begin{lemma}\label{lemma310}
There exists an open cover $\{ O_i \ |\ i = 1, \ldots, m\}$ of $V_\alpha^+$ such that for each $d\geq d_1$ where $d_1$ is the one of Proposition \ref{prop_degree_change}, there exist holomorphic maps 
\beqn
G_{\alpha, i}: O_i \times {\rm Poly}{}^d(\mb{V}_{\geq \alpha}, \mb{W} ) \to {\rm Poly}{}^d ( \mb{V}, \mb{W}),\ i = 1, \ldots, m
\eeqn
satisfying 
\beqn
G_{\alpha, i} (v, Q)(v) = Q(v) \ \  \forall v \in O_i\ {\rm and}\ Q \in {\rm Poly}{}^d(\mb{V}_{\geq \alpha}, \mb{W}  ).
\eeqn
\end{lemma}

\begin{proof}
For each $V_\beta$ which does not contain $V_\alpha$, one has $V_\alpha^+ \cap V_\beta = \emptyset$. Then for $v \in V_\alpha^+$, one can find a linear function $l_{\beta, v}: V \to {\mb C}$ such that $l_{\beta, v}(v) = 1$, $l_{\beta, v}|_{V_\beta} \equiv 0$. We would like to make $l_{\beta, v}$ depending holomorphically on $v\in V_\alpha^+$. In general this is impossible unless $V_\beta$ is a hyperplane. In general, if the codimension of $V_\beta$ is $m_\beta$, then one can find $m_\beta$ hyperplanes $H_{\beta, j}$, $j = 1, \ldots, m_\beta$, each containing $V_\beta$, and linear functionals $l_{\beta, j, v}$ satisfying the requirement and depending holomorphically on $v \in V \setminus H_{\beta, j}$. Moreover, the complements of $H_{\beta, j}$ cover $V_\alpha^+$. Then for any combination ${\bm j}:= (j_\beta)$ for all such $\beta$, one defines
\beqn
O_{{\bm j}} = V_\alpha^+ \setminus \bigcap_{V_\alpha \nsubseteq V_\beta} H_{\beta, j_\beta}.
\eeqn
These open sets cover $V_\alpha^+$. Then define
\beqn
\begin{split}
G_{{\bm j}}: O_{{\bm j}} \times {\rm Poly}{}^d(\mb{V}_{\geq \alpha}, \mb{W} ) \to &\ V \times {\rm Poly}{}^{d+d'}(V, W),\\
(v, Q) \mapsto &\ \left( v,\ \Big( \prod_{V_\alpha \nsubseteq V_\beta} l_{\beta, j_\beta, v}\Big) Q \right)
\end{split}
\eeqn
where $d'$ only depends on ${\mb V}$. Then $G_{{\bm j}}' (v, Q) \in {\rm Poly}^{d+ d'} ( \mb{V}, \mb{W} )$ and it has the same evaluation as $Q$ at $v$. Now recall that one has another map 
\beqn
\sigma: V \times {\rm Poly}{}^{d'}(\mb{V}, \mb{W}) \to {\rm Poly}{}^d(\mb{V}, \mb{W})
\eeqn
from the proof of Proposition \ref{prop_degree_change} (see \eqref{eqn_sigma_map}) which preserves the evaluation. Then the map
\beqn
G_{{\bm j}} = \sigma \circ G_{{\bm j}}': O_{{\bm j}} \times {\rm Poly}{}^d(\mb{V}_{\geq \alpha}, \mb{W} ) \to {\rm Poly}{}^d(\mb{V}, \mb{W})
\eeqn
satisfies the requirement.
\end{proof}

\begin{proof}[Proof of Proposition \ref{prop39}]
In view of Proposition \ref{prop31}, one only needs to compare the Whitney stratifications on $O_i$. Choose a complex-linear splitting
\beqn
{\rm Poly}{}^d(\mb{V}_{\geq\alpha}, \mb{W} ) \cong {\rm Poly}{}^d(\mb{V}, \mb{W}) \oplus H.
\eeqn
Write an element of ${\rm Poly}{}^d(\mb{V}_{\geq \alpha}, \mb{W} )$ as $P + P'$ with respect to this splitting. Define
\beqn
\eta_{\alpha, i}: O_i \times {\rm Poly}{}^d(\mb{V}_{\geq\alpha}, \mb{W} ) \to  O_i \times {\rm Poly}{}^d(\mb{V}, \mb{W})
\eeqn
by 
\beqn
\eta_{\alpha_i}(v, P + P') = (v, P + G_{\alpha, i}(v, P')).
\eeqn
This this is a left inverse to $\xi_\alpha$ and a holomorphic submersion. Lemma \ref{lemma310} shows that $\eta_{\alpha, i}$ preserves the evaluation. Hence by Proposition \ref{prop_submersion_pullback}, one has
\beqn
\eta_{\alpha, i}^* \Big(  {\mf Z}^d({\mb V}, \mb{W})|_{O_i \times {\rm Poly}^d({\mb V}, {\mb W})} \Big) =  {\mf Z}^d({\mb V}_{\geq \alpha}, \mb{W} )|_{O_i \times {\rm Poly}^d({\mb V}_{\geq \alpha}, {\mb W})}.
\eeqn
So 
\beqn
\xi_\alpha^* \Big( {\mf Z}^d({\mb V}_{\geq \alpha}, \mb{W})|_{O_i \times {\rm Poly}^d({\mb V}_{\geq \alpha}, {\mb W})} \Big) =  {\mf Z}^d({\mb V}, {\mb W})|_{O_i \cap {\rm Poly}^d({\mb V}, {\mb W})}.
\eeqn
As all $O_i$ are open and cover $V$, the claim is proved.
\end{proof}

\subsubsection{Splitting}

Now suppose ${\mb V}$ is equipped with a complex-linear splitting 
\beqn
V = V_0 \oplus \check V_0
\eeqn
where $V_0 \subset V$ is the lowest stratum. Then $\check V_0$ has an inherited stratification, denoted by $\check {\mb V}_0$. We write a vector of $V$ as $v = (v_0, \check v_0)$. Consider a map 
\beqn
\tau: V\times {\rm Poly}^d({\mb V}, {\mb W}) \to \check V_0 \times {\rm Poly}^d(\check {\mb V}_0, {\mb W})
\eeqn
defined by 
\beqn
\tau (v_0, \check v_0, P):= (\check v_0, P( v_0, \cdot))
\eeqn
where $P(v_0, \cdot) \in {\rm Poly}^d(\check {\mb V}_0, {\mb W})$ is a partial evaluation of $P$. Then we see
\beqn
\tau^{-1}(Z^d( \check {\mb V}_0, {\mb W})) = Z^d({\mb V}, {\mb W}).
\eeqn
Moreover, $\tau$ is a holomorphic submersion and preserves the linear stratifications on $V$ and $\check V_0$. Then by Proposition \ref{prop_submersion_pullback}, one has the following consequence.

\begin{prop}\label{prop311}
$\tau$ is transverse to ${\mf Z}^d( \check {\mb V}_0, \mb{W})$ and 
\beqn
\tau^* {\mf Z}^d( \check {\mb V}_0, \mb{W}) =  {\mf Z}^d(\mb{V}, \mb{W}).
\eeqn
\end{prop}

\subsubsection{Stabilization}

Let ${\mb U}$ be another complex stratified linear space together with a poset map ${\mc O}^V(V) \to {\mc O}^U(U)$ denoted by $\alpha \mapsto \gamma(\alpha)$. One can define the stabilization of $({\mb V}, {\mb W})$ by ${\mb U}$, which is the complex stratified virtual space
\beqn
({\mb V} \oplus {\mb U}, {\mb W} \oplus {\mb U})
\eeqn
where $V \oplus U$ is stratified by 
\beqn
(V \oplus U)_\alpha:= V_\alpha \oplus U_{\gamma(\alpha)}
\eeqn
and $W \oplus U$ is stratified by 
\beqn
(W \oplus U)_{(\beta, \gamma)} = W_\beta \oplus U_\gamma.
\eeqn
For $d \geq 1$, there is then a stabilization map
\beqn
\begin{split}
{\rm Poly}{}^d(\mb{V}, \mb{W}) \to &\  {\rm Poly}{}^d(\mb{V} \oplus \mb{U}, \mb{W} \oplus \mb{U})\\
P \mapsto &\ P \oplus {\rm Id}_U.
\end{split}
\eeqn
Together with the  inclusion $V \hookrightarrow V \oplus U$, the stabilization map induces a map 
\beqn
\eta: M^d(\mb{V}, \mb{W}) \to M^d(\mb{V}\oplus \mb{U}, \mb{W} \oplus \mb{U})
\eeqn
such that
\beqn
\eta^{-1} \left( Z^d(\mb{V} \oplus \mb{U}, \mb{W} \oplus \mb{U}) \right) = Z^d(\mb{V}, \mb{W}).
\eeqn

\begin{prop}\label{prop313}
When $d$ is sufficiently large, $\eta$ is transverse to the canonical Whitney stratification on $Z^d(\mb{V}\oplus \mb{U}, \mb{V} \oplus \mb{U})$ and 
\beqn
\eta^* {\mf Z}^d(\mb{V} \oplus \mb{U}, \mb{W} \oplus \mb{U}) =  {\mf Z}^d({\mb V}, \mb{W}).
\eeqn
\end{prop}

Again, we need to construct a left inverse of $\eta$. Following the idea of proving \cite[Proposition 3.17]{Bai_Xu_integer}, consider an intermediate space
\beqn
\widecheck{\rm Poly}{}^d(\mb{V}\oplus \mb{U}, \mb{W}\oplus \mb{U}) \subset  {\rm Poly}{}^d(\mb{V} \oplus \mb{U}, \mb{W}\oplus \mb{U})
\eeqn
consisting of stratified polynomial maps of the form 
\beqn
(v, u) \mapsto (P_0(v), u + Q_0(v))
\eeqn
where $P_0 \in  {\rm Poly}{}^d(\mb{V}, \mb{W})$ and $Q_0 \in  {\rm Poly}{}^d(\mb{V}, \mb{U})$. Denote
\beqn
\widecheck{M}^d:= (V \oplus U) \times \widecheck{\rm Poly}{}^d(\mb{V}\oplus \mb{U}, \mb{W} \oplus \mb{U})
\eeqn
and
\beqn
\widecheck{Z}^d:=\widecheck{Z}^d(\mb{V}\oplus \mb{U}, \mb{W} \oplus \mb{U}) \subset \widecheck{M}^d
\eeqn
be the corresponding $Z$-variety, which has a canonical Whitney stratification $\widecheck{\mf Z}^d$. We introduce several abbreviations. Denote
\begin{align*}
&\ \eta(  {M}^d(\mb{V}, \mb{W})) = \widetilde{M}^d,\ &\ \eta(  {Z}^d(\mb{V}, \mb{W})) = \widetilde{Z}^d.
\end{align*}
Let the canonical Whitney stratification on $\widetilde{Z}^d$ be $\widetilde{\mf Z}^d$, which is exactly the canonical Whitney stratification on $Z^d({\mb V}, {\mb W})$. Then $\eta$ is the composition of the two maps:
\begin{align*}
    &\ \eta_1: \widetilde{M}^d \to \widecheck{M}^d,\ &\ \eta_2: \widecheck{M}^d \to  {M}^d.
\end{align*}

\begin{lemma}\label{lemma311}
$\eta_1$ is transverse to $\widecheck{\mf Z}^d$ and $\eta_1^* \widecheck{\mf Z}^d \equiv \widetilde{\mf Z}^d$. 
\end{lemma}

\begin{lemma}\label{lemma312}
$\eta_2$ is transverse to $\wt{\mf Z}^d$ and $\eta_2^*  {\mf Z}^d \equiv \widecheck{\mf Z}^d$.
\end{lemma}

\begin{proof}[Proof of Proposition \ref{prop313}] As $\eta = \eta_2 \circ \eta_1$, Lemma \ref{lemma311} and Lemma \ref{lemma312} imply that $\eta$ is transverse to $ {\mf Z}^d$ and 
\beqn
\eta^*  {\mf Z}^d = (\eta_2 \circ \eta_1)^*  {\mf Z}^d = \eta_1^* \eta_2^*  {\mf Z}^d = \eta_1^* \widecheck{\mf Z}^d  =  \widetilde{\mf Z}^d. \qedhere
\eeqn    
\end{proof}

\begin{proof}[Proof of Lemma \ref{lemma311}]
Define
\beqn
\zeta_1: \widecheck{M}^d \to \widetilde{M}^d
\eeqn
by 
\beqn
(x, y, P_0, {\rm Id}_U + Q_0) \mapsto (x, y + Q_0(x), P_0, {\rm Id}_U).
\eeqn
It is a holomorphic submersion and a left inverse to $\eta_1$. 
Hence by Proposition \ref{prop_submersion_pullback}
\beqn
\zeta_1^* \widetilde{\mf Z}^d =  \widecheck{\mf Z}^d.
\eeqn
Moreover, 
\beqn
\widetilde{\mf Z}^d =  (\zeta_1 \circ \eta_1)^* \widetilde{\mf Z}^d = \eta_1^* \zeta_1^* \widetilde{\mf Z}^d = \eta_1^*  \widecheck{\mf Z}^d. \qedhere
\eeqn
\end{proof}

\begin{proof}[Proof of Lemma \ref{lemma312}]
We do not have an explicit formula for a left inverse of $\eta_2$. Instead, we consider the equation 
\beq\label{stabilization_equation}
y + Q(x, y) = 0\in U,\ x\in V,\ y \in U,\ Q \in {\rm Poly}^d(V \oplus U, U).
\eeq
Notice that if $Q = Q_0 \in {\rm Poly}{}^d(V, U)$, then $y = -Q_0 (x)$ is an obvious solution. Then by the implicit function theorem, there is an open neighborhood $\tilde O$ of 
\beqn
V \times {\rm Poly}{}^d(V, U) \subset V \times {\rm Poly}{}^d( V \oplus U, U)
\eeqn
such that for $(x, Q) \in \tilde O$, there is a unique solution $y = \tilde {\mc F}(x, Q)$ to the equation
\beqn
y + Q(x, y) = 0
\eeqn
such that $\tilde {\mc F}(x, Q_0) = -Q_0(x)$ for all $Q_0 \in {\rm Poly}^d(V, U)$. As the equation is holomorphic, $\tilde {\mc F}$ is a holomorphic map from $\tilde O$ to $U$. Denote
\beqn
O = \tilde O \cap \Big( V \times {\rm Poly}{}^d(\mb{V} \oplus \mb{U}, \mb{U} \Big)
\eeqn
and denote
\beqn
{\mc F}:= \tilde {\mc F}|_O: O \to U.
\eeqn
Notice that the stratification on $\mb{V}$ induces a stratification  on $O$ (where the factor ${\rm Poly}^d({\mb V}\oplus {\mb U}, {\mb U})$ is trivially stratified). We claim that ${\mc F}$ is stratified, namely  
\beqn
{\mc F} \in {\rm Hol}_{\rm stratified} (O, \mb{U}).
\eeqn
Indeed, if we restrict the equation \eqref{stabilization_equation} to $V_\alpha \times {\rm Poly}{}^d(V_\alpha \oplus V_\alpha, U_{\gamma(\alpha)})$, then one shall get a solution
\beqn
\tilde {\mc F}_\alpha: \Big( V_\alpha \times {\rm Poly}{}^d(V_\alpha \oplus U_{\gamma(\alpha)}, U_{\gamma(\alpha)}) \Big) \cap \tilde O \to U_{\gamma(\alpha)}.
\eeqn
By the uniqueness of the solution, we know that 
\beqn
{\mc F}|_{O_\alpha} = \tilde {\mc F}_\alpha|_{O_\alpha}.
\eeqn
Hence ${\mc F}$ is stratified.

Now by Lemma \ref{lemma_holomorphic_generation}, the module ${\rm Hol}_{\rm stratified} (O, \mb{U})$ is finitely generated over the ring of holomorphic functions by ${\rm Poly}{}(\mb{V} \times {\rm Poly}{}^d (\mb{V}\oplus \mb{U}, \mb{U}), \mb{U})$. Then choose a finite collection of generators 
\beqn
F_1, \ldots, F_m \in {\rm Poly}{} \Big(\mb{V} \times {\rm Poly}{}^d (\mb{V} \oplus \mb{U}, \mb{U}), \mb{U} \Big).
\eeqn
We can write
\beqn
{\mc F} = \sum_{i=1}^m h_i F_i
\eeqn
for holomorphic functions $h_i: O \to {\mb C}$. Define
\beqn
\begin{split}
\zeta_2': O \times {\rm Poly}^d(\mb{V} \oplus \mb{U}, \mb{V}) \to &\ \widecheck{M}^{d+d'}\\
 (x, y, P, {\rm Id}_U + Q) \mapsto &\  \Big( x, y, \widecheck{P}, {\rm Id}_U + \widecheck{Q} \Big) 
\end{split}
\eeqn
$d'$ is sufficiently large and where
\begin{align*}
&\ \widecheck{P} = P \left( \cdot, \sum_{i=1}^m h_i(x, Q) F_i(\cdot, Q) \right),\ &\ \widecheck Q = Q \left( \cdot, \sum_{i=1}^m h_i(x, Q) F_i (\cdot, Q) \right).
\end{align*}
Notice that since $P$, $Q$ have degrees at most $d$ and $F_1, \ldots, F_m$ has bounded degrees, for sufficiently large $d'$, the map $\zeta_2'$ is well-defined. Moreover, it is straightforward to check that $\zeta_2'$ is a left inverse of $\eta_2$. Moreover, choose a holomorphic submersion
\beqn
\widecheck \sigma: \widecheck M^{d+d'} \to \widecheck M^d
\eeqn
similar to the construction in the proof of Proposition \ref{prop_degree_change} which is a left inverse to the inclusion
\beqn
\widecheck\iota: \widecheck M^d \to \widecheck M^{d+d'}.
\eeqn
Define
\beqn
\zeta_2:= \widecheck \sigma \circ \zeta_2'.
\eeqn
This is a left inverse to $\eta_2$ and again a holomorphic submersion. Hence following the same method as proving Proposition \ref{prop_degree_change} and Lemma \ref{lemma310}, the conclusion holds.
\end{proof}

\subsubsection{Products}

The following property of the canonical Whitney stratification is not immediately used in this paper. However, it is necessary to have when one considers direct products of moduli spaces. Notice that complex stratified virtual spaces admit direct sums. Let $({\mb V}_1, {\mb W}_1)$ and $({\mb V}_2, {\mb W}_2)$ be two complex stratified virtual vector spaces. For $d\geq 0$, there is a natural inclusion
\beqn
M^d(\mb{V}_1, \mb{W}_1) \times M^d(\mb{V}_2,\mb{W}_2) \hookrightarrow M^d(\mb{V}_1 \oplus \mb{V}_2, \mb{W}_1 \oplus \mb{W}_2).
\eeqn

\begin{prop}\label{prop_product_Whitney}
For $d$ sufficiently large, the inclusion pulls back the canonical Whitney stratification on $ Z^d(\mb{V}_1 \oplus \mb{V}_2, \mb{W}_1 \oplus \mb{W}_2)$ to the product of the canonical ones. 
\end{prop}

\begin{proof}
We needs to construct a left inverse. We can write a polynomial from $V_1 \oplus V_2$ to $W_1 \oplus W_2$ as $P_1(v_1, v_2) + P_2(v_1, v_2)$ where $v_i \in V_i$ and $P_i$ takes value in $W_i$. Define
\beqn
\mu(v_1, v_2, P_1, P_2) = (v_1, P_1(\cdot, v_2); v_2, P_2(v_1, \cdot)) \in M^d(\mb{V}_1 , \mb{W}_1) \times  M^d(\mb{V}_2, \mb{W}_2).
\eeqn
This is clearly a left inverse to the inclusion, preserves the stratification, and preserves the evaluation. Moreover, $\mu$ is a holomorphic submersion. Hence in the same way as before, we showed that the inclusion pulls back the canonical Whitney stratification to the canonical one on the product. Further, it was proved in \cite[Appendix B]{Bai_Xu_integer} that the canonical Whitney stratification on the product coincides with the product of canonical Whitney stratifications of factors. 
\end{proof}

\subsection{NCS Transversality}

\subsubsection{Transversality for bundle maps}

We first see the notion of strong transversality for local models. Let $M$ be a smooth manifold and let $(F, E)$ be a complex stratified virtual vector bundle over $M$ (Definition \ref{defn_NCS_bundle}) . Recall that there is a $C^\infty(F)$-submodule of bundle maps 
\beqn
C_{\rm NCS}^\infty (F, E) \subset C^\infty_{\rm stratified} (F, E).
\eeqn

\begin{defn}\label{defn_bundle_transversality}
Let $S\in C_{\rm NCS}^\infty (F, E)$ be an NCS bundle map.
\begin{enumerate}

\item A {\bf lift} of $S$ is a (nonlinear) bundle map
\beqn
{\mf p}: F \to {\rm Poly}^d_{\rm stratified} (F, E)
\eeqn
for some integer $d \geq 0$ such that for all $(x, v) \in F$, 
\beqn
S(x, v) = {\mf p}(x, v)(v) \in E_x.
\eeqn

\item $S$ is said to be {\bf NCS transverse} at $(x, v) \in F$ if for any lift ${\mf p}$ of $S$, its graph ${\rm graph}({\mf p})$ intersect transversely with $Z^d(F, E)$ (with respect to the canonical Whitney stratification) at the point $((x, v), {\mf p}(x,x)) \in M^d(F, E)$.
\end{enumerate}
\end{defn}

\begin{lemma}
The NCS transversality is independent of the choice of lift. 
\end{lemma}

\begin{proof}
The main idea was due to B. Parker \cite{BParker_integer} in the context of FOP transversality condition (see \cite[Lemma 3.35]{Bai_Xu_integer}). First, Proposition \ref{prop_degree_change} implies that, for the same lift ${\mf p}: F \to {\rm Poly}{}^d_{\rm stratified} (F, E)$, its transversality is independent of sufficiently large $d$. Then, suppose 
\beqn
{\mf p}_1, \mf{p}_2: F \to  {\rm Poly}{}_{\rm stratified}^d(F, E)
\eeqn
are two lifts of the same bundle map, one can define a 1-parameter family of fibre-preserving self-diffeomorphisms $\Phi_t$ of $M^d(F, E)$ given by
\beqn
\Phi_t(v, P)  = \Big( v, P + t({\mf p}_2(v) - {\mf p}_1(v))\Big)
\eeqn
with $\Phi_0 = {\rm Id}$ and $\Phi_1$ sends ${\rm graph}(\mf{p}_1)$ to ${\rm graph}(\mf{p}_2)$. Moreover, $\Phi_t$ preserves the bundle $ {Z}^d(F, E)$ as well as the linear stratification on $  M^d(F,E)$. Then by Proposition \ref{prop31}, $\Phi_t$ preserves the canonical stratification on $  Z^d(F , E)$. As ${\rm graph}({\mf p}_2) = \Phi_1( {\rm graph}({\mf p}_1))$, the transversality condition for ${\mf p}_1$ and ${\mf p}_2$ are the same. 
\end{proof}

\begin{lemma}
The NCS transversality condition is open.
\end{lemma}

\begin{proof}
This follows from the fact that transversality against a Whitney stratified subset is an open condition. See \cite{Trotman_1978} discussions of the openness of the transversality condition.
\end{proof}

The inductive procedure depends on the following property of the NCS transversality condition. Still let $(F, E)$ be a complex stratified virtual vector bundle over $M$. As the consideration is local, we may assume that $F = M \times {\mb V}$ and $E = M \times {\mb W}$ where $({\mb V}, {\mb W})$ is a complex stratified virtual space. 
Then for each stratum $V_\alpha \subset V$, there is a well-defined virtual space $({\mb V}_{\geq \alpha}, {\mb W})$. Denote 
\beqn
F_{\geq \alpha} = M \times {\mb V}_{\geq \alpha}.
\eeqn
Then $(F_{\geq \alpha}, E)$ is a stratified virtual vector bundle over $M$. The map $\xi_\alpha$ of Proposition \ref{prop39} induces a map 
\beqn
\xi_\alpha: C_{\rm NCS}^\infty (F, E) \to C_{\rm NCS}^\infty (F_{\geq\alpha}, E).
\eeqn

\begin{lemma}
Let $F_\alpha^+ \subset F$ be the open subbundle of vectors in strata above or equal to $\alpha$. Let $S \in C^\infty_{\rm NCS}(F, E)$ and $x \in F_\alpha^+$. Then $S$ is NCS transverse at $x$ if and only if $ \xi_\alpha(S) \in C^\infty_{\rm NCS}(F_{\geq \alpha}, E )$ is NCS transverse at $x$.
\end{lemma}

\begin{proof}
This is a consequence of Proposition \ref{prop39}.
\end{proof}

In addition, suppose ${\mb V}$ has a splitting. Denote $F_\alpha^\vee = M \times V_\alpha^\vee$. Then one can identify the total spaces
\beqn
F_{\geq\alpha}  \cong \pi_{F_\alpha}^* F_\alpha^\vee
\eeqn
where the latter is a bundle over the total space $F_\alpha$. Then there is a natural map
\beqn
C^\infty_{\rm NCS} (F_{\geq \alpha}, E) \to C^\infty_{\rm NCS} (\pi_{F_\alpha}^* F_\alpha^\vee, \pi_{F_\alpha}^* E).
\eeqn

\begin{lemma}
A map $S \in C^\infty_{\rm NCS}(F_{\geq\alpha}, E)$ is NCS transverse at $x \in F_\alpha^+$ if and only if it is NCS transverse as an NCS bundle map from $\pi_{F_\alpha}^* F_\alpha^\vee$ to $\pi_{F_\alpha}^* E$.
\end{lemma}

\begin{proof}
This is a consequence of Proposition \ref{prop311}.
\end{proof}

We prove the existence and extension result of NCS transverse bundle maps.

\begin{prop}\label{bundle_transversality}
Let $S_0 \in C^\infty_{\rm NCS}(F, E)$ be NCS transverse near a closed subset $Y \subset F$. Then for any $\epsilon>0$, there exists an NCS transverse map $S_1 \in C_{\rm NCS}^\infty (F, E)$ such that $\| S_0 - S_1 \|_{C^0} \leq \epsilon$ and such that $S_0$ and $S_1$ coincides near $Y$.
\end{prop}

\begin{proof}
One only needs to modify a lift ${\mf p}_0: F \to  {\rm Poly}{}^d_{\rm stratified} (F, E )$ locally where the bundles can be trivialized. In general, for smooth maps $f: U \to M$ between smooth manifolds and a smooth submanifold $Z \subset U \times M$, the transversality between the graph of $f$ and $Z$ can be achieved by small perturbations of $f$ and one can require the perturbation to be supported away from a closed subset near which $f$ is already transverse.
\end{proof}

Lastly we prove the invariance under stabilizations.

\begin{prop}\label{prop_stabilization_transversality}
Let $H \to M$ be another complex stratified vector bundle and $(F, H)$ is a stratified virtual vector bundle over $M$. Suppose $S \in C_{\rm NCS}^\infty (F, E)$ is NCS transverse. Then its stabilization
\beqn
S \oplus {\rm Id}_H \in C_{\rm NCS}^\infty ( F \oplus H, E \oplus H )
\eeqn
is NCS transverse.
\end{prop}

\begin{proof}
This is a corollary of Proposition \ref{prop313}. 
\end{proof}

\subsubsection{Transversality on manifolds}

Now we consider the case over manifolds.

\begin{defn}\label{defn_manifold_transversality}
Let $(U, E)$ be an NCS virtual manifold (Definition \ref{defn_NCS_virtual_manifold}). Suppose $E$ is flat and $U$ is equipped with a straightening (Definition \ref{defn_manifold_straightening}). A stratified section $S \in \Gamma_{\rm stratified} (U, E )$ is said to be {\bf NCS transverse} with respect to this straightening if, for each stratum $\alpha$, the bundle map
\beqn
S_\alpha : N^\epsilon U_\alpha^* \to E|_{U_\alpha}^*
\eeqn
induced by the straightening and parallel transport of $E$ along normal directions, is the restriction of a NCS transverse bundle map from $NU_\alpha^*$ to $E|_{U_\alpha^*}$ in the sense of Definition \ref{defn_bundle_transversality}.
\end{defn}

Notice that the NCS transversality condition depends on the straightening on $U$. In fact the bundle map $S_\alpha$ being a normally complex bundle map depends on the tubular neighborhood.

An important consequence of NCS transversality is that the zero locus is nicely stratified. One has the following characterization.

\begin{lemma}\label{lemma_zero_stratification}
Let $S\in \Gamma_{\rm stratified} (U, E)$ be an NCS transverse section. Then $S^{-1}(0) \subset U$ is a Whitney stratified set such that adjacent strata have even (real) codimensions.
\end{lemma}

\begin{proof}
We only needs to verify it locally. For $x \in S^{-1}(0) \cap U_\alpha^*$, using the straightening, locally $S$ can be lifted to a map 
\beqn
{\mf p}: U_\alpha^* \times V \to  {\rm Poly}{}^d({\mb V}, {\mb W})
\eeqn
where $({\mb V}, {\mb W})$ is a complex stratified virtual space. Then 
\beqn
S^{-1}(0) = {\rm graph}({\mf p}) \cap (U_\alpha^* \times  Z^d({\mb V}, {\mb W})).
\eeqn
As $ Z^d({\mb V}, {\mb W})$ is Whitney stratified such that adjacent strata have even codimensions and the intersection is transverse, the same property holds for $S^{-1}(0)$.
\end{proof}

\begin{prop}\label{prop_NCS_extension_manifold}
Let $S_0 \in \Gamma_{\rm stratified} ( U, E)$ be NCS transverse near a closed subset $Y \subset U$. Then for any $\epsilon>0$, there exists an NCS transverse section $S_1\in \Gamma_{\rm stratified} (U, E)$ such that $\| S_0 - S_1 \|_{C^0} \leq \epsilon$ and such that $S_1 = S_0$ near $Y$.
\end{prop}

\begin{proof}
The proof is obtained by combining Proposition \ref{bundle_transversality} and an induction argument. First consider a lowest stratum $U_\alpha \subset U$. Using the straightening, one can write $S_0$ as a stratified bundle map 
\beqn
S_{0, \alpha}: N^\epsilon U_\alpha \to E|_{U_\alpha}.
\eeqn
By Proposition \ref{bundle_transversality}, one can slightly perturb $S_{0, \alpha}$ to obtain an NCS transverse bundle map which agrees with $S_{0, \alpha}$ outside a tubular neighborhood. Then one can reduce the problem to the complement $U\setminus U_\alpha$ which has one fewer stratum and such that the NCS transversality is already achieved near a closed subset. Then by applying Proposition \ref{bundle_transversality} inductively, a global NCS transverse section $S_1$ can be obtained. One can control the size of the perturbation in each step so that the $C^0$ distance from $S_0$ to $S_1$ is no greater than any given $\epsilon$.

\end{proof}

\subsubsection{Transversality on orbifolds}

It is then easy to generalize the NCS transversality to stratified orbifolds. As equivariant transversality may fail, one needs to use multisections to break local symmetry. 

\begin{defn}
Let $({\mc U}, {\mc E})$ be an NCS virtual orbifold (Definition \ref{defn_NCS_virtual_orbifold}) equipped with a straightening. A multisection ${\mc S} \in \Gamma_{\rm stratified}^{\rm multi}({\mc U}, {\mc E})$ is said to be {\bf NCS transverse} with respect to the straightening if for each bundle chart $\hat C = (G, U, E, \hat\phi)$ on which ${\mc S}$ is lifted to sections
\beqn
S_1, \ldots, S_k \in \Gamma_{\rm stratified}(U, E)
\eeqn
where $U$ has the induced straightenings, each $S_i$ is NCS transverse in the sense of Definition \ref{defn_manifold_transversality}.
\end{defn}

\begin{thm}\label{thm_transversality_orbifold}
Let $({\mc U}, {\mc E})$ be as above.
\begin{enumerate}
    \item Let ${\mc S}_1 \in \Gamma_{\rm stratified}^{\rm multi}({\mc U}, {\mc E})$ be NCS transverse near a closed subset $Y  \subset {\mc U}$ and is $\epsilon$-close to a single-valued stratified section ${\mc S}_0 \in \Gamma_{\rm stratified}({\mc U}, {\mc E})$. Then there exists an NCS transverse multisection ${\mc S}_2 \in \Gamma_{\rm stratified}^{\rm multi}({\mc U}, {\mc E})$ which is $2\epsilon$-close to ${\mc S}_0$ and which agrees with ${\mc S}_1$ near $Y$.

    \item Let ${\mc S}_1, {\mc S}_2 \in \Gamma_{\rm stratified}^{\rm multi}({\mc U}, {\mc E})$ be two NCS transverse multisections. Then there exists an NCS transverse multisection
    $\tilde {\mc S} \in \Gamma_{\rm stratified}^{\rm multi}({\mc U} \times [0,1], {\mc E} \times [0,1])$ whose boundary restrictions coincide with ${\mc S}_1$ and ${\mc S}_2$. Moreover, if ${\mc S}_1^{-1}(0)$ and ${\mc S}_2^{-1}(0)$ are compact, we can choose $\tilde {\mc S}$ such that $\tilde {\mc S}^{-1}(0)$ is compact.
\end{enumerate}
\end{thm}

\begin{proof}
For (1), one can cover ${\mc U}$ by countably many bundle charts (which are locally finite) and inductively apply Proposition \ref{prop_NCS_extension_manifold} on the charts. For (2), one can first arbitrarily extend ${\mc S}_1$ and ${\mc S}_2$ to ${\mc U} \times [0, 1]$ as a normally complex section. As the NCS transversality condition is open, the extension is NCS transverse near the boundary. Hence one can apply (1) in this case. 
\end{proof}

\section{Kuranishi Atlases and Virtual Fundamental Classes}\label{section4}

In this section we set up the abstract theory for Kuranishi atlases. For the purpose of our application, we will first define a category of ``equivariant Kuranishi charts with flat Hermitian obstruction bundles'' 
\beqn
\uds{\bf Kur}^{G, \flat}({\mf M}).
\eeqn
Here ${\mf M}$ is typically a moduli space and $G$ is a compact Lie group. We develop a framework of equivariant Kuranishi atlases, where charts resp. coordinate changes are objects resp. morphisms in $\uds{\bf Kur}^{G,\flat}({\mf M})$. For applying the NCS perturbation method, we promote the category to 
\beqn
\uds{\bf Kur}_{\rm NCS}^{G, \flat}({\mf M})
\eeqn
where extra structures are included. We will then demonstrate how to use NCS transverse multivalued perturbations to define a virtual fundamental class of the ``main stratum'' of such atlases.

\subsection{Equivariant Kuranishi atlases with flat Hermitian obstruction bundles}

\begin{defn}\label{defn_equivariant_Kuranishi_chart}
Let $G$ be a compact Lie group and ${\mf M}$ be a locally compact and Hausdorff space.
\begin{enumerate}

\item A {\bf $G$-equivariant Kuranishi chart with flat Hermitian obstruction bundle} on ${\mf M}$ (a $\uds{\bf Kur}^{G,\flat}$-chart) is a quadruple
\beqn
K = (U, E, S, \Psi)
\eeqn
where $U$ is a smooth $G$-manifold (possibly with boundary and  corner), $E \to U$ is a $G$-equivariant flat Hermitian vector bundle, $S: U \to E$ is a continuous section, and $\Psi: S^{-1}(0)/G \to {\mf M}$ is a homeomorphism onto an open subset (the {\bf footprint} of $K$).

\item A {\bf morphism} of $\uds{\bf Kur}^{G, \flat}$-charts on ${\mf M}$ from $K_1 = (U_1, E_1, S_1, \Psi_1)$ to $K_2 = (U_2, E_2, S_2, \Psi_2)$, denoted by ${\bm \Phi}_{21} = (\Phi_{21}, \hat \Phi_{21})$, consists of a $G$-equivariant smooth map $\Phi_{21}: U_1 \to U_2$, a $G$-equivariant flat isometric bundle map $\hat \Phi_{21}: E_1 \to E_2$ which covers $\Phi_{21}$ 
such that the following diagrams commute.
\begin{align*}
&\ \xymatrix{  E_1 \ar[r]^{\Phi_{21}^E} & E_2 \\
            U_1 \ar[r]_{\Phi_{21}^U} \ar[u]^{S_1}  & U_2  \ar[u]_{S_2} }\ &\ \xymatrix{ S_1^{-1}(0) \ar[r]^{\Phi_{21}^U} \ar[d]_{\Psi_1}  & S_2^{-1}(0) \ar[d]^{\Psi_2} \\
                {\mf M} \ar@{=}[r] & {\mf M} }
            \end{align*}
Notice that morphisms can be composed.

\item The {\bf stabilization} of a $\uds{\bf Kur}^{G, \flat}$-chart $K = (U, E, S, \Psi)$ on ${\mf M}$ by a $G$-equivariant disk bundle $N^\epsilon$ contained in a $G$-equivariant flat Hermitian vector bundle $\pi_N: N \to U$ is the $\uds{\bf Kur}^{G, \flat}$-chart
\beqn
\tilde K = (N^\epsilon, \pi_N^* E \oplus \pi_N^* N, \pi_N^* S \oplus \tau_N, \tilde \Psi)
\eeqn
where $\tau_N$ is the tautological section of $\pi_N^* N \to N$ and $\tilde\Psi$ is the composition
\beqn
\xymatrix{ (\pi_N^* S \oplus \tau_N)^{-1}(0)/G \ar[r] & S^{-1}(0)/G  \ar[r]^-{\Psi} & {\mf M}}.
\eeqn
Notice that the bundle $\pi_N^* E \oplus \pi_N^* N \to N^\epsilon$ is canonically flat. Also notice that there is a canonical morphism $K \to \tilde K$. 


\item An {\bf open embedding} of $\uds{\bf Kur}^{G, \flat}$-charts from $K_1$ to $K_2$ is a morphism ${\bm \Phi}_{21} = (\Phi_{21}, \hat\Phi_{21})$ where $\Phi_{21}$ is a smooth open embedding and $\hat \Phi_{21}$ is an isomorphism of flat Hermitian vector bundles.
            

\item An {\bf embedding} of K-charts from $K_1 = (U_1, E_1, S_1, \Psi_1)$ to $K_2 = (U_2, E_2, S_2, \Psi_2)$ is a morphism ${\bm \Phi}_{21} = (\Phi_{21}, \hat\Phi_{21}): K_1 \to K_2$ such that there exists a disk bundle $N_{21}^\epsilon \to U_1$ contained in a $G$-equivariant flat Hermitian vector bundle $\pi_{N_{21}}: N_{21} \to U_1$ and an open embedding
\beqn
{\bm \Phi}_{21}^\epsilon = (\Phi_{21}^\epsilon, \hat \Phi_{21}^\epsilon): {\rm Stab}_{N_{21}^\epsilon}(K_1) \to K_2
\eeqn
which extends ${\bm \Phi}_{21}$. 
The pair $(N_{21}^\epsilon, {\bm \Phi}_{21}^\epsilon)$ is called a {\bf flat tubular neighborhood} of ${\bm \Phi}_{21}$. 

\item Two flat tubular neighborhoods $(N_{21}^\epsilon, {\bm \Phi}_{21}^\epsilon)$ and $(N_{21}^{'\epsilon}, {\bm \Phi}_{21}^{'\epsilon})$ are  called {\bf equivalent} if there exists a $G$-equivariant flat Hermitian isomorphism $N_{21} \cong N_{21}'$ such that (after shrinking $N_{21}^\epsilon$ to a smaller disk bundle) the diagram 
\beqn
\xymatrix{ {\rm Stab}_{N_{21}^\epsilon} (K_1) \ar[r]^-{{\bm \Phi}_{21}^\epsilon} \ar[d]  &  K_2 \ar[d]^{\rm Id}\\
        {\rm Stab}_{N_{21}^{'\epsilon}} (K_1) \ar[r]_-{{\bm \Phi}_{21}^{'\epsilon}}   &  K_2 }
        \eeqn
        commutes. A {\bf combed embedding} from $K_1$ to $K_2$ is consists of an embedding and an equivalence class of flat tubular neighborhoods.
\end{enumerate}
\end{defn}

\begin{prop}
Combed embeddings of $\uds{\bf Kur}^{G, \flat}$-charts on ${\mf M}$ can be composed. Hence there is a category with objects being $\uds{\bf Kur}^{G, \flat}$-charts on ${\mf M}$ and morphisms being combed embeddings, denoted by 
\beqn
\uds{\bf Kur}^{G, \flat}({\mf M}).
\eeqn
\end{prop}

\begin{proof}
The key point is that, for a flat bundle $F$ over the total space of another vector bundle $E \to B$, the total space of $F \to E$ can be canonically identified with the total space of $E \oplus (F|_B)$ via parallel transport of $F$ along fibres of $E$. 
\end{proof}


Now we define the more general notion of coordinate changes. 

\begin{defn}
Let $K_1, K_2$ be $\uds{\bf Kur}^{G, \flat}$-charts on ${\mf M}$. A {\bf coordinate change} from $K_1$ to $K_2$ is a pair $(U_{21}, {\bm \Phi}_{21})$ where $U_{21} \subset U_1$ is a $G$-invariant open subset and ${\bm \Phi}_{21}: K_1|_{U_{21}} \to K_2$ is a combed embedding, such that 
\beqn
{\rm Im}\Psi_1 \cap {\rm Im}\Psi_2 = \Psi_1( U_{21} \cap S_1^{-1}(0)).
\eeqn
\end{defn}

The composition of combed coordinate changes is defined as follows. Let $K_i$, $i=1,2,3$ be $G$-equivariant K-charts on ${\mf M}$ and $(U_{21}, {\bm \Phi}_{21})$ is a coordinate change from $K_1$ to $K_2$, and $(U_{32}, {\bm \Phi}_{32})$ is a coordinate change from $K_2$ to $K_3$. Define
\beqn
U_{321}:= \Phi_{21}^{-1}(U_{32}) \cap U_{31} \subset U_{21} \cap U_{31}
\eeqn
(which might be an empty set). Then the composition ${\bm \Phi}_{32} \circ {\bm \Phi}_{21}$ is the composition of the combed embeddings
\beqn
\xymatrix{ K_1|_{U_{321}} \ar[r]^-{{\bm \Phi}_{21}} &  K_2|_{U_{32}} \ar[r]^-{{\bm \Phi}_{32}} & K_3   }.
\eeqn

Now we define the $G$-equivariant versions of the notions of Kuranishi atlas. Our concept combines the notion of Kuranishi atlas (see \cite{MW_3}) and good coordinate systems (see \cite{Fukaya_Ono}\cite{FOOO_shrinking}\cite{FOOO_Kuranishi}).  

\begin{defn}\label{defn_Kuranishi_atlas}
Let ${\mf M}$ be a topological space. A {\bf $\uds{\bf Kur}^{G, \flat}$-atlas} on ${\mf M}$, denoted by ${\mf A}$, consists of a finite collection of objects
\beqn
K_I = (U_I, E_I, S_I, \Psi_I) \in {\rm Ob}\uds{\bf Kur}^{G, \flat}({\mf M})
\eeqn
indexed by elements of a partially ordered set ${\mc I}$ and coordinate changes ${\bm \Phi}_{JI}: K_J \to K_I$ for each pair $I \leq J$ satisfying the following conditions.
\begin{enumerate}
\item {\bf (Covering)} One has ${\mf M} = \bigcup_{I \in {\mc I}} {\rm Im}(\Psi_I)$.

\item {\bf (Cocycle)} ${\bm \Phi}_{II}$ is the identity and for each triple $I \leq J \leq K$, 
\beqn
{\bm \Phi}_{KI} = {\bm \Phi}_{KJ} \circ {\bm \Phi}_{JI}
\eeqn
on their common domain as morphisms of $\uds{\bf Kur}^{G, \flat}({\mf M})$.

\item {\bf (Virtual Neighborhood)} On the disjoint union $\bigsqcup_I U_I$  define the relation
\beq\label{chart_relation}
U_I \ni x \curlyvee y \in U_J \Longleftrightarrow I\leq J,\ y = \Phi_{JI}(x)\ {\rm or}\ J \leq I,\ x = \Phi_{IJ}(y).
\eeq
Then $\curlyvee$ is an equivalence relation. Moreover, the quotient space
\beqn
|{\mf A}|:= \left( \bigsqcup_I U_I \right)/ \curlyvee
\eeqn
with the quotient topology is Hausdorff\footnote{Notice that $|{\mf A}|$ is in generally not locally compact.}; for each $I$, the  inclusion
\beqn
\iota_I: U_I \to |{\mf A}|
\eeqn
is a homeomorphism onto its image.
\end{enumerate}
\end{defn}

Kuranishi atlases do not have products. However, if ${\mf A}$ is a $\uds{\bf Kur}^{G, \flat}$-atlas on ${\mf M}$, then there is an obvious product with the interval ${\mf A}\times [0, 1]$ which is a $\uds{\bf Kur}^{G, \flat}$-atlas on ${\mf M}\times [0, 1]$ whose charts are indexed by the same poset. We often needs this product with an interval to compare different choices.

\subsubsection{Subcategories and shrinkings}

To construct perturbations inductively on all charts, one typically needs to shrink the charts. The basic reason is that transverse perturbations can only be extended from closed sets, while the embedding image $\Phi_{JI}(U_{JI}) \subset U_J$ for a coordinate change ${\bm \Phi}_{JI}$ is typically not closed. To formulate such considerations, we introduce a more general notion of subcategories of Kuranishi atlases. 

As did in \cite[Definition 2.3.6]{McDuff_Wehrheim}, we view an atlas ${\mf A}$ as a category whose object set and morphism set are
\begin{align*}
&\ {\rm Ob} {\mf A} = \bigsqcup_I U_I,\ &\ {\rm Mor}{\mf A} = \bigsqcup_{I \leq J} U_{JI}.
\end{align*}
One can also consider the quotient category ${\mf A}/G$ whose object set is the $G$-quotient of ${\rm Ob}{\mf A}$ with induced morphisms. 

\begin{defn}
Let ${\mf A}$ be a $\uds{\bf Kur}^{G, \flat}$-atlas on ${\mf M}$.
\begin{enumerate}

\item A {\bf subcategory} of ${\mf A}$ is a $G$-invariant full subcategory ${\mf V}$ of ${\mf A}$ with 
\beqn
{\rm Ob} {\mf V} = \bigsqcup_{I \in {\mc I}} V_I
\eeqn
where $V_I \subset U_I$ is a $G$-invariant subset. 
Denote the subcategory by ${\mf V} = (V_I)_{I\in {\mc I}} \subset {\mf A}$. The {\bf closure} of a subcategory ${\mf V} = (V_I)_{I \in {\mc I}} \subset {\mf A}$ is the subcategory $\ov{\mf V} = (\ov{V_I})_{I \in{\mc I}}$.

\item The zero locus of the Kuranishi sections form a subcategory
\beqn
{\mf S}^{-1}(0) = (S_I^{-1}(0))_{I \in {\mc I}} \subset {\mf A}.
\eeqn

\item The {\bf realization} of a subcategory ${\mf V}$ is the quotient
\beqn
|{\mf V}|:= \left( \bigsqcup_{I \in {\mc I}} V_I \right)/\curlyvee
\eeqn
equipped with the quotient topology. The realization of the quotient category ${\mf V}/G$ is the quotient $|{\mf V}/G|:=|{\mf V}|/G$, whose topology agrees with the quotient topology of ${\rm Ob}({\mf V}/G)$ under coordinate change.

\item A {\bf shrinking}\footnote{Our notion of shrinking is different from that of \cite{FOOO_shrinking} which allows the subcategory not being a full subcategory. Our definition is simplified because we already have a ``good coordinate system.''} of ${\mf A}$ is a precompact open  subcategory ${\mf A}'$ of ${\mf A}$ which contains ${\mf S}^{-1}(0)$.
\end{enumerate}
\end{defn}

\begin{lemma}
When ${\mf M}$ is compact, ${\mf A}$ admits a shrinking. 
\end{lemma}

\begin{proof}
Denote $F_I = \Psi_I( S_I^{-1}(0))$, which gives an open cover of ${\mf M}$. As ${\mf M}$ is compact and Hausdorff, it is a normal space (separating closed sets). Hence one can find precompact open subsets $F_I' \sqsubset F_I$ which still cover ${\mf M}$. Then inside $U_I$ one can find precompact $G$-invariant open subsets $U_I' \sqsubset U_I$ such that $U_I' \cap S_I^{-1}(0) = \Psi_I^{-1}(F_I')$. 
\end{proof}

It is easy to see that a shrinking ${\mf A}' = (U_I')_{I \in {\mc I}} \subset {\mf A}$ is automatically a $\uds{\bf Kur}^{G, \flat}$-atlas on ${\mf M}$ where the coordinate changes are the restrictions of ${\bm \Phi}_{JI}$ on 
\beqn
U_{JI}':= U_{JI} \cap U_I' \cap \Phi_{JI}^{-1}(U_J').
\eeqn
Notice that its closure $\ov{U_{JI}'}$ in $U_{JI}$ is compact and contained in 
\beqn
\ov{\ov{U_{JI}'}}:= U_{JI} \cap \ov{U_I'} \cap \Phi_{JI}^{-1}(\ov{U_J'})
\eeqn
while the inclusion $\ov{U_{JI}'} \subseteq \ov{\ov{U_{JI}'}}$ is in general not an equality.

\begin{rem}
A very subtle point is that the natural map $|{\mf A}'| \to |{\mf A}|$ is continuous but not necessarily a homeomorphism onto its image. On the other hand, if we consider the closure $\ov{\mf A'}$, then the map $|\ov{\mf A'}| \to |{\mf A}|$ is a homeomorphism onto its image. 
\end{rem}

We prove a technical result about compatible Riemannian metrics. The argument is a prototype of many constructions in the rest of this paper. 

\begin{defn}
Let ${\mf A}$ be a $\uds{\bf Kur}^{G, \flat}$-atlas on ${\mf M}$. Let ${\mf V} = (V_I)_{I \in {\mc I}}$ be an open subcategory of ${\mf A}$. A {\bf Riemannian metric} on ${\mf V}$ consists of $G$-invariant metrics on $V_I:=\iota_I^{-1}({\mf V}) \subset U_I$ such that for each pair $I < J$, the germ of open embedding $\Phi_{JI}^\epsilon: N_{JI}^\epsilon|_{V_I\cap U_{JI}} \to V_J$ is isometric. Here $N_{JI}^\epsilon$ has the bundle metric determined by the metric on $V_I \cap U_{JI}$, the flat connection and Hermitian metric on $N_{JI}$.    
\end{defn}

Obviously the set of $G$-invariant metrics on open subcategories of ${\mf A}$ form a sheaf. Then one can talk about germs of $G$-invariant metrics.

\begin{lemma}\label{lemma_compatible_metrics}
Let ${\mf A}$ be a $\uds{\bf Kur}^{G, \flat}$-atlas on ${\mf M}$ and let ${\mf A}'$ be a shrinking. Then there exists a Riemannian metric near $\ov{\mf A'}$.
\end{lemma}

\begin{proof}
For any minimal element $I$, one can choose a $G$-invariant Riemannian metric on $U_I$. Inductively, for a general $I$, suppose one has constructed Riemannian metrics on $U_{J}$ for all $J < I$ such that for all $J_1 < J_2 < I$, the open embedding $\Phi_{J_2 J_1}^\epsilon: N_{J_2 J_1}^\epsilon \to U_{J_2}$ is an isometry near $\Phi_{J_2 J_1}(\ov{\ov{U_{J_2 J_1}'}})$. Then using the bundle metric on $N_{IJ}^\epsilon$ and the open embedding $\Phi_{IJ}^\epsilon$, one obtains a Riemannian metric in the image of the open embedding. The cocycle condition implies that, for $J_1, J_2 < I$, the two metrics agree near $\Phi_{IJ_1}(\ov{\ov{U_{IJ_1}'}}) \cap \Phi_{IJ_2}(\ov{\ov{U_{IJ_2}'}})$. Then there is a well-defined germ of $G$-invariant metrics near the compact set 
\beqn
\bigcup_{J< I} \Phi_{IJ}( \ov{\ov{U_{IJ}'}}) \subset U_I.
\eeqn
Then one can extend it to a $G$-invariant metric on $U_I$. Together with the existing ones on $U_J$, this newly constructed one is compatible with them near $\ov{\mf A'}$. Hence the induction can be continued.
\end{proof}

\subsubsection{NCS structures}

Now we impose the NCS structure on K-charts.

\begin{defn}\label{defn_Kuranishi_NCS}
Let ${\mf M}$ be a topological space.

\begin{enumerate}

\item An {\bf NCS structure} on a $\uds{\bf Kur}^{G, \flat}$-chart $K = (U, E, S, \Psi)$ on ${\mf M}$ consists of a $G$-equivariant NCS virtual manifold  structure on $(U, E)$ such that $S$ is a stratified section. 
A $\uds{\bf Kur}^{G, \flat}$-chart equipped with an NCS structure is called a $\uds{\bf Kur}^{G, \flat}_{\rm NCS}$-chart.

\item Let $K = (U, E, S, \Psi)$ be a $\uds{\bf Kur}^{G, \flat}_{\rm NCS}$-chart on ${\mf M}$. Let $N \to U$ be a $G$-equivariant flat stratified complex vector bundle with a cosheaf map ${\mc O}^U \to {\mc O}^N$ and a $G$-invariant parallel Hermitian metric. Let $N^\epsilon \subset N$ be a $G$-invariant disk bundle. The {\bf stabilization} of $K$ by $N^\epsilon$ is the stabilization ${\rm Stab}_{N^\epsilon}(K)$ as a $\uds{\bf Kur}^{G, \flat}$-chart  equipped with the NCS structure coming from stabilization (see Definition \ref{defn_NCS_manifold_stabilization}).

\item Let $K_i = (U_i, E_i, S_i, \Psi_i)$, $i = 1, 2$ be two $\uds{\bf Kur}^{G, \flat}_{\rm NCS}$-charts on ${\mf M}$ and ${\bm \Phi}_{21}: K_1 \to K_2$ be an embedding. A {\bf flat tubular neighborhood} of ${\bm \Phi}_{21}$ is a flat tubular neighborhood $(N_{21}^\epsilon, {\bm \Phi}_{21}^\epsilon)$ (see Definition \ref{defn_equivariant_Kuranishi_chart}) such that $N_{21}$ is a $G$-equivariant flat stratified Hermitian vector bundle equipped with a cosheaf map ${\mc O}^{U_1} \to {\mc O}^{N_{21}}$ and such that the open embedding ${\bm \Phi}_{21}^\epsilon$ respects the NCS structures. It is straightforward to define germ equivalence of flat tubular neighborhoods. A {\bf combed embedding} from $K_1$ to $K_2$ is an embedding together with a germ   of flat tubular neighborhoods.
\end{enumerate}
\end{defn}

One can also see that  combed embeddings can be composed. Hence one obtains a category 
\beqn
\uds{\bf Kur}_{\rm NCS}^{G, \flat}({\mf M})
\eeqn
whose objects are $\uds{\bf Kur}^{G, \flat}_{\rm NCS}$-charts on ${\mf M}$ and whose morphisms are NCS combed embeddings. By modifying Definition \ref{defn_Kuranishi_atlas}, one can define the notion of $\uds{\bf Kur}^{G, \flat}_{\rm NCS}$-atlases. We omit the details.

Notice that if ${\mf A}$ is a $\uds{\bf Kur}_{\rm NCS}^{G, \flat}$-atlas on ${\mf M}$, then there is an open subcategory 
\beqn
{\mf A}_{\rm main}^* = (U_{I, {\rm main}}^*)_{I \in {\mc I}} \subset {\mf A}.
\eeqn

\subsubsection{Straightenings}

To obtain a system of NCS transverse perturbations on an atlas, one needs to choose compatible straightenings.

\begin{defn}\label{defn_atlas_straightening}
Let ${\mf A}$ be a $\uds{\bf Kur}_{\rm NCS}^{G, \flat}$-atlas on ${\mf M}$ and ${\mf V} \subset {\mf A}$ be an open subcategory. A {\bf straightening} on ${\mf V}/G$ consists of, for each chart $K_I = (U_I, E_I, S_I, \Psi_I)$, a straightening on the quotient orbifold $U_I/G$ (Definition \ref{defn_orbifold_straightening}) satisfying the following conditions:  whenever $I \leq J$, the straightening of $V_J/G$ agrees with the bundle straightening on the total space of $N_{JI}^\epsilon/G$ near $\Phi_{JI}( V_{JI})/G$ (see Lemma \ref{lemma_stabilization_straightening}).
\end{defn}

The sets of straightenings on open subcategories form a sheaf. We prove the existence of compatible system of straightenings.

\begin{lemma}\label{lemma_atlas_straightening}
Let ${\mf A}$ be a $\uds{\bf Kur}_{\rm NCS}^{G, \flat}$-atlas on ${\mf M}$ and ${\mf A}'$ be a shrinking. Then there exists a straightenings near $\ov{\mf A'}/G$.
\end{lemma}

\begin{proof}
It is a similar construction to the proof of Lemma \ref{lemma_compatible_metrics}. The details are left to the reader.
\end{proof}

\subsubsection{Atlases with boundary}

One can define the notion of atlases with boundary by allowing the domain of the Kuranishi chart $U_I$ to have boundary and requiring that the domain map of coordinate change $\Phi_{JI}: U_{JI} \to U_J$ sends boundary to boundary. In this case, one obtains a boundary subcategory $\partial {\mf A} \subset {\mf A}$. Let
\beqn
\partial {\mf M}:= \bigcup_{I\in {\mc I}} \Psi_I( \partial U_I \cap S_I^{-1}(0))
\eeqn
which is a closed subset. The $\partial {\mf A}$ becomes an atlas on $\partial {\mf M}$. 

The technical constructions such as Riemannian metrics and thickenings can all be extended from boundary to the interior. To simplify the discussion, we always assume that ${\mf A}$ has a collar neighborhood, whose meaning is obvious. In particular, there exists $\tau>0$ sufficiently small such that all $U_I$ have collared neighborhoods of the form $\partial U_I \times [0, \tau)$ such that all structures are constant in the $[0, \tau)$-direction. All notions which are constant in the collared coordinate direction are called {\bf collared}. 

\begin{lemma}\label{lemma413}
Let ${\mf A}$ be a $\uds{\bf Kur}^{G, \flat}_{\rm NCS}$-atlas with boundary on ${\mf M}$. Let ${\mf A}'$ be a collared shrinking of ${\mf A}$ (so $\partial {\mf A}'$ is necessarily a shrinking of $\partial {\mf A}$).
\begin{enumerate}

\item Given a Riemannian metrics near $\ov{\partial {\mf A'}}$, there exists a Riemannian metrics near $\ov{\mf A'}$ whose boundary restriction coincides with the existing one and which is constant in the collar direction near the boundary.

\item Given a straightenings near $\ov{\partial {\mf A}'}/G$, there exists a straightenings near $\ov{\mf A'}/G$ whose boundary restriction coincides with the existing one and which is constant in the collar direction near the boundary.
\end{enumerate}
\end{lemma}

\begin{proof}
One only needs to alter the inductions in the proofs of Lemma \ref{lemma_compatible_metrics} and Lemma \ref{lemma_atlas_straightening} slightly. Starting from a minimal element $I \in {\mc I}$. If $\partial U_I = \emptyset$, then one choose a Riemannian metric or a straightening near $\ov{U_I'}$; if $\partial U_I \neq \emptyset$, then a Riemannian metric or a straightening is already constructed near $\ov{\partial U_I'}$. Then using the collared neighborhood, one can extend the existing structure to a neighborhood of $\ov{\partial U_I'} \times [0, \tau']$ for some $\tau'>0$. Then one extends it to a neighborhood of $\ov{\partial U_I'} \subset U_I$. The inductive step is similar, where the boundary restrictions are already given and one can use the collared neighborhood to obtain an extension to a neighborhood of the boundary. 
\end{proof}

\subsection{Multivalued NCS perturbations}

We consider multivalued perturbations of the Kuranishi section on the quotient virtual orbifolds. We have the following obvious consequences.
\begin{enumerate}

\item Let $K = (U, E, S, \Psi)$ be a $\uds{\bf Kur}_{\rm NCS}^{G, \flat}$-chart on ${\mf M}$. Then by taking quotient, one obtains an NCS virtual orbifold $(U/G, E/G)$ (Definition \ref{defn_NCS_virtual_orbifold}) and a stratified section
\beqn
S/G\in \Gamma_{\rm stratified}(U/G, E/G).
\eeqn

\item Consider the stabilization 
\beqn
{\rm Stab}_{N^\epsilon} K = (\hat U, \hat E, \hat S, \hat \Psi)
\eeqn
by a complex stratified flat disk bundle $N^\epsilon$. Then $(\hat U/G, \hat E/G)$ has an inherited structure of NCS virtual orbifold. Moreover, there is a stabilization map
\beqn
{\rm Stab}_{N^\epsilon}: \Gamma_{\rm stratified}^{\rm multi}(U/G, E/G) \to \Gamma_{\rm stratified}^{\rm multi} (\hat U/G, \hat E/G).
\eeqn
\end{enumerate}

\begin{defn}
Let ${\mf A}$ be a $\uds{\bf Kur}_{\rm NCS}^{G, \flat}$-atlas on ${\mf M}$.

\begin{enumerate}

\item Let ${\mf V}= (V_I)_{I \in {\mc I}} \subset {\mf A}$ be a subcategory. Denote $V_I:= \iota_I^{-1}({\mf V}) \subset U_I$.  A {\bf multisection} on ${\mf V}/G$ consists of a collection ${\mf S}'$ of multisections
\beqn
{\mc S}_I' \in \Gamma^{\rm multi}_{\rm stratified} (V_I/G, E_I/G)
\eeqn
such that for all $I \leq J$, choosing a representative of the flat tubular neighborhood $(N_{JI}^\epsilon, {\bm \Phi}_{JI}^\epsilon)$, with respect to ${\bm \Phi}_{JI}^\epsilon$, one has 
\beqn
{\rm Stab}_{N_{JI}^\epsilon}({\mc S}_I'|_{V_{JI}}) = {\mc S}_J'.
\eeqn
Denote the set of multisections on ${\mf V}$ by 
\beqn
\Gamma_{\rm stratified}^{\rm multi}({\mf V}/G).
\eeqn
Notice that the original Kuranishi sections define a stratified (single-valued) section, denoted by 
\beqn
{\mf S} \in \Gamma_{\rm stratified}({\mf A}/G).
\eeqn

\item The {\bf zero locus} of a multisection ${\mf S}'$ on ${\mf V}/G$ is the subcategory
\beqn
({\mf S}')^{-1}(0):= ( ({\mc S}_I')^{-1}(0))_{I \in {\mf A}} \subset {\mf V}/G.
\eeqn

\item Let ${\mf V}'  = (V_I')_{I \in {\mc I}} \subset {\mf V}$ be a subcategory. For $r\geq 0$ denote
\begin{align*}
&\ {\mf V}_r':= (V_{I, r}')_{I \in {\mc I}}\ &\ {\rm where}\ V_{I, r'} = B_r(V_I') \cap V_I.
\end{align*}
and the radius $r$ neighborhood $B_r(V_I')$ is defined via a given Riemannian metric on ${\mf V}$. Define
\beqn
V_{JI, r}':= V_{I, r}' \cap U_{JI} \cap \Phi_{JI}^{-1}(V_{J, r}').
\eeqn
A multisection ${\mf S}'$ on ${\mf V}$ is said to {\bf $r$-regular near ${\mf V}'$} if for all $I \leq J$, there exists a representative of the flat tubular neighborhood $(N_{JI}^\epsilon, {\bm \Phi}_{JI}^\epsilon)$ such that 
\beqn
\ov{N_{JI}^r}|_{\ov{V_{JI,r}'}} \subset N_{JI}^\epsilon
\eeqn
(here $\ov{N_{JI}^r}$ is the closed disk bundle of radius $r$) and such that near $\Phi_{JI}^\epsilon( \ov{N_{JI}^r}|_{\ov{V_{JI,r}'}})/G$, ${\mc S}_J'$ coincides with the stabilization of ${\mc S}_I'$.

\item Two multisections ${\mf S}_1, {\mf S}_2$ on ${\mf V}$ are {\bf $\delta$-close} if for each $I$, the multisections $S_{I, 1}, S_{I, 2}$ are $\delta$-close with respect to the Hermitian metric on $E_I$, in the sense of Definition \ref{defn_closedness}.

\end{enumerate}
\end{defn}

\subsubsection{Technical results on NCS transverse multisections}

Now we state the transversality result regarding NCS perturbations.

\begin{prop}\label{prop_atlas_NCS_transverse}
Let ${\mf A}$ be a $\uds{\bf Kur}_{\rm NCS}^{G, \flat}$-atlas on ${\mf M}$ and ${\mf A}'$ be a shrinking. Given a Riemannian metric near $\ov{\mf A'}$ and a straightening near $\ov{\mf A'}/G$, there exists $r > 0$ such that for all $\delta>0$, there there exists an NCS transverse multisection ${\mf S}'$ near $\ov{\mf A'}$ with respect to the given straightening which is $r$-regular near $\ov{\mf A'}$ and which is $\delta$-close to the original Kuranishi section. 
\end{prop}

\begin{proof}
As ${\mf A}' \subset {\mf A}$ is precompact, one can choose $r>0$ sufficiently small such that the open subcategory ${\mf A}_r'$ is still precompact. Moreover, we may assume that the Riemannian metric exists near $\ov{{\mf A}_r'}$ and the straightening exists near $\ov{{\mf A}_r'}/G$. For each $I$, extend the straightening to $U_I/G$ arbitrarily while maintaining its value near $\ov{{\mf A}_r'}/G$. Moreover, we may assume that the flat  tubular neighborhoods are all defined near the radius $r$ closed disk bundle $\ov{N_{JI}^r}$ over a neighborhood of $\ov{U_{JI,r}'}$. 

We construct inductively the perturbations on all charts. For a minimal $I$, by Theorem \ref{thm_transversality_orbifold}, one can find an NCS transverse multisection ${\mc S}_I'\in \Gamma_{\rm stratified}^{\rm multi}(U_I/G, E_I/G)$ which is $\delta$-close to ${\mc S}_I$. Inductively, for a general $I$, assume that one has constructed multisections ${\mc S}_J' \in \Gamma_{\rm stratified}^{\rm multi}(U_{J}/G, E_{J}/G)$ for all $J < I$ satisfying the following conditions.

\begin{enumerate}

\item ${\mc S}_J'$ is NCS transverse and $\delta$-close to ${\mc S}_J$.

\item For any $J_1 < J_2$, ${\mc S}_{J_2}'$ coincides with the stabilization of ${\mc S}_{J_1}'$ near $\Phi_{J_2 J_1}^\epsilon(\ov{N_{J_2 J_1}^r}|_{\ov{U_{J_2 J_1, r}'}})/G$.
\end{enumerate}
Then for the chart $U_I$, the existing perturbations induces, via the coordinate changes and stabilizations, a multisection ${\mc S}_I'$ near
\beqn
\left( \bigcup_{J < I} \Phi_{IJ}^\epsilon (\ov{N_{IJ}^r}|_{\ov{U_{IJ,r}'}}) \right) /G.
\eeqn
Since NCS transversality is invariant under stabilization, the induction hypothesis implies that ${\mc S}_{I}'$ is NCS transverse near this closed set. 
The $\delta$-closeness condition is also preserved by stabilization. 
By using Theorem \ref{thm_transversality_orbifold}, one can find an NCS transverse multisection ${\mc S}_I'$ which agrees with the existing values near the above closed set and which is $\delta$-close to ${\mc S}_I$. This construction extends the induction hypothesis to $I$. The induction finishes after finitely many steps. All of these chart-wise multisections defines a multisection near $\ov{\mf A'}$ which is $r$-regular, NCS transverse, and $\delta$-close to the original Kuranishi section.
\end{proof}

\begin{lemma}\label{delta_r}
Let ${\mf A}$ be a $\uds{\bf Kur}^{G, \flat}$-atlas on ${\mf M}$ and ${\mf A}'$ be a shrinking. Suppose near $\ov{\mf A'}$ there exist a Riemannian metric and a straightening. Then for any $r'$ no greater than the one given by Proposition \ref{prop_atlas_NCS_transverse} (which depends on ${\mf A}$, ${\mf A}'$, the Riemannian metric and the straightening), there there exists $\delta(r')>0$ satisfying the following conditions. For any multisection ${\mf S}'$ near $\ov{\mf A'}$ which is $r'$-regular, if it is $\delta(r')$-close to ${\mf S}$, then 
\beqn
|({\mf S}')^{-1}(0)| \cap |\ov{{\mf A}'}/G| = |({\mf S}')^{-1}(0)| \cap |{\mf A'}/G|.
\eeqn
\end{lemma}

\begin{proof}
The argument is carried out on the level of $G$-quotients and orbifolds. We assume that $G$ is trivial so one can avoid writing $\cdot/G$ to save notations. 

Suppose this lemma is not true, then there exist $r'>0$, a sequence $\delta_i \to 0$, a sequence of germs of multisections ${\mf S}_i'$ which is $r'$-regular near $\ov{\mf A'}$ and which are $\delta_i$-close to ${\mf S}$ but there exists a sequence 
\beqn
x_i \in |({\mf S}_i')^{-1}(0)| \cap \Big( |\ov{\mf A'}|\setminus |{\mf A}'| \Big)
\eeqn
As $|\ov{\mf A'}|$ is compact, there exists a subsequence (still denoted by $x_i$) which converges to $x_\infty \in |\ov{\mf A'}|$. 

We first claim that $x_\infty \in |{\mf S}^{-1}(0)|$, i.e., a zero of the original Kuranishi section. Indeed, if not, then ${\mf S}(x_\infty) \neq 0$. As ${\mf S}_i'$ are $\delta_i$-close to ${\mf S}$ and $\delta_i \to 0$, this is impossible. So $x_\infty \in |{\mf S}^{-1}(0)| \subset |{\mf A}'|$. 

As the zero locus $|{\mf S}^{-1}(0)|$ can be covered by $U_I' \cap S_I^{-1}(0)$, one can find $I\in {\mc I}$ and $x_{\infty, I} \in \iota_I^{-1}(x_\infty) \subset U_I' \cap S_I^{-1}(0)$. Moreover, by choosing a subsequence, one can find $J$ such that $x_{i, J} \in \iota_J^{-1}(x_i) \subset \ov{U_J'}$ for all $i$. By choosing a further subsequence, one may assume that $x_{i, J}$ converges to a point $x_{\infty, J} \in \ov{U_J'} \subset U_J$. Then obviously $x_{\infty, J} \curlyvee x_{\infty, I}$. By the definition of Kuranishi atlas, either $I \leq J$ or $J \leq I$. We discuss in the following two cases.

\vspace{0.2cm}

\noindent {\it Case (1)}. Assume $J \leq I$. Then $x_{\infty, J} = \Phi_{JI}^{-1}(x_{\infty, I}) \in \Phi_{JI}^{-1}(U_I')$, which is open in $U_J$. Hence for $i$ sufficiently large, one has $x_{i, J}\in \Phi_{JI}^{-1}(U_I')$. Then $x_{i, J} \in U_{JI}' \subset U_J'$. It implies that $x_i \in |{\mf A}'|$, which contradicts the assumption. 

\vspace{0.2cm}

\noindent {\it Case (2)}. Assume $I < J$. Then 
\beqn
x_{\infty, I} \in U_I' \cap \Phi_{JI}^{-1}(\ov{U_J'}) \subset U_{I, r'}' \cap \Phi_{JI}^{-1}( U_{J, r'}') = U_{JI, r'}',\ x_{\infty, J} \in  \Phi_{JI}( U_{JI, r'}') \subset \Phi_{JI}^\epsilon( N_{JI}^{r'}|_{U_{JI, r'}'}).
\eeqn
The last set is open in $U_J$, hence for $i$ sufficiently large, one has 
\beqn
x_{i, J} \in \Phi_{JI}^\epsilon( N_{JI}^{r'}|_{U_{JI, r'}'}).
\eeqn
However, since ${\mc S}_i'$ is $r'$-regular near $\ov{\mf A'}$, it follows that $x_{i, J} \in \Phi_{JI}(U_{JI,r'}')$. Then one can write $x_{i, J} = \Phi_{JI}(x_{i, I})$ with $x_{i, I} \in U_{JI, r'}'$. One then has $x_{i, I} \to x_{\infty, I} \in U_I'$ inside $U_I$. As $U_I'$ is open, for $i$ sufficiently large, $x_{i, I} \in U_I'$. Then $x_i \in |{\mf A'}|$ which contradicts our assumption. 
\end{proof}

\subsubsection{Extending boundary perturbations}

The key consideration is to establish compact homotopies for NCS transverse perturbations. 

\begin{prop}\label{prop417}
Let ${\mf M}$ be compact. Let ${\mf A}$ be a $\uds{\bf Kur}_{\rm NCS}^{G, \flat}$-atlas with boundary on ${\mf M}$. Let ${\mf A}'$ be a shrinking of ${\mf A}$. Suppose there are a Riemannian metric near $\ov{\mf A'}$ and a straightening near $\ov{\mf A'}/G$ which respect the collared neighborhood of $\partial {\mf A}$ and $\partial {\mf A}'$.

Suppose $\partial {\mf S}'$ is an NCS transverse multisection which is $r_0$-regular near $\ov{\partial \mf A'}$ and $\delta_0$-close to $\partial {\mf S}$. Then there exist $r\leq r_0$ and a multisection ${\mf S}'$ near $\ov{\mf A'}$ satisfying 1) ${\mf S}'$ is NCS transverse with respect to the given straightening; 2) ${\mf S}'$ is $r$-regular near $\ov{\mf A'}$ and $\delta_0$-close to ${\mf S}$; 3) ${\mf S}'$ is collared.
\end{prop}

\begin{proof}
One chooses $r\leq r_0$ as the beginning of the proof of Proposition \ref{prop_atlas_NCS_transverse}. Then the boundary perturbation $\partial {\mf S}'$ is also $r$-regular near $\ov{\partial {\mf A'}}$. One also extends the thickening near $\ov{U_{I, r}'}$ arbitrarily to $U_I$. 

Then one can inductively construct the perturbations in the same way as the proof of Proposition \ref{prop_atlas_NCS_transverse}. For a minimal index $I$, if $\partial U_I = \emptyset$, then one can choose $\delta_0$-close, NCS transverse multisections on $U_I$; if $\partial U_I \neq \emptyset$, then the boundary value is given and one uses the collar structure to extend it to a collared neighborhood of the boundary. The remaining construction is similar where we can always maintain the requirement on multisections for being $r$-regular and $\delta_0$-closed.
\end{proof}

\subsection{Virtual cycle of the main stratum} 

\subsubsection{The virtual count of main stratum}

Now assume that ${\mf M}$ is compact and ${\mf A}$ is a $\uds{\bf Kur}_{\rm NCS}^{G, \flat}$-atlas on ${\mf M}$. ${\mf A}'$ is a shrinking of ${\mf A}$. Assume that ${\mf A}$ is oriented. Choose a Riemannian metric near $\ov{\mf A'}$ and a straightening near $\ov{\mf A'}/G$. Choose a sufficiently small $r>0$ and an NCS transverse multisection near $\ov{\mf A'}$ with respect to the chosen straightening which is $r$-regular near $\ov{\mf A'}$ and which is sufficiently $C^0$-close to ${\mf S}$. Then by Lemma \ref{delta_r}, the perturbed zero locus 
\beqn
{\mf Z}':= |({\mf S}')^{-1}(0)| \cap |{\mf A'}/G|
\eeqn
is compact. Consider its main stratum defined by 
\beqn
{\mf Z}_{\rm main}^{'*}:= {\mf Z}' \cap (|{\mf A}_{\rm main}^{'*} / G|).
\eeqn
When the virtual dimension is zero, it is a compact weighted branched 0-dimensional manifold, i.e., a finite collection of points with rational weights. If ${\mf A}$ is also oriented, then these points also have signs, allowing us to define the weighted sum 
\beqn
\# {\mf Z}_{\rm main}^{'*} \in {\mb Q}.
\eeqn

\begin{prop}
Under the above assumptions, 
\begin{enumerate}

\item when ${\rm dim}^{\rm vir} {\mf A} =0$ and $\partial {\mf A} = \emptyset$, the above weighted count of zeroes in the main stratum of an NCS transverse multisection on ${\mf A}'$ is an invariant of the pair $({\mf A}, {\mf A}')$, denoted by 
\beqn
\#^{\rm vir}({\mf A}, {\mf A}') \in {\mb Q};
\eeqn

\item when ${\rm dim}^{\rm vir} {\mf A} = 1$, ${\mf A}$ has boundary $\partial {\mf A}$, and $({\mf A}, {\mf A}')$ has a collared neighborhood near the boundary, one has
\beqn
\#^{\rm vir}(\partial {\mf A}, \partial {\mf A}') = 0.
\eeqn
\end{enumerate}
\end{prop}

\begin{proof}
To prove that the virtual count is an invariant, consider the product ${\mf A} \times [0, 1]$ as an atlas with boundary on ${\mf M}\times [0, 1]$. Any two Riemannian metrics near $\ov{\mf A'}$ can be extended to a Riemannian metric near $\ov{\mf A'}\times [0,1]$; any two straightenings near $\ov{\mf A'}/G$ can also be extended to a straightening near $\ov{\mf A'}/G \times [0, 1]$. We may choose such extensions so that they are constant over $[0, \tau]$ and over $[1-\tau, 1]$ for a small $\tau>0$. Suppose ${\mf S}_i'$, $i = 0, 1$ are two NCS transverse multisections near $\ov{\mf A'}$ with respect to the two straightenings which are $r_i$-regular and $\delta_i$-close to ${\mf S}$.

To construct an interpolation, choose $r \leq r_0, r_1$ sufficiently small. Then choose $\delta = \delta(r)$ given by Lemma \ref{delta_r} associated to the atlas ${\mf A}\times [0, 1]$, the shrinking ${\mf A}'\times [0, 1]$, the Riemannian metric and the straightening. Then use the collared neighborhood to obtain a family of perturbations ${\mf S}_t'$ for $t \in [0, \tau] \sqcup [1-\tau, 1]$, which is $r$-regular near $\ov{\mf A'}\times ([0, \tau] \sqcup [1-\tau, 1])$, NCS transverse, such that ${\mf S}_\tau'$ and ${\mf S}_{1-\tau}'$ are $\delta$-close to ${\mf S}$. Then over the interval $[\tau, 1-\tau]$, construct an NCS transverse perturbations which is still $r$-regular and $\delta$-close to ${\mf S}$. Denote the whole family by $\tilde {\mf S}'$. Then 
\beqn
|(\tilde {\mf S}')^{-1}(0)| \cap |({\mf A'}\times [0, 1])/G| 
\eeqn
is compact. Let the perturbed zero locus be $\tilde {\mf Z}'$. Then its main stratum $\tilde {\mf Z}_{\rm main}^*$ is a compact 1-dimensional weighted branched manifold whose boundary is the union ${\mf Z}_{0, {\rm main}}^{'*}\sqcup {\mf Z}_{1, {\rm main}}^{'*}$. Then it is standard that the boundary count is zero. Hence we obtained the independence of the virtual count from the choices of Riemannian metrics, straightenings, and perturbations. The second item of this proposition follows from the same argument. 
\end{proof}

\subsubsection{Virtual cycle in general}

For virtual dimensions different from $0$, we ``pushforward'' the expected virtual class to a target orbifold.

\begin{defn}\label{defn_strong_map}
Let ${\mf A}$ be a $\uds{\bf Kur}^{G, \flat}$-atlas on ${\mf M}$ and $Y$ be an orbifold. A (continuous or smooth) {\bf strong map} from a subcategory ${\mf V} = (V_I)_{I \in {\mc I}}$ to $Y$ is a $G$-invariant (continuous or smooth) map
\beqn
{\mf f}: {\rm Ob} {\mf V} \to Y
\eeqn
satisfy the following compatibility condition: for each $I \in {\mc I}$, denote $f_I:= {\mf f}|_{V_I}$. Then for all pairs $I \leq J$, near $\Phi_{JI}(V_{JI}) \subset V_J$, $f_J$ coincides with the pullback of $f_I$ via the tubular projection associated to the tubular neighborhood. 
\end{defn}

Denote by
\begin{align*}
&\ C^0_{\rm strong}({\mf V}, Y),\ & \ C^\infty_{\rm strong}({\mf V}, Y)
\end{align*}
the set of continuous resp. smooth strong maps from ${\mf V} \to Y$. These set obviously form a sheaf on open subcategories of ${\mf A}$. Then if ${\mf V}\subset {\mf A}$ is an arbitrary subcategory, the same notations denote the sets of germs of continuous resp. smooth strong maps defined near ${\mf V}$. Depending on choosing a distance function on $Y$, for ${\mf f}_0, {\mf f}_1 \in C_{\rm strong}^0({\mf V}, Y)$, define
\beqn
{\rm dist}_{\mf V}( {\mf f}_0, {\mf f}_1):= \sup_{x \in {\mf V}} {\rm dist} ({\mf f}_0(x), {\mf f}_1(x)).
\eeqn
When it is less than $\rho> 0$, we say ${\mf f}_0$ and ${\mf f}_1$ are $\rho$-close.

We can perturb continuous strong maps to smooth ones as long as we allow non-compatibility outside a given shrinking.

\begin{lemma}
Let ${\mf A}$ be a $\uds{\bf Kur}_{\rm NCS}^{G, \flat}$-atlas on ${\mf M}$ and ${\mf A}'$ be a shrinking. Suppose ${\mf f} \in C_{\rm strong}^0({\mf A}, Y)$. Choose a distance function on $Y$. Then for any $\rho>0$, there exists ${\mf f}_\rho \in C_{\rm strong}^\infty( \ov{\mf A'}, Y)$ which is $\rho$-close to ${\mf f}$. Moreover, there exists $\rho > 0$ such that for any ${\mf f}_{\rho,0}, {\mf f}_{\rho, 1}\in C_{\rm strong}^\infty(\ov{\mf A'}, Y)$ both of which are $\rho$-close to ${\mf f}$, there exists $\tilde {\mf f}_\rho \in C_{\rm strong}^\infty( \ov{\mf A'} \times [0, 1], Y)$ such that $\tilde {\mf f}|_{\ov{\mf A'} \times \{t\}} = {\mf f}_{\rho, 0}$ and $\tilde {\mf f}|_{\ov{\mf A'} \times \{1-t\}} = {\mf f}_{\rho, 1}$ for all $t\geq 0$ sufficiently small and for all $t \in [0, 1]$, 
\beqn
{\rm dist}_{\ov{\mf A'}}( \tilde {\mf f}_\rho (\cdot, t), {\mf f}) < \rho.
\eeqn

\end{lemma}

\begin{proof}
It is true that on a $G$-manifold $U$, any $G$-invariant continuous function can be $C^0$-approximated by $G$-invariant smooth functions. Hence one can inductively construct  smooth strong maps as $C^0$-approximations. Moreover, as $Y$ is an orbifold, any two sufficient $C^0$-closed maps are homotopic.   
\end{proof}

Now given an oriented $\uds{\bf Kur}_{\rm NCS}^{G, \flat}$-atlas ${\mf A}$ on a compact ${\mf M}$ of virtual dimension $d$, a shrinking ${\mf A}' \subset {\mf A}$, and ${\mf f} \in C_{\rm strong}^0( {\mf A}, Y)$ with $Y$ being a compact oriented orbifold, we would like to define a virtual fundamental class
\beqn
[{\mf A}, {\mf A}']_{\rm main}^{\rm vir} \in H_d(Y; {\mb Q}).
\eeqn
To specify such a homology class, using the fact that $Y$ satisfies Poincar\'e duality over ${\mb Q}$, such a homology class is determined by its pairing with homology classes of the complementary degree. This allows us to define the virtual fundamental class using the case of virtual dimension zero and proving the invariance using the case of virtual dimension 1. 

More precisely, one consider a chain model for rational homology $C_*(Y)$ where $C_k(Y)$ is the space of finite ${\mb Q}$-linear combinations of piecewise smooth cubical chains $\lambda: [0, 1]^k \to Y$ which are constants along normal directions near each facet of the cube. Suppose ${\rm dim} Y = m$. We would like to define a linear functional on $H_{m-d}(Y; {\mb Q})$ as follows. Fix $[\Lambda] \in H_{m-d}(Y; {\mb Q})$. Make the following choices.

\vspace{0.2cm}

\noindent ---Choice (i)--- A distance function on $Y$, a sufficiently small $\rho>0$, and ${\mf f}_\rho \in C_{\rm strong}^\infty( \ov{\mf A'}, Y)$ which is $\rho$-close to $f$. By abuse of notations, assume ${\mf f} = {\mf f}_\rho$ is already smooth.

\vspace{0.2cm}

\noindent ---Choice (ii)--- A cycle $\Lambda$ representing $[\Lambda]$ such that each facet $\lambda_i: Q_i \cong [0, 1]^{k_i} \to Y$ appearing in $\Lambda$ and each chart $K_I$, $\lambda_i$ is transverse to the restriction of $f_I$ to each stratum $U_{I, \alpha}^*$ in the interior of $Q_i$.

\vspace{0.2cm}

Then for each facet $\lambda_i: Q_i \to Y$, one consider an atlas ${\mf A}_{\lambda_i}$ whose charts $K_{I, \lambda_i}$ is the restriction of $K_I \times Q_i$ to the fibre products
\beqn
U_{I, \lambda_i} = \Big\{ (x, t) \in U_I \times Q_i\ |\ f_I(x) = \lambda_i(t) \Big\}.
\eeqn
The transversality implies that $U_{I, \lambda_i}$ is a $G$-invariant submanifold with corners with induced stratification. Moreover, as $\alpha_i$ is constant in normal directions near each corner of $Q_i$, $U_{I, \lambda_i}$ has canonical collar neighborhoods near corners. The bundle $E_I$, the section $S_I$, and the footprint map $\Psi_I$ can all be carried to $U_{I, \lambda_i}$, giving a $\uds{\bf Kur}^{G, \flat}$-chart on 
\beqn
{\mf M}_{\lambda_i} = \{ (x, t)\in {\mf M}\times Q_i\ |\ \ev(x) = \lambda_i(t) \}.
\eeqn
One can also see that one has corresponding coordinate changes $\Phi_{JI, \lambda_i}$ from $K_{I, \lambda_i}$ to $K_{J, \lambda_i}$ induced from $\Phi_{JI}$. Hence one obtains a $\uds{\bf Kur}^{G, \flat}$-atlas on ${\mf M}_{\lambda_i}$. Its virtual dimension is 
\beqn
{\rm dim}^{\rm vir} {\mf A}_{\lambda_i} = {\rm dim}^{\rm vir} {\mf A} + {\rm dim} Q_i - {\rm dim} Y.
\eeqn
Moreover, because $\lambda_i$ is transverse to the restriction of $f_I$ to all strata of $U_I$, the atlas ${\mf A}_{\lambda_i}$ also carries an induced NCS structure. The stratification on $U_{I, \lambda_i}$ is just given by
\beqn
U_{I, \lambda_i, \alpha}^* = (U_{I, \alpha}^* \times Q_i) \cap U_{I, \lambda_i}
\eeqn
whose normal bundle has the induced complex structure. Other requirement for an NCS structure (Definition \ref{defn_Kuranishi_NCS}) can be readily checked. Hence ${\mf A}_{\lambda_i}$ is a $\uds{\bf Kur}_{\rm NCS}^{G, \flat}$-atlas on ${\mf M}_{\lambda_i}$ with corners. 

Then we can construct NCS transverse multivalued perturbations. Notice that if ${\mf A}'$ is a shrinking of ${\mf A}$, then for each $\lambda_i$, there is an induced shrinking ${\mf A}_{\lambda_i}'$ of ${\mf A}_{\lambda_i}$. Then we can make a further choice.

\vspace{0.2cm}

\noindent ---Choice (iii)--- Choose a collection of Riemannian metric near $\ov{{\mf A}_{\lambda_i}'}$ and a collection of straightenings near $\ov{{\mf A}_{\lambda_i}'}/G$ satisfying the following conditions.
\begin{enumerate}

\item The metrics and straightenings are constant in normal directions near each corner of $U_{I, \lambda_i}$.

\item If $\lambda_i$ is a facet of $\lambda_j$, then the metric and straightening coincide with the restrictions of those on $U_{I, \lambda_j}$ to that facet.
\end{enumerate}

Then one can choose perturbations.

\begin{lemma}\label{lemma421}
There exist $r>0$ and for each $\delta>0$ a collection of multisections ${\mf S}_{\lambda_i}'$ near $\ov{{\mf A}_{\lambda_i}'}/G$ satisfying the following conditions.
\begin{enumerate}
\item ${\mf S}_{\lambda_i}'$ is $r$-regular near $\ov{{\mf A}_{\lambda_i}'}/G$.

\item ${\mf S}_{\lambda_i}'$ is constant in normal directions near each corner.

\item ${\mf S}_{\lambda_i}'$ is NCS transverse with respect to the chosen straightenings.

\item ${\mf S}_{\lambda_i}'$ is $\delta$-close to the original Kuranishi section ${\mf S}_{\lambda_i}$. 
\end{enumerate}
\end{lemma}

\vspace{0.2cm}

\noindent ---Choice (iv)--- Choose a collection of multisections ${\mf S}_{\lambda_i}'$ satisfying conditions of Lemma \ref{lemma421} for a $\delta>0$ sufficiently small. 

\vspace{0.2cm}

Then by Lemma \ref{delta_r}, the perturbed zero locus  
\beqn
{\mf Z}_{\lambda_i}':=|({\mf S}_{\lambda_i}')^{-1}(0)| \cap |{\mf A}_{\lambda_i}'/G|
\eeqn
is compact. Consider its main stratum ${\mf Z}_{\lambda_i, {\rm main}}^{'*}$. For dimensional reason, it is empty except for the top-dimensional cells $\lambda_i$, in which case it is a compact $0$-dimensional weighted branched manifold. Together with orientations, we can define $
\# {\mf Z}_{\lambda_i, {\rm main}}^{'*} \in {\mb Q}$. Then define
\beq\label{intersection_number}
[{\mf A}, {\mf A}']_{{\rm main}}^{\rm vir} \cap [\Lambda]  = \sum_{i} \# {\mf Z}_{\lambda_i, {\rm main}}^{' *}.
\eeq

\begin{lemma}
The above number only depends on ${\mf A}$, ${\mf A}'$, the strong map $f$, and the class $[\Lambda]$.
\end{lemma}

\begin{proof}
One needs to compare two sets of choices made in Choices (i)---(iv). The procedure is routine and is left to the reader. 
\end{proof}

From the definition, one can see that the intersection number \eqref{intersection_number} depends linearly on $[\Lambda] \in H_{m-d}(Y;{\mb Q})$. Hence one obtains a well-defined class 
\beqn
[{\mf A}, {\mf A}']_{\rm main}^{\rm vir} \in H_{{\rm dim}^{\rm vir} {\mf A}} (Y; {\mb Q}).
\eeqn

Using the cobordism argument again, one can prove the following fact. 

\begin{prop}\label{prop_VFC_boundary}
Suppose ${\mf A}$ is an oriented $\uds{\bf Kur}_{\rm NCS}^{G, \flat}$-atlas on ${\mf M}$ with boundary. Let ${\mf A}'$ be a shrinking of ${\mf A}$. Let ${\mf f}: {\mf A} \to Y$ be a continuous strong map. Suppose ${\mf A}, {\mf A}', {\mf f}$ are constant in the collar direction near the boundary. Then 
\beqn
[\partial {\mf A}, \partial {\mf A}']_{\rm main}^{\rm vir} = 0 \in H_{{\rm dim}^{\rm vir} {\mf A}-1}(Y; {\mb Q}).
\eeqn
\end{prop}

\subsubsection{Comparison to the ordinary VFC}

By forgetting the NCS structure on the atlas ${\mf A}$, one obtains an ordinary Kuranishi atlas. Together with a shrinking ${\mf A}' \subset {\mf A}$, one has the ordinary (total) virtual fundamental cycle
\beqn
[{\mf A}, {\mf A}']^{\rm vir}_{\rm total} \in H_*(Y; {\mb Q}).
\eeqn
Although the interesting situation is when the two virtual classes are not equal, under certain conditions, one can prove that they do agree.

Indeed, one defines the virtual dimension of the $\alpha$-th stratum by 
\beqn
{\rm dim}_{\alpha}^{\rm vir} {\mf A} = {\rm dim} U_{I, \alpha}^* - {\rm rank} E_{I, \beta(\alpha)}^* - {\rm dim} G
\eeqn
for any chart $K_I$ included in this atlas. One can easily see that this integer does not depend on the choice of the chart. When $\alpha$ is the main stratum, it agrees with the ordinary virtual dimension.

\begin{prop}\label{prop424}
Suppose for all $\alpha$ different from the main stratum there holds
\beqn
{\rm dim}_\alpha^{\rm vir} {\mf A} < {\rm dim}^{\rm vir} {\mf A}
\eeqn
then 
\beqn
[{\mf A}, {\mf A}']_{\rm main}^{\rm vir} = [{\mf A}, {\mf A}']_{\rm total}^{\rm vir}.
\eeqn
\end{prop}

\begin{proof}
The ordinary virtual fundamental cycle can also be defined using the same method using cubical chains and reduce the problem to zero-dimensional situation. Hence we may assume the virtual dimension is zero. One then chooses an NCS transverse perturbation. Because of the virtual dimension condition of this proposition, the perturbed zero locus is contained in $|{\mf A}_{\rm main}^{'*}/G|$. This implies that the NCS transverse perturbation is also transverse in the ordinary sense. Hence the ordinary VFC coincides with the main stratum VFC.
\end{proof}

\subsection{Free quotients and stabilizations}

\subsubsection{Free quotients}

\begin{defn}
Let ${\mf A}$ be an $\uds{\bf Kur}_{\rm NCS}^{G, \flat}$-atlas on ${\mf M}$. Suppose $H \subset G$ is a closed normal subgroup which acts freely on the domains of all charts. Then the {\bf free quotient} ${\mf A}/H$ is the $\uds{\bf Kur}_{\rm NCS}^{G/H, \flat}$-atlas whose charts are 
\beqn
K_I/H = (U_I/H, E_I/H, S_I/H, \Psi_I/H)
\eeqn
and whose coordinate changes are the induced ones. If ${\mf f}: {\mf A} \to Y$ is a strong map, then the quotient by $H$ induces a strong map 
\beqn
{\mf f}/H: {\mf A}/H \to Y.
\eeqn
\end{defn}

\begin{prop}\label{prop_free_quotient}
There holds
\beqn
[{\mf A}, {\mf A}']_{\rm main}^{\rm vir}  = [{\mf A}/H, {\mf A}'/H]_{\rm main}^{\rm vir}.
\eeqn
\end{prop}

\begin{proof}
To construct the virtual cycle, one needs to choose a Riemannian metric near the closure of the shrinking, a straightening, and a sufficiently regular and close NCS multisection. All such constructions for the pair $({\mf A}, {\mf A}')$ canonically descends to $({\mf A}/H, {\mf A}'/H)$. 
\end{proof}

\subsubsection{Stabilizations}

\begin{defn}
Let ${\mf A}$ be an $\uds{\bf Kur}_{\rm NCS}^{G, \flat}$-atlas on ${\mf M}$. Let $Q$ be a unitary representation of $G$. Let $Q^r \subset Q$ be the radius $r$ ball centered at the origin. Then the {\bf stabilization} of ${\mf A}$ by $Q^r$, denoted by ${\rm Stab}_{Q^r} ({\mf A})$, is the $\uds{\bf Kur}_{\rm NCS}^{G, \flat}$-atlas whose charts are 
\beqn
{\rm Stab}_{Q^r} (K_I) = (U_I \times Q^r, E_I \oplus Q, S_I \times {\rm Id}_Q, \hat \Psi_I)
\eeqn
where $Q$ is regarded as a trivial $G$-equivariant Hermitian vector bundle and $\hat \Psi_I$ is naturally induced from $\Psi_I$; the coordinate changes are also canonically induced. 
\end{defn}

If $f: {\mf A} \to Y$ is a strong map, then there is a canonically induced strong map $\hat {\mf f}: {\rm Stab}_{Q^r} ({\mf A}) \to Y$. Suppose ${\mf A}' \subset {\mf A}$ is a shrinking, when $R>r$, ${\rm Stab}_{Q^r}({\mf A}')$ is a shrinking of ${\rm Stab}_{Q^R} ({\mf A})$.

\begin{prop}\label{prop_stabilization}
For any $R>0$ there holds
\beqn
[{\mf A}, {\mf A}']_{\rm main}^{\rm vir} = [{\rm Stab}_{Q^R} ({\mf A}), {\rm Stab}_{Q^r} ({\mf A}')]_{\rm main}^{\rm vir}.
\eeqn
\end{prop}

\begin{proof}
The stabilization of NCS transverse multisections is still NCS transverse. As a result, one has NCS transverse multisection perturbations on both ${\mf A}$ and the stabilization with identical zero locus. 
\end{proof}

\section{Reduced Gromov--Witten Invariants}\label{section5}

We use the previously developed abstract framework to define reduced Gromov--Witten invariants for all genus and general compact symplectic manifolds. Let $(X, \omega)$ be a compact symplectic manifold and $J$ be an $\omega$-compatible almost complex structure. For $A \in H_2(X; {\mb Z})$ and $g, n \geq 0$, let $\ov{\mc M}{}_{g,n}(X, J, A)$ be the moduli space of genus $g$, $J$-holomorphic stable maps with $n$ marked points in degree $A$. We know that $\ov{\mc M}{}_{g,n}(X, J, A)$ is a compact orbispace. Moreover, there are the evaluation map
\beqn
\ev: \ov{\mc M}{}_{g,n}(X, J, A) \to X^n
\eeqn
and the stabilization map
\beqn
{\rm st}: \ov{\mc M}{}_{g,n}(X, J, A) \to \ov{\mc M}{}_{g,n}.
\eeqn
The associated Gromov--Witten invariant is defined via constructing a virtual fundamental class supported on $\ov{\mc M}{}_{g,n}(X, J, A)$ (or supported in a ``neighborhood'' of it, depending on the regularization scheme), whose push-forward via the above two maps is a well-defined homology class
\beqn
[\ov{\mc M}{}_{g,n}(X, J, A)]^{\rm vir} \in H_*( X^n \times \ov{\mc M}{}_{g,n}; {\mb Q}).
\eeqn
By constructing virtual fundamental chains on parametrized moduli spaces one can prove that the above homology class is independent of the almost complex structure $J$ and only depends on the deformation class of $\omega$.

\begin{rem}
There are many ways to define (symplectic) Gromov--Witten invariants by regularizing the generally singular space $\ov{\mc M}{}_{g,n}(X, J, A)$. When $(M, \omega)$ is semi-positive, Ruan \cite{Ruan_96} and Ruan--Tian \cite{Ruan_Tian, Ruan_Tian_97} used geometric perturbations to regularize the moduli spaces and defined Gromov--Witten invariants in this case. When $(X,\omega)$ is rational, Cieliebak--Mohnke \cite{Cieliebak_Mohnke} used the stabilizing divisor technique to define GW invariants in genus zero (extended by \cite{Gerstenberger_2013} to higher genus). For general symplectic manifolds, Li--Tian \cite{Li_Tian}, Fukaya--Ono \cite{Fukaya_Ono}, Ruan \cite{Ruan_1999}, Siebert \cite{Siebert_virtual} defined GW invariants using different versions of the virtual technique. The virtual technique was further developed by Hofer--Wysocki--Zehnder \cite{HWZ_polyfold, HWZ-GW}, Pardon \cite{Pardon_virtual}, Abouzaid--McLean--Smith \cite{AMS, AMS2}, Hirschi--Swaminathan \cite{Hirschi_Swaminathan_2024}, leading to alternate definitions of GW invariants. 
\end{rem}

The technique we will use to define the reduced GW invariants combines the theory of Kuranishi structures developed by Fukaya--Ono \cite{Fukaya_Ono} (extended by Fukaya--Oh--Ohta--Ono \cite{FOOO_Book, FOOO_Kuranishi}) and the theory of global Kuranishi charts developed by Abouzaid--McLean--Smith \cite{AMS, AMS2} and Hirschi--Swaminathan \cite{Hirschi_Swaminathan_2024}. We state the main result below.

\begin{thm}\label{thm52}
Fix $g, n, A$.  
\begin{enumerate}

\item For each $\omega$-compatible almost complex structure $J$, there exist a compact Lie group $G$, a $\uds{\bf Kur}_{\rm NCS}^{G, \flat}$-atlas ${\mf A}$ on $\ov{\mc M}_{g,n}(X, J, A)$, a continuous strong map ${\mf f}: {\mf A} \to \ov{\mc M}_{g,n}\times X^n$ (which extends the evaluation and the stabilization maps on the moduli space), and a class of shrinkings ${\mf A}'$ of ${\mf A}$ so that one can define the virtual class
\beqn
[{\mf A}, {\mf A}']_{\rm main}^{\rm vir} \in H_* (\ov{\mc M}_{g, n}\times X^n; {\mb Q}).
\eeqn

\item The virtual class of the main component only depends on $g, n, A$ and the symplectic deformation class of $\omega$. We denote this class by 
\beqn
[\ov{\mc M}{}_{g,n}(X, J, A)]^{\rm red} \in H_*(\ov{\mc M}_{g,n} \times X^n; {\mb Q}).
\eeqn

\item When $g = 0$, $[\ov{\mc M}_{0, n}(X, J, A)]^{\rm vir}_{\rm main} = [\ov{\mc M}_{0, n}(X, J, A)]^{\rm vir}$.
\end{enumerate}
\end{thm}

\subsection{The NCS structure in a local chart}

To help understanding the construction in the general situation, we explain how to build the expected stratification on the Kuranishi charts in a greatly simplified situation. We also recall the original local gluing construction, especially how obstruction spaces arise from individual components, which inspired the current work. 

For simplicity, assume that $n = 0$. Choose a point $p \in \ov{\mc M}_{g,0}(X, J, A)$ and a representative 
\beqn
u_p: C_p \to X
\eeqn
where $C_p$ is a prestable curve. 

In the typical approaches such as \cite{Fukaya_Ono} and \cite{pardon-VFC}, the general local construction requires stabilizing the domain. Here for simplicity, assume that $C_p$ is stable to the stabilizing process is unnecessary. We also assume that not only the automorphism group of $u_p$, but also the automorphism group of $C_p$, is trivial. Then choose a universal unfolding 
\beqn
C \to \Delta
\eeqn
where $\Delta \subset {\mb C}^n$ is a small ball centered at $0$ and $C$ is a family of (stable) curves with the central fibre identified with $C_p$. Then $\Delta$ is homeomorphic to a neighborhood of the isomorphism class of $C_p$ in the Deligne--Mumford space $\ov{\mc M}_{g,0}$. Then one can consider the set
\beqn
{\mc M}(C/\Delta, X)
\eeqn
consisting of pairs $(\phi, u)$ where $\phi \in \Delta$ and $u: C_\phi \to X$ is a $J$-holomorphic map. It contains the representative $u_p: C_p \to X$. Then one can prove that ${\mc M}(C/\Delta, X)$ contains an open neighborhood of $u_p$ which parametrizes an open neighborhood of $p$ in $\ov{\mc M}_{g,n}(X, J, A)$ with respect to the Gromov topology. We need then to put ${\mc M}(C/\Delta, X)$ into a regular space.

Let $\mathring C \subset C$ be the complement of nodes, which is a smooth manifold. Then there is a vector bundle 
\beqn
\Lambda^{0,1}_{\mathring C/\Delta} \to \mathring C.
\eeqn
of vertical $(0,1)$-forms. A local Kuranishi chart can be obtained by choosing a ``thickening datum'' on the family $C \to \Delta$, which is a bundle map
\beqn
\lambda_p: \mathring C \times X \times W_p \to  \pi_{\mathring C \times X \to X}^* \Lambda_{\mathring C/\Delta}^{0,1} \otimes \pi_{\mathring C\times X \to X}^* TX
\eeqn
over $\mathring C \times X$, where $W_p$ is a finite-dimensional vector space. We say that $\lambda_p$ is transverse to $u_p$ if its restriction to the graph of $u_p$, which is a linear map 
\beqn
W_p \to \Omega^{0,1}(\mathring C_p, u_p^* TX)
\eeqn
is transverse to the image of the linearization of the holomorphic map equation, denoted by $D_p$.

Given a transverse thickening datum, one can write down the corresponding thickened moduli space
\beqn
U_p:= \{ (\phi, u, e)\ |\ \phi \in \Delta, u: C_\phi \to X,\ e \in W,\ \ov\partial_J u + \lambda_p(e) = 0 \}.
\eeqn
Via gluing, the transversality condition implies that $U_p$ is a topological manifold. Moreover, there is the ``Kuranishi section''
\beqn
S_p: U_p \to W_p,\ (\phi, u, e) \mapsto e.
\eeqn
The zero locus then ${\mc M}(C/\Delta, X)$. The  ``footprint map'' is the composition
\beqn
\psi_p: S_p^{-1}(0) \to {\mc M}(C/\Delta, X) \to \ov{\mc M}_{g,n}(X, J, A).
\eeqn
Then the quadruple $(U_p, W_p, S_p, \psi_p)$ is a (topological) Kuranishi chart. 

The way to obtain a stratified Kuranishi chart is to choose a thickening datum coming from each individual component. This kind of thickening data are actually the most obvious ones and were used in the early works of constructing virtual cycles. We explain how one can achieve so. Let $C_v \subset C_p$ denote an irreducible component and $\mathring C_v \subset C_v$ the complements of nodes. Let $u_v: C_v \to X$ be the restriction. The linearization of $u_v$ is a Fredholm operator
\beqn
D_v: \Omega^0(C_v, u_v^* TX) \to \Omega^{0,1}(C_v, u_v^* TX).
\eeqn
Moreover, let $\mathring \Omega^0(C_v, u_v^* TX)$ be the subspace of infinitesimal deformations which vanish at nodes and let $\mathring D_v$ be the restriction to this subspace which is still Fredholm. Then for each $v$, choose a finite-dimensional vector space $W_v$ and a bundle map
\beqn
\mathring C_v \times X \times W_v \to \Lambda^{0,1} T^* \mathring C_v \otimes TX
\eeqn
over $\mathring C_v \times X$ such that its restriction to $u_v$ is transverse to the image of $\mathring D_v$. Define $W_p$ to be the direct sum of all $W_v$. Then one obtains a bundle map 
\beqn
\lambda_p: \mathring C_p\times X \times W_p \to \Lambda^{0,1} T^* \mathring C_p \otimes TX.
\eeqn

Further, one can extend the bundle map to the universal curve $\mathring C$ in the following way. Let $\check C_p \subset \mathring C_p$ be a compact subset such that the image of $\lambda_p$ are all supported within $\check C_p$. Then one can find a smooth map
\beqn
\check C_p \times \Delta \to \mathring C
\eeqn
which sends $\check C_p \times \{\phi\}$ diffeomorphically onto a subset $\check C_\phi \subset \mathring C_\phi$ whose central fibre restriction is the inclusion $\check C_p \hookrightarrow \mathring C_p$. Then $\lambda_p$ can be translated to all nearby fibres. This provides a thickening datum on $C\to \Delta$. We call such a thickening datum {\bf component-refined}; indeed, for each $\phi \in \Delta$, there is a canonical decomposition
\beqn
W_p = \bigoplus_{v \in {\rm Irre}C_\phi} W_{\phi, v}
\eeqn
where $W_{\phi, v}$ is the part which only perturb the component $C_{\phi, v}$, where other summands perturb other corresponding components. 

On a Kuranishi chart constructed using a component-refined thickening datum, one can define a stratification as follows. For simplicity, we assume that $C_p$ has only two components connected by one node, an effective component $C_p^{\rm eff}$ and a ghost component $C_p^{\rm gho}$. The obstruction space $W_p$ is then decomposed as the direct sum $W_p = W_p^{\rm eff} \oplus W_p^{\rm gho}$. The manifold $U_p$ can be stratified into three strata described as follows.
\begin{enumerate}

\item The lowest one consists of triples $(\phi, u, e)$ with $C_\phi$ having the same combinatorial type, i.e., $C_\phi = C_\phi^{\rm eff} \cup C_\phi^{\rm gho}$, and $e^{\rm gho} = 0$ (which implies that $u|_{C_\phi^{\rm gho}}$ is a constant map).

\item The intermediate stratum consists of triples $(\phi, u, e)$ with $C_\phi$ having the same combinatorial type, and $e_{\phi}^{\rm gho} \neq 0$ (which does not necessarily imply that $u|_{C_\phi^{\rm gho}}$ is nonconstant).

\item The top stratum consists of $(\phi, u, e)$ with $C_\phi$ being a smooth surface.
\end{enumerate}
The non-obvious claim that the lowest stratum is a manifold is because the (obstructed) ghost component, is cut off cleanly.

Meanwhile, the obstruction bundle is also stratified. In the current simplified setting, only fibres over the lowest stratum are stratified in a nontrivial way. More precisely, at $(\phi, u, e)$ in the lowest stratum, $W_p$ has a nontrivial stratum being $W_p^{\rm eff}$. Then $W_p \to U_p$ becomes a stratified vector bundle in the sense of Definition \ref{defn_stratified_vector_bundle}.

There is also a natural cosheaf morphism ${\mc O}^{U_p} \to {\mc O}^{W_p}$, which is only nontrivial on the lowest stratum. As the lowest stratum is defined by where $e^{\rm gho}= 0$, the Kuranishi section $S_p: U_p \to W_p$ is obviously stratified. 

Lastly, the manifold $U_p$ has a natural NCS structure. Indeed, the normal bundle to each lower stratum is either given by the gluing parameter, or the subbundle $W_p^{\rm gho}$, or their direct sum which are all complex. The obstruction bundle $W_p$ is also complex by our choice.

\subsection{General AMS construction}

The proof of Theorem \ref{thm52} is based on a modification of the construction of Abouzaid--McLean--Smith \cite{AMS2}. We first recall the original AMS construction for higher genus GW moduli space. A parallel treatment is provided in \cite{Hirschi_Swaminathan_2024}.

\subsubsection{Curves in projective spaces}

We review the construction of \cite[Section 4.2]{AMS2}. Fix $g,n \geq 0$ and $d \geq g$. Consider the moduli space 
\beqn
\ov{\mc M}{}_{g,n}(\mb{CP}^{d-g}, d)
\eeqn
of $n$ marked genus $g$ stable maps into $\mb{CP}^{d-g}$ of degree $d$. Define
\beqn
B_{g, n, d}\subset \ov{\mc M}{}_{g,n}(\mb{CP}^{d-g}, d)
\eeqn
to be the subset of curves $\phi: \Sigma \to \mb{CP}^{d-g}$ such that $H^1(\phi^*{\mc O}(1)) = 0$ and such that $\phi$ has a trivial automorphism group. Then $B_{g,n,d}$ is a smooth quasiprojective variety with a holomorphic $U(d-g+1)$-action. Let 
\beqn
C_{g, n, d} \to B_{g, n, d}
\eeqn
be the universal curve. For $\phi \in B_{g, n, d}$, denote by $C_\phi$ the fibre over $\phi$. When $g, n, d$ are fixed in the context, abbreviate $B_{g, n,d} = B$, $C_{g, n, d} = C$, $U(d-g+1) = G$, and $GL(d-g+1) = G^{\mb C}$. 

The manifold $B_{g, n, d}$ is stratified by combinatorial type of curves. 

\begin{defn}[Map Types]\label{defn_map_type}
A {\bf map type} consists of a graph $\tilde \Gamma$ with $n$ tails (labelling marked points) with a genus function and a degree function
\begin{align*}
&\ {\rm genus}: {\rm Vert}(\tilde \Gamma) \to {\mb Z}_{\geq 0},\ &\ {\rm degree}: {\rm Vert}(\tilde\Gamma) \to {\mb Z}_{\geq 0}
\end{align*}
(here ${\rm genus}(v)$ is regarded as the genus of the normalization of the component). The total genus of $\tilde \Gamma$ is 
\beqn
g = \sum_{v\in {\rm Vert}( \tilde \Gamma)} {\rm genus} (v) + {\rm dim} H_1(|\tilde\Gamma|)
\eeqn
(here $|\tilde\Gamma|$ is regarded as a 1-dimensional complex) and the total degree is 
\beqn
d = \sum_{v \in {\rm Vert}( \tilde \Gamma)} {\rm degree} (v).
\eeqn
We require that ${\rm degree} (v) \geq {\rm genus} (v)$ and for each unstable vertex $v \in {\rm Vert}( \tilde \Gamma)$, ${\rm degree} (v) > 0$. A {\bf morphism} of map types from $\tilde\Gamma$ to $\tilde \Delta$ is a graph map $f: \tilde \Gamma \to \tilde \Delta$, which is surjective on the set of vertices, bijective on the set of tails, such that for each $w \in {\rm Vert}(\tilde\Delta)$, there holds
\beqn
\sum_{v \in {\rm Vert}(f^{-1}(w))} {\rm genus} (v) + {\rm dim} H^1(|f^{-1}(v)|) = {\rm genus} (w) + {\rm dim} H^1( |\tilde\Delta_w|)
\eeqn
and 
\beqn\sum_{v \in {\rm Vert} (f^{-1}(w))} {\rm degree} (v) = {\rm degree} (w).
\eeqn
Here $f^{-1}(w)$ is regarded as a subgraph of $\tilde\Gamma$ and $\tilde\Delta_w$ is a subgraph of $\tilde\Delta$ (containing all self-connecting edges).  
\end{defn}

We use $\Gamma$ to denote an isomorphism class of map types. There are a natural partial order among isomorphism classes of map types. 

\begin{defn}\label{defn_partial_order_realization}
Let $\Gamma$ and $\Delta$ be isomorphism classes of map types. We define $\Gamma \leq \Delta$ if there exist representatives $\tilde\Gamma$ and $\tilde\Delta$ of them respectively, and a morphism $f: \tilde \Gamma \to \tilde \Delta$. We say such a morphism $f$ {\bf realizes} the partial order relation $\Gamma \leq \Delta$.
\end{defn}

For each map $\phi \in B$, there is a corresponding map type $\tilde \Gamma_\phi$ whose vertices correspond to irreducible components of $C_\phi$, whose edges correspond to nodes, whose tails correspond to marked points, and whose degree and genus function are given by the degree and genus of corresponding component. Let 
\beqn
B_{\Gamma}^* \subset B
\eeqn
be the locus of curves whose underlying map type is in the class ${\Gamma}$.

\begin{lemma}\label{base_simple_stratification}
The partition
\beqn
B = \bigsqcup_{\Gamma} B_{\Gamma}^*
\eeqn
makes $B$ a stratified $G$-manifold (see Definition \ref{defn_manifold_stratification}).
\end{lemma}

\begin{proof}
Locally $B_{\Gamma}^*$ is defined by the vanishing of gluing parameters. 
\end{proof}

We would like to look at the structural cosheaf ${\mc O}^B$ (see Definition \ref{defn_structural_cosheaf}) locally. The following lemma is obvious.

\begin{lemma}\label{lemma56}
For each map type $\Gamma$ and $\phi \in B_\Gamma^*$, there exists an open neighborhood $O(\phi) \subset B$ such that for each $\psi \in B(\phi)$, there is a canonical morphism of map types $\tilde\Gamma_\phi \to \tilde\Gamma_\psi$ which realizes the partial order $\Gamma \leq \Gamma_\psi$ (where the latter is the isomorphism class of $\tilde\Gamma_\psi$).
\end{lemma}

We then define
\beqn
C_{\Gamma}^* \subset C
\eeqn
be the fibres of the map $C \to B$ over $B_{\Gamma}^*$. Each $C_{\Gamma}^*$ is not a submanifold at nodal points; later we will consider a refinement which provides a stratified manifold structure on the universal curve $C$.

\subsubsection{Framings and group reductions}

\begin{defn}\cite[Definition 4.23]{AMS2}
Fix $A\in H_2(X; {\mb Z})$. A {\bf line bundle datum} is a triple 
\beqn
{\mb L} = (L, k, {\mf D})
\eeqn
where
\begin{enumerate}
\item $L \to X$ is a Hermitian line bundle with a Hermitian connection whose curvature is $-2\pi {\bf i} \Omega$ with $\Omega$ a symplectic form taming $J$. We also require that $L$ admits a root of order at least $3$.

\item $k>0$ is an integer. Denote
\beq\label{degree_formula}
d = k(\Omega(A) + 2g-2+n).
\eeq

\item ${\mf D}$ is a consistent domain metric for $B$ (see \cite[Definition 4.9]{AMS2}).
\end{enumerate}
\end{defn}

Upon choosing a line bundle datum ${\mb L}$, one can introduce the notion of frames. For $\phi \in B$ (which has marked points $p_1, \ldots, p_n \in C_\phi$) and a smooth map $u: C_\phi \to X$, define 
\beqn
{\mb L}_u: = \Big( \omega_{C/B}(p_1, \ldots, p_n)|_\phi \otimes u^* L \Big)^{\otimes k}.
\eeqn
The connection on $L \to X$ induces a holomorphic structure on ${\mb L}_u$. If $u$ represents the curve class $A$, then ${\rm deg}({\mb L}_u) = d$ where $d$ is given by the formula \eqref{degree_formula}. A {\bf holomorphic ${\mb L}$-framing} on $u$ is a basis 
\beqn
F = (f_0, \ldots, f_{d-g})
\eeqn
of $H^0(L_u)$ for which the $(d-g+1)\times(d-g+1)$  matrix $H_F$ with entry
\beqn
H_{F, ij}:=\int_{C_\phi} \langle f_i, f_j \rangle \Omega_\phi
\eeqn
is positive definite. Here $\Omega_\phi \in \Omega^2(C_\phi)$ is induced from the consistent domain metric ${\mf D}$. An {\bf ${\mb L}$-framed map} is a tuple $(\phi, u, F)$ where $\phi \in B$, $u: C_\phi \to X$ is a smooth map, and $F$ is a holomorphic ${\mb L}$-framing on $u$. Denote
\beqn
{\rm Map}_{\mb L}^{\rm fr}(C/B, X, A):= \left\{ (\phi, u, F)\ \left| \ \begin{array}{ll}(\phi, u, F)\ \text{is an ${\mb L}$-framed map},\ u_*[C_\phi] = A, \\
\forall v\in {\rm Vert}(\tilde\Gamma_\phi),\ {\rm degree}(v) = k \big( \Omega( u(C_{\phi, v})) + 2 {\rm genus}(v) - 2 + n(v) \big)   \end{array} \right.  \right\}.
\eeqn
Here $n(v)$ is the number of nodes on the component $C_{\phi, v} \subset C_\phi$. In other words, the conditions in the above definition require that the map $u$ has the correct topological type class specified by $A$.

Notice that there is an action on ${\rm Map}_{\mb L}^{\rm fr}(C/B, X, A)$ by $G^{\mb C}\times G^{\mb C}$ where the first $G^{\mb C}$ acts on $\phi$ by reparametrizing the map (via the $G^{\mb C}$-action on $B$) and the second $G^{\mb C}$ acts on $F$ by linearly transforming the framing. There is then a map (which is invariant under the first $G^{\mb C}$-action)
\beqn
S_Q: {\rm Map}_{\mb L}^{\rm fr}(C/B, X, A) \to Q,\ S_Q(\phi, u, F) = \log H_F.
\eeqn
Here $Q$ is the space of $(d-g+1)\times (d-g+1)$ Hermitian matrices. 

Moreover, for each holomorphic $\mb{L}$-framing $F$, there is an associated map 
\beqn
\phi_F: C_\phi \to \mb{CP}^d;
\eeqn
together with the marked points represents a point in $B \subset \ov{\mc M}{}_{g,n}(\mb{CP}^d, d-g)$, still denoted by $\phi_F$. Notice that $\phi_F$ and $\phi$ are in the same stratum of $B$.

Now we can have a convenient parametrization of the moduli space. Define
\beqn
{\mc M}_{\mb L}^{\rm fr}(C/B, X, A) = \Big\{(\phi, u, F)\in {\rm Map}_{\mb L}^{\rm fr}(C/B, X, A)\ |\ \phi = \phi_F \in B,\ \ov\partial_J u = 0,\ S_Q(\phi, u, F) = 0 \Big\}.
\eeqn
Notice that the condition $\phi = \phi_F$ is only invariant under the diagonal $G^{\mb C}$-action; the condition $s_Q = 0$ reduces the symmetry to the diagonal $G$-action.  

\begin{lemma}\label{lemma_AMS_parametrization}
There is an isomorphism of orbispaces
\beq\label{parametrization}
{\mc M}_{\mb L}^{\rm fr}(C/B, X, A) /G \cong \ov{\mc M}_{g,n}(X, J, A).
\eeq
\end{lemma}
\begin{proof}
This is essentially \cite[Theorem 4.30]{AMS2}.
\end{proof}

\begin{rem}
For any subset $B_0 \subset B$, one can restrict the family $C \to B$ to $B_0$ and consider the subset
\beqn
{\mc M}_{\mb L}^{\rm fr}(C/B_0, X) \subset {\mc M}_{\mb L}^{\rm fr}(C/B, X).
\eeqn
By the compactness of the moduli space, one can choose a precompact $G$-invariant open subset $B_0 \subset B$ such that 
${\mc M}_{\mb L}^{\rm fr}(C/B_0, X) = {\mc M}_{\mb L}^{\rm fr}(C/B, X)$.
\end{rem}

\subsubsection{Thickening data}

We perturb the Cauchy--Riemann  equation by adding finite-dimensional inhomogeneous terms.
Consider a more general situation. Let $G$ be a compact Lie group and let $C\to B$ be a $G$-equivariant family of prestable curves. We consider pairs $(\phi, u)$ where $\phi \in B$ and $u: C_\phi \to X$ is a smooth map (meaning that $u$ is smooth away from nodes and extends continuously over nodes). For each such $u$, there is a well-defined linearization of Cauchy--Riemann operator
\beqn
D_u: \Omega^0(C_\phi, u^* TX) \to \Omega^{0,1}(C_\phi, u^* TX).
\eeqn
Here $\Omega^0(C_\phi, u^* TX)$ is the space of sections of $u^* TX \to C_\phi$ which are smooth over components and extend continuously over nodes; $\Omega^{0,1}(C_\phi, u^* TX)$ is the space of smooth sections of $\Lambda^{0,1}\otimes u^* TX$ over the normalization of $C_\phi$. Let 
\beqn
\mathring C \subset C
\eeqn
be the complement of nodal points and marked points. Then there is a well-defined smooth vector bundle
\beqn
\Lambda^{0,1}_{\mathring C/B} \to \mathring C
\eeqn
of fibrewise anti-holomorphic differentials. 

\begin{defn}[Thickening datum]\label{defn_thickening_datum}
A {\bf (local) thickening datum} over the $G$-equivariant family of curve $C \to B$ is a triple
\beqn
\mu = (B_\mu, W_\mu, \lambda_\mu)
\eeqn
where $B_\mu \subset B$ is a $G$-invariant open subset, $W_\mu \to B_\mu$ is a $G$-equivariant flat Hermitian vector bundle, and $\lambda_\mu$ is  a $G$-equivariant bundle map (over $\mathring C_\mu \times X$)
\beqn
\lambda_\mu: (\pi_{\mathring C_\mu \times X \to B_\mu})^* W_\mu \to (\pi_{\mathring C_\mu \times X \to \mathring C_\mu})^* \Lambda^{0,1}_{\mathring C_\mu /B_\mu } \otimes (\pi_{\mathring C_\mu \times X \to X})^* TX
\eeqn
satisfying the following condition: there exists a $G$-invariant open neighborhood of nodal points of $C_\mu$ such that $\lambda_\mu$ vanishes in that neighborhood. Then for each $\phi \in B_\mu$ and a smooth map $u: C_\phi \to X$, there is a linear map 
    \beqn
    \lambda_{\mu, u}: W_{\mu, \phi} \to \Omega_c^{0,1}(C_\phi, u^* TX).
    \eeqn
The thickening datum is said to be {\bf transverse} over a $G$-invariant subset $B_\mu'\subset B_\mu$ if for each $(\phi, u, F) \in {\mc M}_{\mb L}^{\rm fr}(C/B_\mu', X, A)$, one has 
\beqn
{\rm Im} D_u + {\rm Im} \lambda_{\mu, u} = \Omega^{0,1}(C_\phi, u^* TX).
\eeqn
A thickening datum $\mu = (B_\mu, W_\mu, \lambda_\mu)$ is said to be {\bf trivial over $B_\mu' \subset B_\mu$} if $\lambda_\mu$ vanishes over $B_\mu'$;  {\bf supported in $B_\mu'  $} if it is trivial over $B_\mu \setminus B_\mu'$.
\end{defn}

Notice that if $\mu_i = (B_{\mu_i}, W_{\mu_i}, \lambda_{\mu_i})$, $i = 1, 2$, are two thickening data, then one can take the direct sum and obtain a thickening datum
\beqn
\mu_1 \oplus \mu_2 = (B_{\mu_1} \cap B_{\mu_2}, W_{\mu_1} \oplus W_{\mu_2}, \lambda_{\mu_1} \oplus \lambda_{\mu_2} ).
\eeqn
On the other hand, if $\rho: B_\mu \to {\mb R}$ is a $G$-invariant function, then 
\beqn
\rho\mu = (B_\mu, W_\mu, \rho \lambda_\mu)
\eeqn
is also a thickening datum.


\subsubsection{The AMS charts}

There is an indirect way to construct a $G$-smoothing (after stabilization) following \cite{AMS} via the approach of equivariant stable smoothing of topological manifolds. In order to maintain the stratification of the Kuranishi charts, we would like to consider a more direct construction.

\begin{defn}\label{defn_AMS_thickening}
Choose a line bundle data ${\mb L} = (L, k, {\mf D})$ and a thickening datum $\mu = (B_\mu, W_\mu, \lambda_\mu)$. The associated {\bf AMS Kuranishi chart} on $\ov{\mc M}{}_{g,n}(X, J, A)$ is the $G$-equivariant Kuranishi chart 
\beqn
K_\mu:= (U_\mu, E_\mu, S_\mu, \Psi_\mu)
\eeqn
where $G= U(d-g+1)$, and 
\begin{enumerate}
    \item $U_\mu$ is the {\bf thickening} consisting of quadruples $(\phi, u, e, F)$ satisfying
    \beqn
    \ov\partial_J u + \lambda_{\mu, u}(e) = 0
    \eeqn
    and $F$ is a holomorphic ${\mb L}$-framing on $u: C_\phi \to X$ satisfying
    \beqn
    \phi = \phi_F \in B.
    \eeqn

    \item $G$ acts on $U_\mu$ by 
    \beqn
    g(\phi, u, e, F):= (g\phi, u \circ g^{-1}, g(e), g(F)).
    \eeqn
    
    \item $E_\mu$ is the direct sum
    \beqn
    E_\mu:= \pi_\mu^* W_\mu \oplus Q
    \eeqn
    where $Q$ is the space of $(d-g+1)\times (d-g+1)$ Hermitian matrices. 
    
    \item $S_\mu$ is the direct sum $ S_{W_\mu} \oplus S_Q$ where 
    \begin{align*}
    &\     S_{W_\mu} (\phi, u, e, F) = e,\ &\     S_Q (\phi, u, e, F) = \log H_F.
    \end{align*}

    \item $\Psi_\mu$ sends $(\phi, u, e, F)$ to the equivalence class of the map $u: C_\phi \to X$ (with markings).
\end{enumerate}
\end{defn}

\subsubsection{Smooth structure}

If the thickening datum $\mu = (B_\mu, W_\mu, \lambda_\mu)$ is transverse, then the basic gluing construction of holomorphic curves implies that $U_\mu$ is a topological manifold with a $G$-action. The following lemma also provides certain canonical structures on the thickening. 

\begin{lemma}\label{lemma_stratum_smooth}
Suppose the thickening datum $\mu = (B_\mu, W_\mu, \lambda_\mu)$ is transverse throughout $B_\mu$, for each isomorphism class of map types $\Gamma$, the preimage
\beqn
U_{\mu, \Gamma}^*:= \pi_\mu^{-1}( B_{\mu, \Gamma}^*)
\eeqn
has a canonical smooth structure such that the map
\beqn
\pi_\mu: U_{\mu,\Gamma}^* \to B_{\mu,\Gamma}^*
\eeqn
is a smooth submersion. Moreover, the restriction of $E_\mu$ to $U_{\mu,\Gamma}^*$ has a canonical smooth structure smooth and the restriction of $S_\mu$ to $U_{\mu,\Gamma}^*$ is smooth.
\end{lemma}

\begin{proof}
Implicit function theorem. \end{proof}

A difficult task in the Kuranishi approach is about the smooth structure on Kuranishi charts. The key issue is the lack of differentiability of the gluing construction with respect to the gluing parameter. A working approach is to reparametrize the gluing parameter in a particular way (see \cite{FOOO_smooth}); such a choice was also taken in the polyfold approach. With respect to this change of coordinate one can establish the smoothness. 

\begin{rem}
Abouzaid--McLean--Smith invented a new way in \cite{AMS} to obtain smooth Kuranishi charts by using the equivariant smoothing theory in topological manifold theory. It does not fit well with the current problem because the topological smoothing may not give smooth strata labelled by combinatorial types.
\end{rem}

\begin{prop}\label{prop_smooth}
There exists a $G$-equivariant smooth structure on $B_\mu$ such that for each isomorphism class of map type $\Gamma$, $U_{\mu, \Gamma}^*$ is a smooth submanifold whose smooth structure coincides with the one of Lemma \ref{lemma_stratum_smooth}. Moreover, the normal bundle $NU_{\mu, \Gamma}^*$ is $G$-equivariantly isomorphic to $\pi_\mu^* NB_{\mu,\Gamma}^*$. 
\end{prop}

\begin{proof}
We define the smooth structure on $U_\mu$ as follows. Recall that each the domain moduli $B$ is a smooth $G$-manifold and for each isomorphism class of map types $\Gamma$, $B_\Gamma^* \subset B$ is a smooth submanifold. Then the normal bundle $NB_\Gamma^* \to B_\Gamma^*$ is a smooth $G$-equivariant vector bundle with a fibrewise stratification. Then by Lemma \ref{lemma_stratum_smooth}, the pullback 
\beqn
\pi_\mu^* NB_\Gamma^* \to U_{\mu, \Gamma}^*
\eeqn
is a smooth $G$-equivariant vector bundle. The total spaces of these normal bundles have canonical smooth structures. Now we would like to define new fibrewise coordinate. Choose a $G$-invariant bundle metric. Locally, the normal bundle splits naturally into line bundles indexed by nodes and we write a vector as $v = (v_1, \ldots, v_m)$. Define the fibrewise reparametrization (with in a small disk bundle)
\beq\label{fibre_reparametrization}
\lambda(x, v_1, \ldots, v_m) = \left( x, - \frac{v_1}{|v_1| \log |v_1|}, \ldots, - \frac{v_m}{|v_m| \log |v_m|} \right)\footnote{This reparametrization is equivalent to using the parameter $s = \frac{1}{T}$ where $T$ is the usual cylindrical gluing parameter as part of local coordinates in \cite{FOOO_smooth}.}
\eeq
which equip the total space a new smooth structure on the disk bundle. 

Then locally one can construct gluing maps whose domains are certain open subset of the disk bundle in the normal bundle. The gluing construction is stratum-preserving. These gluing constructions provide local coordinates on the chart $U_\mu$ near $U_{\mu, \Gamma}^*$ for all $\Gamma$. The main results of \cite{FOOO_smooth} implies that these coordinates are smoothly compatible. In particular, the smooth structure does not depend on choices made for constructing the local gluing maps. Hence $U_\mu$ has a $G$-invariant smooth structure. Moreover, the stratum $U_{\mu, \Gamma}^*$ are still smooth and the smooth structure is identical to the one established using the implicit function theorem (without gluing). 
\end{proof}

\begin{rem}
Another way to construct (non-canonical) smooth structures without applying the fibrewise reparametrization \eqref{fibre_reparametrization} is to construct $G$-equivariant global gluing map for each $\Gamma$. 
\end{rem}

\subsection{Refined AMS construction}

One of the basic idea of defining the reduced GW invariants is to realize that constant maps are only cut off  cleanly. In order to separate the contribution from ghost configurations, one needs to make each type of ghost configurations a stratum of the thickened moduli space. To do this, one needs to choose thickening datum for which we can turn off perturbations on ghost components. 

We begin by making the first choice of the construction. 

\begin{step}\label{step1}
Choose a line bundle datum ${\mb L} = (L, k, {\mf D})$.
\end{step}

\subsubsection{Families of curves parametrizing components}

Consider pairs $(\tilde\Gamma, v)$ where $\tilde\Gamma$ is a map type (a graph) and $v \in {\rm Vert}(\tilde\Gamma)$. Let $\gamma$ be the isomorphism class of $(\tilde\Gamma, v)$, which is also denoted by $[\tilde\Gamma, v]$. Then there is a natural forgetful map $\gamma \mapsto \Gamma$. Notice that given any representative $\tilde\Gamma$ of $\Gamma$, there is a one-to-one correspondence 
\beqn
{\rm Vert}(\tilde\Gamma)/ {\rm Aut}(\tilde\Gamma) \cong \{ \gamma \mapsto \Gamma\}.
\eeqn
In the universal curve there are refined piece
\beqn
\mathring C_{\gamma}^* \subset \mathring C_{\Gamma}^*
\eeqn
which parametrizes points in the smooth part of the component labelled by $v$ in curves of map type $\Gamma$.  

We define a partial order as follows. Indeed, we define $\gamma:= [\tilde\Gamma, v]\leq [\tilde \Delta, w]=:\delta$ if there exists a map $f: \tilde \Gamma \to \tilde\Delta$ realizing the partial order $\Gamma \leq \Delta$ (see Definition \ref{defn_partial_order_realization}) such that $f(v) = w$. In this case, we also say that $f$ realizes the partial order $\gamma\leq\delta$.

Given any $\gamma \mapsto \Gamma$, one also defines
\beqn
{\esc B}_{\gamma}^*:= \Big\{(\phi, v)\ |\ \phi \in B_{\Gamma}^*, v \in {\rm Vert} (\tilde \Gamma_\phi),\ [\tilde\Gamma_\phi, v] = \gamma \Big\}
\eeqn 
which is a smooth $G^{\mb C}$-manifold. Take the disjoint union 
\beqn
{\ms B}:= \bigsqcup_{\gamma} {\ms B}_\gamma^*
\eeqn
which has a non-Hausdorff topology, as in degeneration a curve component can split into multiple ones. However, each point of ${\ms B}$ still has an open neighborhood which is a complex manifold. Let 
\beqn
\pi_{{\ms B}\to B}: {\ms B} \to B
\eeqn
be the natural forgetful map, whose restriction to each stratum ${\ms B}_\gamma^*$ is a topological covering onto $B_\Gamma^*$.

We would like to specify certain neighborhoods of ${\ms B}_\gamma^*$. For each isomorphism class of map types $\Gamma$, choose a family of $G$-invariant neighborhoods 
\beqn
B_\Gamma^\epsilon \subset B
\eeqn
of $B_\Gamma^*$ satisfying
\begin{align*}
&\ \epsilon_1 < \epsilon_2 \Longrightarrow B_\Gamma^{\epsilon_1} \subset B_\Gamma^{\epsilon_2},\ &\ \bigcap_\epsilon B_\Gamma^\epsilon = B_\Gamma^*.
\end{align*}
The preimage $\pi_{\ms B}^{-1}(B_\Gamma^\epsilon) \subset {\ms B}$ is stratified by $\delta\geq \gamma$, but it is not a Hausdorff space in general. We describe a smooth resolution of this preimage. 

Recall that by Lemma \ref{lemma56}, for each $\phi \in B_\Gamma^*$, there exists a sufficiently small $G_\phi$-invariant neighborhood $O(\phi) \subset B$ such that for all $\psi \in O(\phi)$ there is a canonical morphism $\tilde\Gamma_\phi \to \tilde\Gamma_\psi$ of their underlying map types. We choose for each $\phi$ such a neighborhood and denote temporarily by $O(\phi)$. Define
\beqn
O^\epsilon(\phi) = O(\phi) \cap B_\Gamma^\epsilon.
\eeqn
Then for $(\phi, v) \in {\ms B}_\gamma^*$, one can choose a $G$-invariant open neighborhood
\beqn
O(\phi, v) \subset \pi_{{\ms B} \to B}^{-1}(O(\phi)) \subset {\ms B}
\eeqn
of $(\phi, v)$ with the projection $O(\phi, v) \to O(\phi)$ being 1-1. We can choose these neighborhood sufficiently small such that when $O(\phi_1, v_1) \cap O(\phi_2, v_2) \neq \emptyset$, there is a canonical isomorphism $\tilde\Gamma_{\phi_1} \cong \tilde\Gamma_{\phi_2}$ which sends $v_1$ to $v_2$. 

\begin{defn}
For $\epsilon$ sufficiently small, let ${\ms B}_\gamma^\epsilon$ be the set
\beqn
\Big\{ ((\phi, v); (\psi, w))\ |\ (\phi, v) \in {\ms B}_\gamma^*, (\psi, w)  \in O(\phi, v),\ \psi \in O^\epsilon(\phi) \Big\}/ \sim
\eeqn
where the equivalence relation $\sim$ is defined as follows. We define
\beqn
((\phi_1, v_1); (\psi, w)) \sim ((\phi_2, v_2); (\psi, w)) \Longleftrightarrow O(\phi_1, v_1) \cap O(\phi_2, v_2) \neq \emptyset.
\eeqn
\end{defn}

In other words, ${\ms B}_\gamma^\epsilon$ parametrizes  curves $C_\psi$ near $C_\phi$ together with the component where the component $C_{\phi, v}$ is glued into. Notice that the composition 
\beqn
\pi_\gamma: {\ms B}_\gamma^\epsilon \to \pi_{\ms B}^{-1}(B_\Gamma^\epsilon) \to B_\Gamma^\epsilon
\eeqn
is a $G$-equivariant topological covering. One also have a family of curves 
\beqn
{\ms C}_\gamma^\epsilon \to {\ms B}_\gamma^\epsilon
\eeqn
whose fibres are not having the same genus. In fact, if $(\phi, v) \in {\ms B}_\gamma^*$, then the fibre is the component ${\ms C}_{\phi, v}$; but if $(\psi, w)$ is a nearby point which is less degenerate, then the fibre ${\ms C}_{\psi, w}$ is perhaps a ``larger'' curve because the component is glued from ${\ms C}_{\phi, v}$ and perhaps other component. However, if we take the complement of nodes and markings, which is denoted by 
\beqn
\mathring {\ms C}_\gamma^\epsilon \to {\ms B}_\gamma^\epsilon
\eeqn
then it is a $G$-equivariant holomorphic submersion.

\subsubsection{Component-wise thickening datum}

\begin{defn}\label{defn_gamma_thickening}
Let $\gamma$ be an isomorphism class of pairs. A {\bf $\gamma$-thickening datum} is a thickening datum 
\beqn
\upmu_\gamma = ({\ms O}_\gamma, {\ms W}_\gamma, \uplambda_\gamma)
\eeqn
on the family ${\ms C}_\gamma^\epsilon \to {\ms B}_\gamma^\epsilon$ for some $\epsilon>0$ where ${\ms O}_\gamma \subset {\ms B}_\gamma^\epsilon$ is the preimage of a $G$-invariant open subset $O_\Gamma \subset B_\Gamma^\epsilon$ under the map ${\ms B}_\gamma^\epsilon \to B_\Gamma^\epsilon$. We require ${\ms W}_\gamma \to {\ms O}_\gamma$ to be a $G$-equivariant flat Hermitian bundle.
\end{defn}

The transversality for $\gamma$-thickening datum is defined as follows. Suppose $(\phi, u, F) \in {\mc M}_{\mb L}^{\rm fr}(C/B_\Gamma^*, X, A)$. Then for each $(\phi, v) \in {\ms B}_\gamma^*$, $u$ restricts to a map $u_v: C_{\phi, v} \to X$. Then a $\gamma$-thickening datum $\upmu_\gamma = ({\ms O}_\gamma, {\ms W}_\gamma, \uplambda_\gamma)$ with $(\phi, v) \in {\ms O}_\gamma$ induces a linear map
\beqn
\uplambda_{\gamma, u_v}: {\ms W}_\gamma|_\phi \to  \Omega_c^{0,1}(\mathring C_{\phi, v}, u_v^* TX).
\eeqn
We say that $\upmu_\gamma$ is {\bf transverse at $(\phi, v)$} if for all $(\phi, u, F) \in {\mc M}_{\mb L}^{\rm fr}(C/B_\Gamma^*, X, A)$, the image of $\uplambda_{\gamma, u_v}$ is transverse to the linearized Cauchy--Riemann operator at $u_v$ restricted to the subspace
\beqn
\mathring \Omega^0(C_{\phi, v}, u_v^* TX) \subset \Omega^0(C_{\phi, v}, u_v^* TX)
\eeqn
of infinitesimal deformations which vanish at nodes and markings. We say $\upmu_\gamma$ is {\bf transverse} if it is transverse at $(\phi, v)$ for all $(\phi,v) \in {\ms O}_\gamma$.

A $\gamma$-thickening datum  induces a ``pushforward'' $(\pi_\gamma)_! \upmu_\gamma$ over $O_\Gamma$ described as follows. Define a vector bundle $W_\gamma \to O_\Gamma$ whose fibre over $\phi \in O_\Gamma$ is
\beqn
W_\gamma|_\phi:= \bigoplus_{\pi_{{\ms B} \to B} (\phi, v) = \phi}  {\ms W}_\gamma|_{(\phi, v)}.
\eeqn
One can see readily that since ${\ms O}_\gamma \to O_\Gamma$ is a covering, $W_\gamma$ is a $G$-equivariant vector bundle; moreover, since ${\ms W}_\gamma$ is flat, $W_\gamma$ is also flat.\footnote{However, ${\ms W}_\gamma$ being trivial does not imply that $W_\gamma$ is trivial. This is the reason why we consider flat bundles.} Then the thickening is naturally induced from $\uplambda_\gamma$ as one has natural inclusions $\mathring {\ms C}_{(\phi, v)} \hookrightarrow \mathring C_\phi$. Because ${\ms W}_\gamma$ is flat, $W_\gamma$ is also flat.


A $\gamma$-thickening datum also induces a $\delta$-thickening datum for any $\delta \geq \gamma$. Indeed, let $\Delta \geq \Gamma$ be the isomorphism class of map types underlying $\delta$. Denote
\beqn
{\ms O}_\delta^\epsilon:= \pi_\delta^{-1}( O_\Gamma \cap B_{\Delta}^\epsilon).
\eeqn
Notice that one has the commutative diagram
\beqn
\xymatrix{ {\ms O}_\gamma \ar[r] \ar[rd]_{\pi_\gamma} &   {\ms O}_\delta^\epsilon \ar[d]^{\pi_\delta} \\
&  B }.
\eeqn
Then the bundle $ {\ms W}_\gamma \to {\ms O}_\gamma$ can also be pushed forward to ${\ms O}_\delta^\epsilon$, hence providing a $\delta$-thickening datum over ${\ms O}_\delta^\epsilon$. The $G$-bundle over ${\ms O}_\delta^\epsilon$ is also flat.

\begin{defn}\label{defn_component_refined_thickening_data}
Let $\Gamma$ be an isomorphism class of map types and $O_\Gamma \subset B_\Gamma^\epsilon$ be a $G$-invariant open subset. A {\bf component-refined thickening datum} over $O_\Gamma$ is a collection of $\gamma$-thickening datum $\upmu_\gamma$ over ${\ms O}_\gamma:= \pi_\gamma^{-1}(O_\Gamma)$ for all $\gamma \mapsto \Gamma$. It is called {\bf component-wise transverse} over $O_\Gamma$ if each $\upmu_\gamma$ is transverse at all $(\phi, v) \in {\ms O}_\gamma$.
\end{defn}

Given a component-refined thickening datum $\{  \upmu_\gamma\ |\ \gamma \mapsto \Gamma\}$ over $O_\Gamma$ by taking the direct sum of the pushforward, namely,
\beqn
\mu_\Gamma:= \bigoplus_{\gamma \mapsto \Gamma} (\pi_\gamma)_{!}  \upmu_\gamma.
\eeqn
By abuse of notations, we also call the pushforward thickening datum $\mu_\Gamma$ a component-refined thickening datum on $O_\Gamma$. Notice that the component-wise transversality implies the ordinary transversality. When using a component-refined thickening datum to thicken the moduli space, one is able to specify the perturbation on a particular component.

\subsubsection{Constructing thickening data}

\begin{prop}\label{prop_thickening_invariance}
For each $G$-invariant compact subset $Z_\Gamma \subset B_\Gamma^*$, there exists a component-refined thickening data $\mu_\Gamma$ (see Definition \ref{defn_component_refined_thickening_data}) over a neighborhood of $Z_\Gamma$ in $B$ which is component-wise transverse near $Z_\Gamma$. 
\end{prop}

\begin{proof}
Recall that one has the parametrization \eqref{parametrization} of the relevant moduli space. Consider 
\beqn
{\mc M}_{\mb L}^{\rm fr}(C/Z_\Gamma, X, A) \subset {\mc M}_{\mb L}^{\rm fr}(C/B, X, A)
\eeqn
which is compact. For each $\gamma \mapsto \Gamma$, denote by 
\beqn
{\ms Z}_\gamma \subset {\ms B}_\gamma^*
\eeqn
be the covering of $Z_\Gamma$. Then one also has a corresponding moduli 
\beqn
{\mc M}_{\mb L}^{\rm fr}( {\ms C}_\gamma^* / {\ms Z}_\gamma, X, A)
\eeqn
whose elements are $(\phi, v, u)$ where $\phi \in Z_\Gamma$, $(\phi, v) \in {\ms Z}_\gamma$, and $u: C_{\phi, v} \to X$ is a $J$-holomorphic map which coincides with the restriction of an element in ${\mc M}_{\mb L}^{\rm fr}(C/Z_\Gamma, X, A)$. Hence there is a $G$-equivariant continuous surjective map
\beqn
{\mc M}_{\mb L}^{\rm fr}(C/Z_\Gamma, X, A) \to {\mc M}_{\mb L}^{\rm fr}({\ms C}_\gamma^* / {\ms Z}_\gamma, X, A)
\eeqn
hence the codomain is compact.

We first consider the Fredholm theory. For each $\gamma \mapsto \Gamma$, there is a well-defined Banach manifold
\beqn
\esc{B}_{\gamma}^*(X) = \Big\{ (\phi, v, u)\ |\ (\phi, v) \in \esc{B}_{\gamma}^*, u \in W^{1,p}(\mathring {\ms C}_{\phi, v}, X)\Big\}\ \ (\text{for a fixed $p>2$})
\eeqn
which has a smooth $G$-action and a smooth $G$-equivariant projection $\esc{B}_{ \gamma}^*(X) \to \esc{B}_{ \gamma}^*$. Then there is an obvious inclusion
\beqn
{\mc M}_{\mb L}^{\rm fr}({\ms C}_\gamma^*/{\ms Z}_\gamma, X, A) \subset {\ms B}_\gamma^*(X)
\eeqn
For each $(\phi, v, u) \in \esc{B}_{ \gamma}^*(X)$, let 
\beqn
T_{(\phi,v, u)}^{\rm vt} \esc{B}_{ \gamma}^* (X) \subset T_{(\phi,v, u)} \esc{B}_{ \gamma}^*(X)
\eeqn
be the vertical tangent bundle which only parametrizes deformations of $u$. The Cauchy--Riemann operator induces a smooth Fredholm section
\beqn
\esc{F}_{ \gamma}: \esc{B}_{ \gamma}^*(X) \to \esc{E}_{ \gamma}^*(X)
\eeqn
where $\esc{E}_{ \gamma}^*(X)$ has fibre over $(\phi, v, u)$ being
\beqn
\esc{E}_{ \gamma}^*(X)|_{(\phi, v, u)} = L^p( C_{\phi, v}, \Lambda^{0,1} \otimes u^* TX).
\eeqn
For each $(\phi, v, u) \in {\ms F}_{\gamma}^{-1}(0)$, the vertical differential of ${\ms F}_{ \gamma}$, denoted by
\beqn
D^{\rm vt}_{(\phi, v, u)}: T^{\rm vt}_{(\phi,v, u)} \esc{B}_{ \gamma}^*(X) \to \esc{E}_{ \gamma}^*|_{(\phi,v, u)}
\eeqn
is Fredholm. Consider the subbundle
\beqn
\mathring T^{\rm vt} \esc{B}_{ \gamma}^* (X) \subset T^{\rm vt} \esc{B}_{\gamma}^*(X)
\eeqn
consisting of infinitesimal deformations of $u: \mathring C_{\phi, v} \to X$ which vanish at marked and nodal points. Denote
\beqn
\mathring D_{(\phi, v, u)}^{\rm vt}:=D_{(\phi, v, u)}^{\rm vt}|_{\mathring T_{(\phi, v, u)}^{\rm vt} \esc{B}_{\gamma}^*(X)}.
\eeqn
Then it is still a Fredholm operator. 

We first find transverse thickening data locally. Fix $(\phi, v, u) \in {\mc M}_{\mb L}^{\rm fr}({\ms C}_\gamma^* /{\ms Z}_\gamma, X, A) \subset {\ms B}_\gamma^*(X)$. Let $G_u \subset G$ be the (finite) automorphism group of $u$. The Fredholm property of $\mathring D_{(\phi, v, u)}^{\rm vt}$ allows us to find a finite-dimensional complex representation of $G_u$, denoted by ${\ms W}_{\phi, v, u}$, and a $G_u$-equivariant bundle 
\beqn
\uplambda_{\phi, v, u}: {\ms W}_{\phi, v, u} \times \mathring {\ms C}_{\phi, v} \times X \to \Lambda^{0,1} T^* \mathring C_{\phi, v} \otimes TX
\eeqn
such that after restricting to the graph of $u$, the image is transverse to the image of $\mathring D_{(\phi, v, u)}^{\rm vt}$. 

One can then extend ${\ms W}_{\phi, v, u}$ to be a $G$-equivariant vector bundle over a $G$-invariant neighborhood of $(\phi, v) \in {\ms B}_\gamma^*$ and extend $\uplambda_{\phi, v, u}$ to be a $G$-equivariant bundle map over $\mathring {\ms C}_{\gamma}^* \times X$. By stabilizing the bundle to a trivial bundle, one can extend it to a trivial $G$-bundle over ${\ms B}_\gamma^*$. Using a $G$-invariant cut-off function, one can also extend $\uplambda_{\phi, v, u}$ globally. 

We choose the datum $({\ms W}_{\phi, v, u}, \uplambda_{\phi, v, u})$ as above for each $(\phi, v, u) \in {\mc M}_{{\mb L}}^{\rm fr}({\ms C}_\gamma /{\ms Z}_\gamma, X, A)$ which is transverse over a $G$-invariant neighborhood of $(\phi, v, u)$. By the compactness, one can find finitely many such whose transversal loci cover ${\mc M}_{\mb L}^{\rm fr}({\ms C}_\gamma / {\ms Z}_\gamma, X, A)$. As each ${\ms W}_{\phi, v, u}$ is a $G$-equivariant vector bundle over ${\ms B}_\gamma^*$, one can take the direct sum of these finitely many and obtain a thickening datum 
\beqn
\uplambda_\gamma:= ({\ms Z}_\gamma, {\ms W}_\gamma, \uplambda_\gamma)
\eeqn
which is transverse over ${\ms Z}_\gamma$. Notice that ${\ms W}_\gamma \to {\ms B}_\gamma^*$ is trivial. 

We would like to extend this thickening datum to a neighborhood of ${\ms Z}_\gamma$ inside ${\ms B}_\gamma^\epsilon$. Indeed, consider the subset $\mathring {\ms C}_\gamma^\epsilon$ (the complement of nodes and markings) which is a $G$-manifold. Choose a $G$-invariant Riemannian metric. Let $N\mathring {\ms C}_\gamma^* \to \mathring{\ms C}_\gamma^*$ be the normal bundle identified with the orthogonal complement of the tangent bundle of $\mathring {\ms C}_\gamma^*$. Then using the exponential map, one can identify a tubular neighborhood with a disk bundle $N^\epsilon \mathring {\ms C}_\gamma^*$ and a tubular projection, temporarily denoted by $\pi$. Then we can find a $G$-equivariant bundle isomorphism
\beqn
\pi^* \Lambda^{0,1}_{\mathring {\ms C}_\gamma^* / {\ms B}_\gamma^*} \cong \Lambda^{0,1}_{ {\ms C}_\gamma^\epsilon/  {\ms B}_\gamma^\epsilon}|_{|N^\epsilon \mathring {\ms C}_\gamma^*|}
\eeqn
Then one can define the thickening map $  \uplambda_\gamma$ to be the pullback of the already-chosen thickening map via the map
\beqn
\pi^* \Lambda^{0,1}_{\mathring {\ms C}_\gamma^*/ {\ms B}_\gamma^*} \otimes TX \to \Lambda^{0,1}_{\mathring {\ms C}_\gamma^\epsilon/ {\ms B}_\gamma^\epsilon} \otimes TX.
\eeqn
This provides a $\gamma$-thickening datum (Definition \ref{defn_gamma_thickening}) over a neighborhood of ${\ms Z}_\gamma$ inside ${\ms B}_\gamma^\epsilon$ which is also transverse. To finally get a component-refined thickening datum, one apply the above construction for all $\gamma \mapsto \Gamma$ and take direct sum. 
\end{proof}

\begin{cor}\label{cor520}
One can find a collection of component-wise transverse component-refined thickening data $\mu_\Gamma$ for all isomorphism classes of map types $\Gamma$, denoted by $\mu_\Gamma = (O_\Gamma, W_\Gamma, \lambda_\Gamma)$ such that
\beqn
{\mc M}_{\mb L}^{\rm fr}(C/B, X, A) = \bigcup_\Gamma {\mc M}_{\mb L}^{\rm fr}(C/O_\Gamma, X, A).
\eeqn
\end{cor}

\begin{proof}
By the compactness of $\ov{\mc M}_{g,n}(X, J, A)$ and Lemma \ref{lemma_AMS_parametrization}, one can find a $G$-invariant precompact open subset $B_0 \subset B$ such that 
\beqn
{\mc M}_{\mb L}^{\rm fr}(C/B, X, A)= {\mc M}_{\mb L}^{\rm fr}(C/B_0, X, A).
\eeqn
Then we choose $\mu_\Gamma$ inductively. For any lowest stratum $\Gamma$, the set $\ov{B_0} \cap B_\Gamma^*$ is compact. Hence by Proposition \ref{prop_thickening_invariance}, one can find a component-wise transverse and component-refined thickening datum $\mu_\Gamma = (O_\Gamma, W_\Gamma, \lambda_\Gamma)$ with $\ov{B_0} \cap B_\Gamma^* \subset O_\Gamma$. Then one can inductively construct required thickening data for all strata.
\end{proof}

\begin{step}\label{step2}
Choose a collection of component-wise transverse component-refined thickening datum $\mu_\Gamma$ which satisfy the properties of Corollary \ref{cor520}.
\end{step}

\subsection{Constructing Kuranishi atlas}

Now we can construct an atlas. The process differ slightly with previous works such as \cite{Fukaya_Ono} \cite{MW_3} \cite{Tian_Xu_geometric_2}. The advantage of the current approach (via the AMS parametrization) is that the indexing set of charts are already given by the set of isomorphism classes of map types. 

\begin{step}\label{step3}
Choose $G$-invariant precompact open subsets 
\beqn
O_{\Gamma}'' \sqsubset O_{\Gamma}' \sqsubset O_{\Gamma}
\eeqn
such that 
\beqn
{\mc M}_{\mb L}^{\rm fr}(C/B, X, A) = \bigcup_\Gamma {\mc M}_{\mb L}^{\rm fr}(C/O_{\Gamma}'', X, A).
\eeqn
Then we choose $G$-invariant cut-off functions
\beqn
\rho_{\Gamma}: O_{\Gamma} \to [0, 1]
\eeqn
which is identically $1$ on $O_{\Gamma}''$ and which is zero outside $O_{\Gamma}'$. Lastly, we choose an infinite sequence of precompact $G$-invariant open subsets
\beq\label{nested_sets}
O_{\Gamma}' \sqsubset G_{\Gamma}^1 \sqsubset H_{\Gamma}^1 \sqsubset \cdots \sqsubset G_{\Gamma}^m \sqsubset H_{\Gamma}^m \sqsubset \cdots  O_{\Gamma}.
\eeq
(We only use finitely many of them.) 
\end{step}

Then we define the (not necessarily everywhere transverse) thickening datum 
\beqn
\rho_{\Gamma} \mu_{\Gamma}
\eeqn
over $O_{\Gamma}$. 

Now we can describe the atlas. First define ${\mc I}$ to be the set of all nonempty subsets of $\Gamma$'s. The partial order on ${\mc I}$ is induced by inclusions of subsets. For each $I \in {\mc I}$, define
\beqn
O_I:= \bigcap_{\Gamma \in I} H_{\Gamma}^{\# I} \setminus \bigcup_{\Gamma \notin I} \ov{G_{\Gamma}^{\# I}}
\eeqn
(which could be empty).

\begin{lemma}\label{lemma_overlapping}
The following condition holds for each pair $I, J \in {\mc I}$:
\beqn
\ov{O_I} \cap \ov{O_J} \neq \emptyset \Longrightarrow I \subset J\ {\rm or}\ J \subset I.
\eeqn
\end{lemma}

\begin{proof}
It is a straightforward check. The details can be found in the proof of \cite[Lemma 5.3.1]{MW_3}. 
\end{proof}

Consider the direct sum thickening datum over $O_I$: 
\beqn
\mu_I:= \bigoplus_{\Gamma \in I} \rho_\Gamma \mu_{\Gamma}.
\eeqn

\begin{lemma}
All $\mu_I$ are transverse.
\end{lemma}

\begin{proof}
Indeed, for each point $x \in S_I^{-1}(0)$ which projects down to $y \in O_I$, $y$ is covered by at least one $O_{\Gamma}''$ (where $\rho_{\Gamma} \equiv 1$). By the definition of $O_I$, for all $\Gamma \notin I$, $\rho_\Gamma (y) = 0$. Hence there must be some $\Gamma \in I$ such that $\rho_{\Gamma}(y) = 1$. As the original $\mu_{\Gamma}$ is transverse, it implies that $\mu_I$ is transverse at $x$. 
\end{proof}

By Definition \ref{defn_AMS_thickening}, it leads to a Kuranishi chart
\beqn
K_I = (U_I, E_I, S_I, \Psi_I).
\eeqn
We equip $U_I$ with $G$-invariant smooth structure using the construction of Proposition \ref{prop_smooth}. As $E_I \to U_I$ is a $G$-equivariant flat bundle, it automatically becomes a smooth $G$-equivariant vector bundle.

On the other hand, one can define coordinate changes for $I \subset J$. Simply define 
\beqn
U_{JI}
\eeqn
to be the restriction of $K_I$ to $O_I \cap O_J$. Then the embedding
\beqn
U_{JI} \to U_{J}
\eeqn
is canonically defined as the latter is obtained by having an additional summand of thickening datum.

\begin{lemma}
For all $I \leq J$, the embedding $U_{JI} \to U_J$ is smooth.    
\end{lemma}

\begin{proof}
This is the consequence of the main result of \cite{FOOO_smooth}.   
\end{proof}

As the obstruction bundle naturally embeds, one obtains an embedding of $G$-equivariant Kuranishi charts
\beqn
{\bm \Phi}_{JI}: K_I|_{U_{JI}} \to K_J.
\eeqn

\begin{lemma}\label{lemma524}
For each pair $I \leq J$ and $\Gamma \in J \setminus I$, one has 
\beqn
\rho_\Gamma |_{O_I \cap O_J} \equiv 0.
\eeqn
\end{lemma}

\begin{proof}
Suppose $\# I = k$ and $\# J = l>k$. Then by the definitions of $O_I$ and $O_J$, one has 
\beqn
\begin{split}
O_I \cap O_J = &\ \left( \bigcap_{\Gamma\in I} H_\Gamma^k \setminus \bigcup_{\Gamma \notin I} \ov{G_\Gamma^k}  \right) \cap \left( \bigcap_{\Gamma \in J} H_\Gamma^l \setminus \bigcup_{\Gamma \notin J} \ov{G_\Gamma^l} \right)\\
\subset &\ \bigcap_{\Gamma \in I} H_\Gamma^k \cap \left( \bigcap_{\Gamma \in J \setminus I}  (H_\Gamma^l \setminus \ov{G_\Gamma^k})    \right) \setminus \bigcup_{\Gamma \notin J} \ov{G_\Gamma^l}\subset \bigcap_{\Gamma \in J \setminus I} (O_\Gamma \setminus O_\Gamma').
\end{split}
\eeqn
As $\rho_\Gamma$ vanishes outside $O_\Gamma'$, the claim follows.
\end{proof}

\begin{cor}
For all pairs $I \leq J$, the coordinate change ${\bm \Phi}_{JI}$ has a canonical equivalence class of flat tubular neighborhoods (Definition \ref{defn_equivariant_Kuranishi_chart}). 
\end{cor}

\begin{proof}
Lemma \ref{lemma524} implies that on the overlap $O_I \cap O_J$, the additional thickening data $\rho_\Gamma \mu_\Gamma$ for $\Gamma \in J \setminus I$ are all trivial. Hence the corresponding summands $E_\Gamma \subset E_J$ does not perturb the equation there. 
\end{proof}

\begin{prop}
The collection ${\mf A} = (K_I; {\bm \Phi}_{JI})$ 
form a $\uds{\bf Kur}^{G, \flat}$-atlas (see Definition \ref{defn_Kuranishi_atlas}) on ${\mf M}$. Moreover, there exists a strong map (Definition \ref{defn_strong_map}) 
\beqn
{\mf f}: {\mf A} \to \ov{\mc M}{}_{g,n}\times X^n
\eeqn
which extends the evaluation-stabilization map on the moduli space.
\end{prop}

\begin{proof}
The overlapping condition is obvious since the open sets $O_I$. Without considering the combed condition, the cocycle condition is also obvious from the construction. 

Now we verify the virtual neighborhood property of Definition \ref{defn_Kuranishi_atlas}. Consider the binary relation $\curlyvee$. The only non-obvious thing to check is the transitivity. Suppose $x \curlyvee y \curlyvee z$ where $x \in U_I$, $y \in U_J$, $z \in U_L$. Then their underlying domains in $B$ coincides. Then $O_I \cap O_J \cap O_K \neq \emptyset$. Then Lemma \ref{lemma_overlapping} implies that there must be a total order among $I$, $J$, and $K$. So the cocycle condition implies that $x \curlyvee z$. Therefore $\curlyvee$ is an equivalence relation.

The remaining claims that $|{\mf A}|$ is Hausdorff and that the inclusion $U_I \to |{\mf A}|$ is a homeomorphism onto its image can both be proved using elementary point-set topology. We omit the details.

Lastly, the chartwise map 
\beqn
f_I: U_I \to \ov{\mc M}_{g,n} \times X^n
\eeqn
is obviously defined. To see that they form a strong map, one needs to show that they are invariant under coordinate change and the involved stabilization in the normal direction. This is the case as the additional thickening data do not perturb the equation. 
\end{proof}

\subsubsection{Specific shrinkings}

Recall that a priori the virtual fundamental class depends on shrinkings of the Kuranishi atlas. Before we move on, we discuss the class of shrinkings we would like to choose. The specific choices are similar to the notion of ``reduction'' of McDuff--Wehrheim \cite{MW_3} in their version of Kuranishi atlas. 

\begin{defn}\label{defn_basic_shrinking}
A {\bf basic shrinking} of the sequence \eqref{nested_sets} consists of a sequence of intermediate $G$-invariant open subsets
\beqn
G_{\Gamma}^l \sqsubset \dot G_{\Gamma}^l \sqsubset \dot H_{\Gamma}^l \sqsubset H_{\Gamma}^l,\ l = 1, \ldots
\eeqn
\end{defn}

Notice that if one defines
\beqn
O_I':= \bigcap_{\Gamma \in I} \dot H_{\Gamma}^{\#I} \setminus \bigcup_{\Gamma \notin I} \ov{ \dot G_{\Gamma_j}^{\#I}}
\eeqn
then $O_I'$ is precompact in $O_I$. Moreover,  
\beqn
{\mc M}_{\mb L}^{\rm fr}(C/\ov{O_I'}, X, A) \subset U_{I}
\eeqn
is compact. 

\begin{step}\label{step4}
Choose a basic shrinking of the sequence \eqref{nested_sets} and choose a precompact $G$-invariant open subset $U_I' \sqsubset U_I$ such that 
\beqn
S_I^{-1}(0) \cap \ov{U_I'} = {\mc M}_{\mb L}^{\rm fr}(C/\ov{O_I'}, X, A).
\eeqn
\end{step}

Then one obtained a shrinking of ${\mf A}$, called a {\bf basic shrinking}. The comparison on choices of basic shrinking will be provided when proving the invariance of the virtual fundamental class.

\subsection{The NCS refinement}

In this section we promote the $\uds{\bf Kur}^{G, \flat}$-atlas to a $\uds{\bf Kur}_{\rm NCS}^{G, \flat}$-atlas. Basically, we will construct/identify structures of NCS virtual manifolds on the Kuranishi charts which are $G$-invariant and which are compatible with coordinate changes.

\subsubsection{Refined stratification}

Now we can describe the stratification on the Kuranishi charts designed for defining reduced Gromov--Witten invariants. 

\begin{defn}\label{defn_refined_map_type}
A {\bf refined map type} is a pair $(\tilde \Gamma, {\rm Vert}^{\rm gho}_{\rm off}(\tilde\Gamma))$ where $\tilde \Gamma$ is a map type and ${\rm Vert}^{\rm gho}_{\rm off}(\tilde\Gamma) \subset {\rm Vert}^{\rm gho}(\tilde\Gamma)$ is a subset. An isomorphism of refined map types from $(\tilde \Gamma, {\rm Vert}^{\rm gho}_{\rm off} (\tilde\Gamma))$ to $(\tilde \Gamma', {\rm Vert}^{\rm gho}_{\rm off}(\tilde\Gamma'))$ is an isomorphism $\rho: \tilde \Gamma \to \tilde \Gamma'$ of map types such that $\rho({\rm Vert}^{\rm gho}_{\rm off}(\tilde\Gamma)) = {\rm Vert}^{\rm gho}_{\rm off}(\tilde\Gamma')$. Let $\alpha$ denote an isomorphism class of refined map types. 
\end{defn}

The subset ${\rm Vert}^{\rm gho}_{\rm eff}(\tilde\Gamma)$ labels the set of ghost vertices where one would like to turn off the thickening data. Using refined map types one can define a refined stratification of Kuranishi charts described in the following lemma. 

\begin{lemma}
Let $\mu = (O, W, \lambda)$ be a component-refined and component-wise transverse thickening datum and $K_\mu = (U_\mu, E_\mu, S_\mu, \Psi_\mu)$ be the associated $G$-equivariant Kuranishi chart. For each isomorphism class of refined map types $\alpha$ (represented by $(\tilde \Gamma, {\rm Vert}^{\rm gho}_{\rm off}(\tilde\Gamma))$, define
\beqn
U_{\mu, \alpha}^*:= \left\{(\phi, u, e, F) \in U_\mu\ \left| \ \begin{array}{ll} \exists f: \tilde \Gamma_\phi \cong \tilde \Gamma\ {\rm s.t.}\ \forall v \in {\rm Vert}^{\rm gho}(\tilde\Gamma_\phi)\\
f(v) \in {\rm Vert}^{\rm gho}_{\rm off}(\tilde\Gamma)  \Longleftrightarrow e_{\phi, v} = 0 \end{array}\right. \right\}
\eeqn
Then the partition 
\beqn
U_\mu = \bigsqcup_\alpha U_{\mu, \alpha}^*
\eeqn
makes $U_\mu$ a stratified $G$-manifold (Definition \ref{defn_manifold_stratification}).  Moreover, the partial order among isomorphism classes of refined map types induced from the stratification is equivalent to the following relation: if $\alpha$ is represented by $(\tilde \Gamma_\alpha, {\rm Vert}^{\rm gho}_{\rm off}(\tilde \Gamma_\alpha))$ and $\beta$ is represented by $(\tilde \Gamma_\beta, {\rm Vert}^{\rm gho}_{\rm off}(\tilde\Gamma_\beta))$, then
\beqn
\alpha \leq \beta \Longleftrightarrow \exists \rho: \tilde \Gamma_\alpha \to \tilde \Gamma_\beta\ {\rm s.t.}\ \rho(v) \in {\rm Vert}^{\rm gho}_{\rm off} (\tilde \Gamma_\beta) \Longrightarrow v \in {\rm Vert}^{\rm gho}_{\rm off}(\tilde\Gamma_\alpha).
\eeqn
\end{lemma}

\begin{proof}
To save notations, in this proof, we drop the index $\mu$. We first verify that $U_\alpha^* \subset U$ is a $G$-invariant smooth submanifold. The $G$-invariance follows from the requirement for the thickening datum and the definition of $U_\alpha^*$. As $U_\Gamma^* \subset U$ is a smooth submanifold, to show that $U_\alpha^*$ is a smooth submanifold, it suffices to prove that $U_\alpha^* \subset U_{\Gamma}^*$ is a smooth submanifold. Choose $x = (\phi, u, e, F) \in U_\alpha^*$. Then $\alpha$ identifies a set of ghost vertices in $\tilde \Gamma_\phi$ where the perturbation is turned off. Though the constant maps on higher genus ghost components are not transverse, but they are cut off cleanly. Therefore, $U_\alpha^*$ is a smooth submanifold of $U_{\Gamma}^*$. 

We then verify that this is a stratification. The local finiteness of the partition and local closedness of each stratum are easy to verify. To see the axiom of frontier, suppose $U_\alpha^* \cap U_\beta \neq \emptyset$. We need to show $U_\alpha^* \subset U_\beta$. By assumption, there exists a sequence $x_i = (\phi_i, u_i, e_i, F_i) \in U_\beta^*$ converging to $x = (\phi, u, e, F) \in U_\alpha^*$. It follows that $\phi_i \to \phi$ in $B$. Therefore, choosing a subsequence if necessary, for $i$ sufficiently large, the map type $\tilde \Gamma_{\phi_i}$ can be identified with a map type $\tilde \Gamma_\beta$ and the convergence $\phi_i \to \phi$ specifies a  surjective graph map
\beqn
\rho: \tilde \Gamma_\alpha \to \tilde \Gamma_\beta.
\eeqn
Moreover, for each $v \in {\rm Vert}_{\rm off}^{\rm gho}(\tilde \Gamma_\beta)$, $e_{\phi_i, v} = 0$. As $e_i$ converges to $e$, it follows that for each $v' \in \rho^{-1}(v)$ (which must be a ghost component), one has $e_{\phi, v'} = 0$, meaning $v' \in {\rm Vert}^{\rm gho}_{\rm off}(\tilde\Gamma_\alpha)$.  Therefore, 
\beqn
\rho(v') \in {\rm Vert}^{\rm gho}_{\rm off}(\tilde\Gamma_{\beta})  \Longrightarrow v' \in {\rm Vert}^{\rm gho}_{\rm off}(\tilde\Gamma_\alpha).
\eeqn
Then we can show that $U_\alpha^* \subset U_\beta$.

Lastly, to see this is a manifold stratification (see Definition \ref{defn_manifold_stratification}), one only needs to give the local linear model as one has the transversality condition. Indeed, for each point $x  = (\phi, u, e, F) \in U_\alpha^*$, consider the vertical linearized Cauchy--Riemann operator
\beqn
D_{\alpha, u}: \Omega^0(u^* TX) \oplus \bigoplus_{v \notin {\rm Vert}_{\rm off}^{\rm gho}(\tilde\Gamma_\phi)} W_{\phi, v} \to \Omega^{0,1}(u^* TX).
\eeqn
It is not surjective unless ${\rm Vert}_{\rm off}^{\rm gho}(\tilde\Gamma_\phi) = \emptyset$. However, the cokernels of such operators form a $G$-equivariant vector bundle, denoted by 
\beqn
H^1_\alpha \to U_\alpha^*.
\eeqn
Define another vector bundle $W_\alpha^{\rm off} \to U_\alpha^*$ whose fibre over $(\phi, u, e, F)\in U_\alpha^*$ is the direct sum 
\beqn
\bigoplus_{v \in {\rm Vert}_{\rm off}^{\rm gho}(\tilde\Gamma_\phi)} W_{\phi, v} \subset W_\phi.
\eeqn
As the thickening datum $\lambda$ is transverse, it induces a surjective bundle map
\beqn
W_\alpha^{\rm off} \to H_\alpha^1.
\eeqn
Then one can see that the normal fibre at $x = (\phi, u, e, F)$ can be identified (non-canonically) with
\beq\label{chart_normal_bundle}
N_x U_\alpha^* \cong N_\phi B_{\Gamma_\alpha}^* \oplus \Big( {\rm ker} (W_\alpha^{\rm off} \to H_\alpha^1 )\Big)
\eeq
where $N_\phi B_{\Gamma_\alpha}^*$ parametrizes the directions of gluing the perturbed solutions while each $W_{\phi, v}^{\rm gho}$ parametrizes the release of perturbations on ghost components. The normal bundles are also linearly stratified by all $\beta \geq \alpha$. Then a stratified chart can be constructed using the gluing construction. 
\end{proof}

\subsubsection{Stratification on the obstruction bundle}

Recall that the obstruction bundle $E$ is the direct sum $W \oplus Q$ where $W$ comes from the component-refined thickening datum and $Q$ is the trivial bundle (which is not complex). 

We would like to specify a natural structure of $G$-equivariant flat 
stratified Hermitian vector bundles of $W$ over $U$. By Definition \ref{defn_stratified_vector_bundle}, one first needs to describe a constructible cosheaf ${\mc O}^W$ over $U$. Indeed, for each $x  = (\phi, u, e, F) \in U$, one has the natural splitting 
\beqn
W_x = \bigoplus_{v \in {\rm Vert}(\tilde\Gamma_\phi)} W_{\phi, v}.
\eeqn
We choose our thickening datum in the way that all $W_{\phi, v}$ are nonzero. Hence $W_x$ is naturally stratified by a poset ${\mc O}^W_x$. Moreover, it provides a stratification-like cosheaf ${\mc O}^W = {\mc O}^E$ over $U$, as along each stratum $U_\alpha^*$ the map types are locally constant and gluing corresponds to merging vertices in the graph.

To make the pair $(U, E)$ a stratified virtual manifold, one also needs to specify the cosheaf map ${\mc O}^U \to {\mc O}^E$. This map is obvious as for each $x \in U_\alpha^*$, ${\mc O}_x^U$ consists of the set of certain subsets of ghost vertices while ${\mc O}_x^E$ consists of the set of all subsets of all vertices. There is hence a natural inclusion ${\mc O}_x^U \to {\mc O}_x^E$. The following fact is obvious.

\begin{lemma}
For each $I$, the Kuranishi section $S_I: U_I \to E_I$ is stratified. 
\end{lemma}

\subsubsection{The NCS structure on the thickenings}

It remains to construct NC structures on the thickenings $U_I$ to make the chart $K_I$ a $\uds{\bf Kur}_{\rm NCS}^{G, \flat}$-chart and to make the atlas a $\uds{\bf Kur}_{\rm NCS}^{G, \flat}$-atlas. There is one more choice to make. By Lemma \ref{lemma_compatible_metrics}, one can do the following step. 

\begin{step}\label{step5}
Choose a Riemannian metric near $\ov{{\mf A}'}$.
\end{step}

Using this construction one can obtain normal complex structures on all thickenings $U_I$. Indeed, for each refined map type $\alpha$, there is a refined map type $\ov\alpha$ which turns on thickening data on all components. Then $\alpha \leq \ov\alpha$. Moreover, for any point $x \in U_{I,\alpha}^*$, there is only one local branch of $\tilde U_{I,\ov\alpha}$ near $x$. Then inside the normal bundle $NU_{I,\alpha}^*$, there is a well-defined subbundle
\beqn
(NU_{I,\alpha}^*)_{\ov\alpha} \subset NU_{I,\alpha}^*.
\eeqn
The Riemannian metric induces a splitting of the exact sequence 
\beq\label{eqn56}
\xymatrix{ 0 \ar[r] & (NU_{I,\alpha}^*)_{\ov\alpha} \ar[r] &    NU_{I,\alpha}^* \ar[r] & N\tilde U_{I, \ov\alpha} \ar[r] & 0}.
\eeq
The splitting is compatible with the coordinate changes near $\ov{\mf A'}$. Hence one only needs to choose complex structures on $(NU_{I,\alpha}^*)_{\ov\alpha}$ and on $N\tilde U_{I,\ov\alpha}$. Once this splitting is obtained, one can forget about the Riemannian metric. 

For the latter, notice that the forgetful map $U_I \to B$ sends $U_{I, \ov\alpha}^*$ to $B_{\Gamma_\alpha}^*$. The forgetful map can be extended to the domain of the corresponding immersions, i.e.
\beqn
\tilde U_{I, \ov\alpha} \to \tilde B_{\Gamma_\alpha}.
\eeqn
One can see that the smooth projection induces a bundle isomorphism
\beqn
N\tilde U_{I, \ov\alpha} \cong ( \pi_{\tilde U_{I, \ov\alpha} \to \tilde B_{\Gamma_\alpha}})^* N\tilde B_{\Gamma_\alpha}
\eeqn
where the latter has a complex structure because $B$ is a complex manifold and $\tilde B_{\Gamma_\alpha} \to B$ is holomorphic.

On the other hand, one can identify a complex structure on the first item of \eqref{eqn56}. 

\begin{lemma}
There exists canonical $G$-invariant complex structures on the normal bundle $(NU_{I, \alpha}^*)_{\ov\alpha}$.
\end{lemma}

\begin{proof}
Consider the surjective complex-linear bundle map over $U_{I, \alpha}^*$
\beqn
W_{I,\alpha}^{\rm off} \to H_{I,\alpha}^1.
\eeqn
Using the Hermitian metric on the thickening space $W_I$, one obtains a $G$-invariant orthogonal and complex-linear splitting
\beqn
W_{I,\alpha}^{\rm off} \cong H_{I,\alpha}^1 \oplus {\rm ker}(W_{I,\alpha}^{\rm off}\to H_{I,\alpha}^1).
\eeqn
By the implicit function theorem (without involving gluing), the derivative of the section $S_{W_I}: U_I \to W_I$ in the direction of $(NU_{I, \alpha}^*)_{\ov\alpha}$ induces an isomorphism
\beqn
(NU_{I, \alpha}^*)_{\ov\alpha} \to W_I \to W_{I,\alpha}^{\rm off} \to {\rm ker}(W_{I, \alpha}^{\rm off} \to H_{I, \alpha}^1).
\eeqn
It induces a complex structure on $(NU_{I, \alpha}^*)_{\ov\alpha}$. 
\end{proof}

By now, one has identified complex structures on both $(NU_{I,\alpha}^*)_{\ov\alpha}$ and $N\tilde U_{I, \ov\alpha}$. Using the splitting of the exact sequence \eqref{eqn56}, one obtains a $G$-invariant complex structure on $NU_{I, \alpha}^*$. It is not hard to see that it extends to the normal bundle of the immersion $N\tilde U_{I, \alpha}$. 

\begin{prop}
The following are true.
\begin{enumerate}

\item The complex structure on $N\tilde U_{I, \alpha}$ induces a NC structure on the bundle $TU|_{\tilde U_{I, \alpha}}$.

\item For each $I$, the collection of NC structures on $TU|_{\tilde U_{I, \alpha}}$ form an NC structure on $TU \to U$. Hence the chart $K_I = (U_I, E_I, S_I, \Psi_I)$ is a $\uds{\bf Kur}_{\rm NCS}^{G, \flat}$-chart on ${\mf M}$.

\item A shrinking ${\mf A}_+'$ containing $\ov{\mf A'}$ together with the restrictions of the collection of the charts and coordinate changes to ${\mf A}_+'$  form a $\uds{\bf Kur}_{\rm NCS}^{G, \flat}$-atlas on ${\mf M}$. 
\end{enumerate}
\end{prop}

\begin{proof}
The normal bundle $N\tilde U_{I, \alpha}$ is stratified by all $\beta \geq \alpha$. By the definition of NC structures on vector spaces (Definition \ref{defn_NCS_space}), we first need to show that for each $\beta\geq\alpha$, $(N\tilde U_{I, \alpha})_\beta \subset N\tilde U_{I, \alpha}$ is a complex subspace. Indeed, this can be checked directly following the way we construct the complex structures. Hence $TU|_{\tilde U_{I, \alpha}}$ becomes a $G$-equivariant NC vector bundle. 

For (2), by the definition of NC structures on stratified vector bundles (Definition \ref{defn_NCS_virtual_manifold}), one needs to show that for $x \in U_{I, \alpha}^* \cap \ov{U_{I, \beta}^*}$ for $\alpha < \beta$, the NC structure on $T_x U_I$ induced from either $\alpha$ or $\beta$ are the same.  Indeed, one has the following commutative diagram.
\beqn
\xymatrix{  0 \ar[r] & (NU_{I, \alpha}^*)_{\ov\alpha} \ar[d] \ar[r] &  NU_{I, \alpha}^* \ar[r] \ar[d]  & N\tilde U_{I, \ov\alpha} \ar[r] \ar[d] & 0 \\
      0 \ar[r] & (N\tilde U_{I, \beta})_{\ov\beta} \ar[r] &  N\tilde U_{I, \beta} \ar[r] &  N\tilde U_{I, \ov\beta} \ar[r] & 0 }
\eeqn
As the complex structures on $NU_{I, \alpha}^*$ and $NU_{I, \beta}^*$ are determined by complex structures on other bundles and the splittings, while the splittings also commute with all arrows in the diagram, we see that the map $N_x U_{I, \alpha}^* \to N_x \tilde U_{I, \beta}$ is complex linear. Hence the compatibility holds. 

For (3), as the compatible complex structures are only constructed near $\ov{\mf A'}$, one needs to take a necessary shrinking which contains $\ov{\mf A'}$.
\end{proof}

\begin{prop}
After the choices made in Steps I---V, one obtains an oriented $\uds{\bf Kur}_{\rm NCS}^{G, \flat}$-atlas ${\mf A}$ (without boundary) on ${\mf M}$ with a shrinking ${\mf A}'$ and a continuous strong map 
\beqn
{\mf f}: {\mf A} \to \ov{\mc M}_{g,n} \times X^n 
\eeqn
which extends the product of the stabilization map and the evaluation map. Hence there is a well-defined class
\beqn
[{\mf A}, {\mf A}']_{\rm main}^{\rm vir} \in H_*( \ov{\mc M}_{g,n} \times X^n; {\mb Q}).
\eeqn
\end{prop}

To finish proving Theorem \ref{thm52}, one needs to compare  two different sets of choices as well as two different pairs $(\omega_0, J_0)$ and $(\omega_1, J_1)$ which are deformation equivalent.

\subsection{Comparing choices}

Now we prove that the virtual fundamental class constructed above only depends on the symplectic deformation class of $\omega$. 

\subsubsection{Symplectic deformation invariance}

We first compare different choices made in Step \ref{step2} through Step \ref{step5} while using the same choice made in Step \ref{step1}. In this stage we could also include the comparison of deformation equivalent symplectic forms and compatible almost complex structures. 

Consider a 1-parameter family of symplectic forms $\omega_t$, $t \in [0, 1]$. Let $J_t$ be a family of $\omega_t$-compatible almost complex structures. Let $\tilde{\mf M}$ be the moduli spaces of pairs $(p, t)$ where $t \in [0, 1]$ and $p \in \ov{\mc M}{}_{g,n}(X, J_t, A)$ and ${\mf M}_t \subset \tilde {\mf M}$ the slice for any fixed $t$. 

We first choose a common line bundle datum as the Step \ref{step1} of the construction. For each $t_0 \in [0, 1]$, there exists a line bundle datum ${\mb L} = (L, k, {\mf D})$ whose curvature form $\Omega$ tames all $J_t$ for $t$ sufficiently close to $t_0$. As $[0, 1]$ is compact, one can break down the comparison to finitely many short intervals. Hence we may assume that $\Omega$ tames $J_t$ for all $t \in [0, 1]$. 

Now we consider the associated family of curves and the stratification by map types. Abbreviate $B_{g, n, d}$ by $B$ where $d$ is given by \eqref{degree_formula}. Denote
\beqn
\tilde{B} = B \times [0, 1]
\eeqn
which is stratified by pairs $\Theta:=(\Gamma, a)$ where $a \in \{0, 1, (0, 1)\}$ and $\Gamma$ is an isomorphism class of map types. Let $\tilde C \to \tilde B$ be the universal curve. Then one can have an AMS parametrization of the moduli space $\tilde{\mf M}$. One can choose a $G$-invariant precompact open subset $B_0 \subset B$ such that 
\beqn
{\mc M}_{{\mb L}}^{\rm fr}( \tilde C/ \tilde B_0, X, A)/G \cong \tilde{\mf M},\ {\rm where}\ \tilde B_0 = B_0 \times [0, 1].
\eeqn

Now we make independent choices in Step \ref{step2}---Step \ref{step5} for the two moduli spaces ${\mf M}_0$ and ${\mf M}_1$. In Step \ref{step2}, for $a = 0, 1$, one chooses component-wise transverse component-refined thickening datum 
\beqn
\mu_{\Gamma, a} = (O_{\Gamma, a}, W_{\Gamma, a}, \lambda_{\Gamma, a})
\eeqn
such that for both $a = 0, 1$, one has
\beqn
\ov{B_0} \subset \bigcup_i O_{\Gamma, a}.
\eeqn
In Step \ref{step3}, one chooses precompact open subsets $O_{\Gamma, a}'' \sqsubset O_{\Gamma, a}' \sqsubset O_{\Gamma, a}$ such that all $O_{\Gamma, a}''$ cover $B_0$, cut-off function $\rho_{\Gamma, a}$ supported in $O_{\Gamma, a}'$ and identically 1 on $O_{\Gamma, a}''$, and infinite sequences of precompact open sets
\beqn
O_{\Gamma, a}' \sqsubset G_{\Gamma, a}^1 \sqsubset H_{\Gamma, a}^1 \sqsubset \cdots \sqsubset O_{\Gamma, a}.
\eeqn

\begin{lemma}\label{lemma532}
For independent constructions of $({\mf A}_0, {\mf f}_0)$ and $({\mf A}_1, {\mf f}_1)$ respectively, there exist the following objects.
\begin{enumerate}

\item a $\uds{\bf Kur}^{G, \flat}$-atlas $\tilde {\mf A}$ with boundary on $\tilde{\mf M}$.



\item A continuous strong map $\tilde {\mf f}: \tilde {\mf A} \to \ov{\mc M}_{g,n} \times X^n$.

\end{enumerate}
They satisfy the following conditions.
\begin{enumerate}

\item The boundary restrictions of $(\tilde {\mf A}, \tilde {\mf f})$ coincide with the disjoint union of $({\mf A}_0, {\mf f}_0)$ and $({\mf A}_1, {\mf f}_1)$.

\item There exists a collared neighborhood of the boundary where $(\tilde {\mf A}, \tilde {\mf f})$ is constant in the direction of the collar coordinate. 
\end{enumerate}
\end{lemma}

\begin{proof}
We carry out Step \ref{step2} and Step \ref{step3} for $\tilde {\mf M}$ where the choices for the two boundary components will be extended. 

Near the boundary, one can simply pullback the existing thickening data to a small collared neighborhood. We first define
\beqn
\tilde O_{\Gamma, a}:= \left\{ \begin{array}{cc} O_{\Gamma, a} \times [0, 4\epsilon),\ &\ a = 0,\\
O_{\Gamma, a} \times (1-4\epsilon, 1],\ &\ a = 1  \end{array}\right.
\eeqn
for $\epsilon>0$ sufficiently small and 
\begin{align*}
&\ \tilde O_{\Gamma, a}'':= \left\{ \begin{array}{cc} O_{\Gamma, a}' \times [0, 2\epsilon),\ &\ a = 0,\\
O_{\Gamma, a} \times (1-2\epsilon, 1],\ &\ a = 1,  \end{array}\right. \ &\ \tilde O_{\Gamma, a}':= \left\{ \begin{array}{cc} O_{\Gamma, a}' \times [0, 3\epsilon),\ &\ a = 0,\\
O_{\Gamma, a} \times (1-3\epsilon, 1],\ &\ a = 1.  \end{array}\right.
\end{align*}
Choose cut-off function
\beqn
\rho_a: [0, 1] \to [0,1]
\eeqn
such that $\rho_0$ is supported within $[0, 3\epsilon]$ and identically 1 within $[0, 2\epsilon]$; $\rho_1$ is supported within  $[1-3\epsilon, 1]$ and identically 1 within $[1-2\epsilon, 1]$. Define
\beqn
\tilde \rho_{\Gamma, a} = \rho_\Gamma \rho_a: \tilde O_{\Gamma, a} \to [0, 1].
\eeqn
The thickening data $\mu_{\Gamma, a}$ can be pulled back to $\tilde O_{\Gamma, a}$ and multiplied with the cut-off function $\tilde \rho_{\Gamma, a}$. This produces (not necessarily regular) Kuranishi charts with boundary, denoted by 
\beqn
\tilde K_{\Gamma, a}.
\eeqn
One can then choose two sequences of numbers
\begin{align*}
&\ 3\epsilon < g_0^1 < h_0^1 < \cdots < \cdots < 4\epsilon,\ &\ 1-3\epsilon < \cdots < h_1^1 < g_1^1 < 1-2\epsilon.
\end{align*}
Define
\begin{align*}
&\ \tilde G_{\Gamma, 0}^l = G_{\Gamma, 0}^l \times [0, g_0^l),\ &\ \tilde H_{\Gamma, 0}^l = H_{\Gamma, 0}^l \times [0, h_0^l)
\end{align*}
and corresponding versions of $\tilde G_{\Gamma, 1}^l$ and $\tilde H_{\Gamma, 1}^l$. Then one has
\beqn
\tilde O_{\Gamma, a}' \sqsubset \tilde G_{\Gamma, a}' \sqsubset \tilde H_{\Gamma, a}' \sqsubset \cdots \sqsubset \tilde O_{\Gamma, a}.
\eeqn

Now consider the region $[\epsilon, 1-\epsilon] \times B_0$. For each isomorphism class of map type $\Gamma$, one can construct transverse component-refined thickening datum 
\beqn
\mu_{\Gamma, (0, 1)} = (\tilde O_{\Gamma, (0, 1)}, \tilde W_{\Gamma, (0, 1)}, \tilde \lambda_{\Gamma, (0, 1)})
\eeqn
over the family $\tilde C \to \tilde O_{\Gamma, (0,1)}$ such that
\beqn
B_0 \times [\epsilon, 1-\epsilon] \subset \bigcup_\Gamma \tilde O_{\Gamma, (0,1)}.
\eeqn
Similar to the two boundary cases, one chooses subsets
\beqn
\tilde O_{\Gamma, (0,1)}'' \sqsubset \tilde O_{\Gamma, (0,1)}' \sqsubset \tilde O_{\Gamma, (0,1)}
\eeqn
such that the union of all $\tilde O_{\Gamma, 0}'', \tilde O_{\Gamma, 1}'', \tilde O_{\Gamma,(0,1)}''$ cover $[0,1]\times B_0$. Then choose $G$-invariant cut-off functions
\beqn
\tilde \rho_{\Gamma, (0,1)}: \tilde O_{\Gamma, (0,1)} \to [0,1]
\eeqn
supported in $\tilde O_{\Gamma, (0,1)}'$ and identically 1 in $\tilde O_{\Gamma,(0,1)}''$. Choose precompact subsets
\beqn
\tilde O_{\Gamma, (0,1)}' \sqsubset \tilde G_{\Gamma, (0,1)}' \sqsubset H_{\Gamma, (0,1)}' \sqsubset \cdots \sqsubset \tilde O_{\Gamma, (0,1)}.
\eeqn

Now we can construct an atlas on $\tilde{\mf M}$. Denote by $\tilde {\mc I}$ the set of all nonempty subsets of the collection of $\Theta$ (being either $(\Gamma, 0)$, $(\Gamma, 1)$, $(\Gamma, (0,1))$). For each $\tilde I \subset \tilde {\mc I}$, define
\beqn
\tilde O_{\tilde I}:= \bigcap_{\Theta \in \tilde I} \tilde H_\Theta^{\# \tilde I} \setminus \bigcup_{\Theta \notin \tilde I} \ov{ G_\Theta^{\# \tilde I}} .
\eeqn
By taking direct sums of the thickening data, one obtains charts and coordinate changes as before, denoted by $K_{\tilde I}$ and ${\bm \Phi}_{\tilde J\tilde I}$ and an atlas $\tilde {\mf A}$ on $\tilde{\mf M}$. It is straightforward to see that the atlas has boundary being identified with the disjoint union ${\mf A}_0$ and ${\mf A}_1$ which is constant in the collar coordinate direction. Moreover, the stabilization and evaluation maps are canonically defined and are strong maps on $\tilde {\mf A}$ satisfying the requirement.
\end{proof}

We then proceed with Step \ref{step4} and Step \ref{step5}. 

\begin{lemma}
In the situation of Lemma \ref{lemma532} one can do the following.

\begin{enumerate}

\item Given basic shrinkings ${\mf A}_0' \subset {\mf A}_0$ and ${\mf A}_1' \subset {\mf A}_1$, there exists a basic shrinking of $\tilde{\mf A}$ whose boundary restriction is the disjoint union of ${\mf A}_0'$ and ${\mf A}_1'$ and which is constant in the direction of the collar coordinate near the boundary.

\item Given Riemannian metrics near $\ov{{\mf A_0'}}$ and near $\ov{{\mf A_1'}}$, there exists a Riemannian metric near $\ov{ \tilde {\mf A}'}$ whose boundary restriction coincides with the given ones and which is constant in the direction of the collared coordinate near the boundary.
\end{enumerate}
\end{lemma}

\begin{proof}
For $a = 0, 1$, the basic shrinking of ${\mf A}_a$ is given by subsets
\beqn
G_{\Gamma, a}^l \sqsubset G_{\Gamma, a}^{'l} \sqsubset H_{\Gamma, a}^{'l} \sqsubset H_{\Gamma, a}^l,\ l = 1, \ldots
\eeqn
and corresponding shrinkings of the charts $K_I$. Choose numbers
\beqn
g_0^l < g_0^{'l} < h_0^{'l} < h_{0}^l,\ l = 1, \ldots
\eeqn
and similarly $g_1^{'l}$, $h_1^{'l}$. One can similarly choose a sequence
\beqn
G_{\Gamma, (0, 1)}^l \sqsubset G_{\Gamma, (0,1)}^{'l} \sqsubset H_{\Gamma, (0,1)}^{'l} \sqsubset H_{\Gamma, (0,1)}^l,\ l=1, \ldots.
\eeqn
Then for $\tilde I$, one obtains precompact open subsets
\beqn
\tilde O_{\tilde I}' \sqsubset \tilde O_{\tilde I}
\eeqn
which still cover $B_0\times [0,1]$. One then shrinking the Kuranishi charts in similar ways which produce a basic shrinking $\tilde {\mf A}' \subset \tilde {\mf A}$ satisfying the requirement. On the other hand, item (2) follows from Lemma \ref{lemma413}.
\end{proof}

Then one can find a shrinking $\tilde {\mf A}_+'$ which contains $\ov{\tilde{\mf A}'}$, where the Riemannian metric is defined, and which is constant in the collar coordinate direction. Then one obtains a $\uds{\bf Kur}_{\rm NCS}^{G, \flat}$-atlas on $\tilde {\mf M}$, a shrinking $\tilde {\mf A}'$, a strong map $\tilde {\mf f}$ whose boundary restriction coincides with the existing ones on ${\mf M}_0$ and ${\mf M}_1$. By Proposition \ref{prop_VFC_boundary}, it follows that 
\beqn
[{\mf A}_0, {\mf A}_0']_{\rm main}^{\rm vir} = [{\mf A}_1, {\mf A}_1' ]_{\rm main}^{\rm vir} \in H_*( \ov{\mc M}{}_{g,n}\times X^n; {\mb Q}).
\eeqn

\subsubsection{Comparing different line bundle data}

To finish the proof of Theorem \ref{thm52}, it remains to prove the independence from the choices of the line bundle data. In this case, we fix the moduli space ${\mf M} = \ov{\mc M}_{g,n}(X, J, A)$. Let ${\mb L}_i = (L_i, k_i, {\mf D}_i)$, $i = 1, 2$ be two line bundle data. Each of them provides an AMS parametrization of ${\mf M}$ which can be used to construct the Kuranishi atlas as we have done. The comparison between the two choices is essentially provided in \cite{AMS2} by using the ``double-framed'' AMS construction. 

\noindent ---{\bf General double-framed construction}--- We first describe a general double-framed construction. let $d$ now denote a pair of integers $(d_1, d_2)$ with 
\beqn
d_i = k_i (\Omega_i (A) + 2g - 2 + n).
\eeqn
Consider the moduli space 
\beqn
\ov{\mc M}_{g, n}(\mb{CP}^{d_1-g} \times \mb{CP}^{d_2-g}, (d_1, d_2))
\eeqn
of $n$-marked genus $g$ stable maps of bi-degree $d$. This space has an action by 
\beqn
G^{\mb C} = G_1^{\mb C} \times G_2^{\mb C} = U(d_1-g+1) \times U(d_2-g+1).
\eeqn
There is a $G^{\mb C}$-invariant open subset $B_{g, n, d} \subset \ov{\mc M}{}_{g,n}(\mb{CP}^{d_1-g+1} \times \mb{CP}^{d_2-g+1}, d)$ of curves $\phi = (\phi_1, \phi_2)$ such that $H^1(\phi_1^* {\mc O}(1)) = H^1(\phi_2^*{\mc O}(1)) = 0$ and that $\phi$ has trivial automorphism group. Then $B_{g, n, d}$ is a complex $G^{\mb C}$-manifold, on which one has a universal curve $C_{g, n, d} \to B_{g, n, d}$. The space $B_{g, n, d}$ is still stratified by isomorphism classes of map types where in the definition of map types (cf. Definition \ref{defn_map_type}), the degree function is now a function
\beqn
{\rm degree} = ({\rm degree}_1, {\rm degree}_2): {\rm Vert}(\tilde\Gamma) \to {\mb Z}_{\geq 0} \times {\mb Z}_{\geq 0}.
\eeqn

We take the following open subset $B \subset B_{g, n, d}$ by eliminating certain ``impossible'' strata. A map type $\tilde \Gamma$ is called {\bf $A$-possible} for $A \in H_2(X; {\mb Z})$ if there exist classes $A_v \in H_2(X; {\mb Z})$ such that 
\begin{align*}
&\ \sum_{v\in {\rm Vert}(\tilde \Gamma)} A_v = A,\ &\ {\rm degree}_i(v) = k_i(\Omega_i(A_v) + 2{\rm genus}(v) - 2 + n(v)),\ i = 1, 2.
\end{align*}
Here $n_v$ is the number of edges connected to $v$). Then define 
\beqn
B \subset B_{g, n, d}
\eeqn
to be the union of strata corresponding to those $A$-possible map types. It is easy to see that $B$ is $G^{\mb C}$-invariant and open. Abbreviate the restriction of $C_{g, n, d}\to B_{g, n, d}$ to $B$ by $C \to B$. 

Then one can follow the general procedure (Step \ref{step2}---\ref{step5}) of constructing a $\uds{\bf Kur}_{\rm NCS}^{G, \flat}$-atlas on ${\mf M}$ starting from the family $C \to B$ and a basic shrinking. Previous discussions on independence from choices made in Step \ref{step2}---\ref{step5} still applies here. Hence the double-framed construction gives a well-defined virtual fundamental class. Hence it remains to the virtual fundamental class obtained from the double-framed  construction with the one obtained from a single-framed construction. 

First, within $B_{g, n, d_i}$, one can also define the open subset of $A$-possible types, denoted by $B_i \subset B_{g, n, d_i}$. Then there are natural projections
\beqn
\pi_{B \to B_i}: B \to B_i.
\eeqn
Notice that an $A$-possible map type $\Gamma_i$ for $B_i$ uniquely determines an $A$-possible map type $\Gamma$ for $B$. Hence $\pi_{B \to B_i}$ induces a one-to-one correspondence between the sets of strata. 

Then one can compare the corresponding AMS parametrizations. 

\begin{lemma}\label{lemma536}
The natural projection
\beqn
\pi_1: {\mc M}_{{\mb L}}^{\rm fr}( C/B, X, A) \to {\mc M}_{{\mb L}_1}^{\rm fr}(C_1/B_1, X, A)
\eeqn
is a $G_1$-equivariant principal $G_2$-bundle. 
\end{lemma}

\begin{proof}
For each $(\phi_1, u_1, F_1) \in {\mc M}_{{\mb L}_1}^{\rm fr}(C_1/ B_1, X, A)$, consider the ample line bundle over $C_{\phi_1}$ induced from $u_1$ and the symplectic form $\Omega_2$, temporarily denoted by $L_2 \to C_{\phi_1}$. The set of unitary frames form a $G_2$-orbit and each frame $F_2$ induces a map $\phi_2: C_{\phi_1} \to \mb{CP}^{d_2-g+1}$ and $\phi_2 \in B_2$. This induces a point $\phi = (\phi_1, \phi_2) \in B$ and an identification $C_\phi \cong C_{\phi_1}$. 
\end{proof}


\noindent ---{\bf Pullback construction}--- Now we will see choices made in Step \ref{step2}---Step \ref{step5} can all be pulled back to do the double-framed construction. First, for any local thickening datum 
\beqn
\mu_{\Gamma_1} = (O_{\Gamma_1}, W_{\Gamma_1}, \lambda_{\Gamma_1})
\eeqn
over a $G_1$-invariant open subset $O_{\Gamma_1} \subset B_1$ can be pulled back to a thickening datum
\beqn
\pi_{B \to B_1}^* \mu_{\Gamma_1} = (\pi_{B \to B_1}^{-1}(O_{\Gamma_1}), \pi_{B \to B_1}^* W_{\Gamma_1}, \pi_{B \to B_1}^* \lambda_{\Gamma_1}).
\eeqn
Notice that the transversality persists. Moreover, denote
\beqn
O_{\Gamma} = (\pi_{B \to B_1})^{-1}(O_{\Gamma_1}).
\eeqn
The pullback thickening datum and this open subset produces a ``basic chart''
\beqn
K_\Gamma = (U_\Gamma, E_\Gamma, S_\Gamma, \Psi_\Gamma).
\eeqn
Moreover, define
\begin{align*}
&\ O_{\Gamma}':=  (\pi_{B \to B_1})^{-1}(O_{\Gamma_1}'),\ &\ O_{\Gamma}'':= (\pi_{B \to B_1})^{-1}(O_{\Gamma_1}'').
\end{align*}
and 
\beqn
\rho_\Gamma:= \rho_{\Gamma_1} \circ \pi_{B \to B_1}: B \to [0, 1]
\eeqn
which is supported in $O_\Gamma'$ and identically $1$ on $O_\Gamma''$. Further, 
define
\begin{align*}
    &\ G_\Gamma^l = (\pi_{B \to B_1})^{-1}(G_{\Gamma_1}^l),\ &\ H_\Gamma^l = (\pi_{B \to B_1})^{-1}( H_{\Gamma_1}^l).
\end{align*}
This completes Step \ref{step3}. At this step, it allows us to construct a $\uds{\bf Kur}^{G, \flat}$-atlas, denoted by ${\mf A}$, which we call the ``pullback'' of ${\mf A}_1$. Notice that the charts of ${\mf A}$ and ${\mf A}_1$ are labelled by the same poset ${\mc I}$. Let them be
\begin{align*}
&\ K_I = (U_I, E_I, S_I, \Psi_I),\ &\ K_{1, I} = (U_{1, I}, E_{1, I}, S_{1, I}, \Psi_{1, I}).
\end{align*}
Notice that since the thickening data is simply pulled back, similar to Lemma \ref{lemma536}, there is a natural projection
\beqn
\pi_{U_I \to U_{1, I}}: U_I \to U_{1, I},\ (\phi, u, e, F)\mapsto (\phi_1, u, e, F_1)
\eeqn
which is a $G_1$-equivariant principal $G_2^{\mb C}$-bundle. The fibres over a point $(\phi_1, u, e, F_1)$ are just all basis of $H^0({\mb L}_{2, u})$. These principal bundles can also transform naturally under coordinate changes. As $G_2^{\mb C} = GL(d_2-g+1)$, the principal bundle is an open subset of a $G_1$-equivariant vector bundle, denoted by $\hat U_I$, which consists of just a $(d_2-g+1)$-tuple of holomorphic sections of ${\mb L}_{2, u}$. Notice that the Hermitian pairing induces a well-defined Hermitian metric on the bundle $\hat U_I$. 

The basic shrinking ${\mf A}_1'$ can also be pulled back to the atlas ${\mf A}$. Let $U_{1, I}' \subset U_{1, I}$ be the precompact open subset in the shrinking ${\mf A}_1'$. Then the preimage
\beqn
U_I':= ( \pi_{U_I \to U_{1, I}})^{-1}(U_{1, I}')
\eeqn
gives an open subcategory ${\mf A}'$ (not a shrinking as $U_I'$ is not precompact). 

To proceed, we construct a system of Hermitian connections on $\hat U_I \to U_I$ which induces Riemannian metrics on $U_I \subset \hat U_I$ (near $\ov{U_I'}$). We omit the details. Notice that the metric also induces an NCS structure.

\noindent ---{\bf A small modification}--- The atlas ${\mf A}$ and the open subcategory constructed so far does not belong to the class of atlas we need to define the virtual fundamental class, because the sets $O_\Gamma'$ and $O_\Gamma''$ are not precompact in $O_\Gamma$. To proceed, we do a simple cut-off.

Choose $B_{1, 0} \subset B_1$ to be a $G_1$-invariant precompact open subset such that 
\beqn
{\mc M}_{{\mb L}_1}^{\rm fr}(C_1/B_{1,0}, X, A) = {\mc M}_{{\mb L}_1}^{\rm fr}(C_1/B_1, X, A).
\eeqn
Denote
\beqn
B_0:= (\pi_{B \to B_1})^{-1} (B_{1,0}) \subset B.
\eeqn
By Lemma \ref{lemma536}, we can choose $G$-invariant open subsets 
\beqn
B_0'' \subset \ov{B_0''} \subset B_0' \subset \ov{B_0'} \subset B_0 
\eeqn
such that 1) $B_0''$ and $B_0'$ are both precompact in $B$, 2)
\beqn
{\mc M}_{{\mb L}}^{\rm fr}(C/B_0'', X, A) = {\mc M}_{\mb L}^{\rm fr}(C/B, X, A)
\eeqn
and 3)
\beqn
\pi_{B \to B_1} (B_0'') = \pi_{B \to B_1}(B_0') = \pi_{B \to B_1} (B_0) = B_{1,0} \subset B.
\eeqn
Choose $G$-invariant cut-off functions
\beqn
\rho_{B_0}: B \to [0, 1]
\eeqn
supported in $B_0'$ which is identically 1 on $B_0''$. Define
\begin{align*}
&\ \mathring O_\Gamma':= O_\Gamma' \cap B_0',\ &\ \mathring O_\Gamma'':= O_\Gamma'' \cap B_0''.
\end{align*}
Then $\mathring O_\Gamma'$ and $\mathring O_\Gamma''$ are precompact in $O_\Gamma$. The using the thickening datum multiplied by cut-off function
\beqn
\mathring \rho_\Gamma:= \rho_\Gamma \rho_{B_0}
\eeqn
(which is supported in $\mathring O_\Gamma'$), one obtains a new chart
\beqn
\mathring K_\Gamma = (\mathring U_\Gamma, \mathring E_\Gamma, \mathring S_\Gamma, \mathring \Psi_\Gamma).
\eeqn
Moreover, choose a sequence 
\beqn
\mathring O_\Gamma' \sqsubset \mathring G_\Gamma' \sqsubset \mathring H_\Gamma' \sqsubset \cdots \sqsubset \mathring G_\Gamma^m \sqsubset \mathring H_\Gamma^m \sqsubset \cdots \mathring O_\Gamma
\eeqn
of $G$-invariant open subsets as in \eqref{nested_sets}. Following the same construction, one obtains a new $\uds{\bf Kur}_{\rm NCS}^{G, \flat}$-atlas $\mathring {\mf A}$ which belongs to the class of double-framed constructions. One can also choose a basic shrinking of the above sequence (see Definition \ref{defn_basic_shrinking}) to obtain a basic shrinking $\mathring {\mf A}'$ of $\mathring {\mf A}$.

\begin{lemma}
One has 
\beqn
[\mathring {\mf A}, \mathring {\mf A}']_{\rm main}^{\rm vir} = [{\mf A}_1, {\mf A}_1']_{\rm main}^{\rm vir} \in H_*( \ov{\mc M}_{g,n} \times X^n; {\mb Q}).
\eeqn
\end{lemma}

\begin{proof}
First notice that ${\mf A}$ has an open subcategory which coincides with an open subcategory of $\mathring {\mf A}$. Outside this common open subcategory, the component of the section $S_{I, Q_2}$ is non-vanishing. Hence a sufficiently $C^0$-close perturbation does not create zeroes outside this open subcategory. Therefore, we only need to consider perturbations on the auxiliary atlas ${\mf A}$. 

Now we choose perturbations on each $U_I$. Notice that one has the decomposition
\beqn
E_I = \pi_{U_I \to U_{1, I}}^* E_{1, I} \oplus Q_2
\eeqn
where the section $S_{I, Q_2}: U_I \to Q_2$ is geometrically defined. We do not know if it is smooth at this moment. However it is a $C^0$ structural group reduction of $U_I \to U_{1, I}$ from $G_2^{\mb C}$ to $G_2$. Inductively, one can perturb them to $C^\infty$ structural group reductions to 
\beqn
S_{I, Q_2}': U_I \to Q_2.
\eeqn
After this perturbation and appropriate shrinkings, we obtained a cobordant atlas. Hence we can assume that $S_{I, Q_2}$ is indeed smooth (we do not need its geometric meaning anymore). 

Then one can see that the two atlases ${\mf A}$ and ${\mf A}_1$ are related by a free quotient (by $G_2$) and a stabilization by $Q_2$. Therefore, the conclusion follows from Proposition \ref{prop_free_quotient} and Proposition \ref{prop_stabilization}.
\end{proof}

This finishes the proof of the second item of Theorem \ref{thm52} and hence the first part of the main theorem (Theorem \ref{thm11}) of this paper. 

\subsection{Genus zero case}

We  verify that genus zero reduced invariants coincide with the ordinary Gromov--Witten invariants. This is essentially an index calculation. 

Let $\alpha$ be an isomorphism class of refined map types and let $\Gamma$ be the underlying isomorphism class of map types. Then within the subset ${\mf M}_\Gamma^* \subset {\mf M} =  \ov{\mc M}_{g,n}(X, J, A)$, one can consider the deformation theory of curves in ${\mf M}_\Gamma^*$ with respect to the refined map type $\alpha$. More precisely, if $K = (U, E, S, \Psi)$ is a $\uds{\bf Kur}_{\rm NCS}^{G, \flat}$-chart on ${\mf M}$, then for each $x \in U_\alpha^*$, the fibre of the obstruction bundle $E_x$ contains a subspace $E_{x, \alpha} \subset E_x$. The difference
\beqn
{\rm dim}_\alpha^{\rm vir} {\mf A}:= {\rm dim} U_{x, \alpha}^* - {\rm dim} E_{x, \alpha}
\eeqn
only depends on $\alpha$. Geometrically, this is the expected dimension of the space of stable maps with prescribed ghost components. 

\begin{lemma}\label{lemma537}
For each refined map type $\alpha$ other than the main type, we have
\beqn
{\rm dim}_\alpha^{\rm vir} {\mf A} \leq {\rm dim}^{\rm vir} \ov{\mc M}{}_{0,n}(X, J, A) - 2.
\eeqn
\end{lemma}

\begin{proof}
It suffices to consider the case each effective subtree or ghost subtree has only one vertices, and all perturbations on ghost subtrees are turned off. We prove the theorem by induction on the number of ghost vertices. When there is only one ghost vertex, denote it by $v_0$. Suppose there are $m\geq 1$ effective vertices, denoted by $v_1, \ldots, v_m$, attached to $v_0$ and $l\leq n$ marked points attached to $v_0$. Suppose $v_i$ is labeled by effective classes $A_i$ and attached by $k_i$ marked points. Then the expected dimension of such a configuration is 
\beqn
\begin{split}
&\ \sum_{i=1}^m \Big( {\rm dim} X + 2c_1(A_i) + 2k_i - 4 \Big) + \Big( 2m + 2l - 6 \Big) + (m-1) {\rm dim} X\\
= &\ {\rm dim} X + 2c_1(A) + 2n - 2m - 6\\
= &\ {\rm dim}^{\rm vir} \ov{\mc M}_{0, n}(X, J, A) - 2m.
\end{split}
\eeqn
When $\alpha$ has more than one ghost vertices, choose one of them and apply the induction hypothesis to each component of its complement graph. 
\end{proof}

The ordinary GW invariants is defined via the total VFC of the constructed Kuranishi atlas (and the basic shrinking). Then using Proposition \ref{prop424} and Lemma \ref{lemma537}, one has
\beqn
[\ov{\mc M}{}_{0,n}(X, J, A)]^{\rm red} = [\ov{\mc M}{}_{0, n}(X, J, A)]^{\rm vir}.
\eeqn

\bibliographystyle{amsalpha}

\bibliography{mathref}

\end{document}